\documentclass[a4paper,11pt]{scrartcl}

\usepackage{tikz-cd} 
\usepackage{wrapfig}
\usepackage{caption}
\usepackage{graphicx}

\usepackage{amsthm}

\usepackage{amsmath,amssymb}

\usepackage{dsfont}

\usepackage{mathtools}

\usepackage{mathrsfs}

\usepackage[new]{old-arrows}

\usepackage{longtable}

 \usepackage{comment}

 \usepackage {hyperref}
 \hypersetup{
    colorlinks=true,
    linkcolor=blue,
    filecolor=magenta,      
    urlcolor=cyan,
    citecolor=magenta,
    pdftitle={Word Measures on Wreath Products II},
    pdfpagemode=FullScreen,
    }

\usepackage[
backend=biber,
style=alphabetic,
sorting=nyt,
backref=true,
giveninits=true,
maxnames=99,
doi=true,
url=false,
]{biblatex}
\renewbibmacro{in:}{}
\DeclareFieldFormat[article]{volume}{ Vol. #1 }
\DeclareFieldFormat[article]{number}{ No. #1 }

\DefineBibliographyStrings{english}{%
  backrefpage = {p.},
  backrefpages = {pp.},
}

\usepackage{microtype}
\usepackage{cleveref}

\usepackage{stmaryrd} 
\usepackage{extarrows} 

\usepackage{placeins} 

\usepackage{etoolbox} 

\usepackage[margin=1in]{geometry}

\newcommand{\divides}{\mid} 

\newcommand{\enquote}[1]{``{#1}''}

\robustify{\underset}

\DeclareRobustCommand{\bbone}{\text{\usefont{U}{bbold}{m}{n}1}}
\DeclareMathOperator{\EX}{\mathbb{E}} 
\DeclareMathOperator{\PR}{\mathbb{P}} 
\newcommand{\defeq}{\overset{\mathrm{def}}{=\joinrel=}}

\newcommand\ontop[2]{\genfrac{}{}{0pt}{}{#1}{#2}}

\DeclarePairedDelimiter{\floor}{\lfloor}{\rfloor}
\DeclarePairedDelimiterX{\inner}[1]{\langle}{\rangle}{#1}
\DeclarePairedDelimiterX{\brackets}[1]{[}{]}{#1}
\DeclarePairedDelimiterX{\braces}[1]{\{}{\}}{#1}
\DeclarePairedDelimiterX{\prn}[1]{(}{)}{#1}
\DeclarePairedDelimiterX{\abs}[1]{|}{|}{#1}
\def \tabularcenter#1 {\begin{tabular}{c}{#1}\end{tabular}}
\def\acts{\curvearrowright} 
\def\convolve{\circledast} 

\def\leff{\overset{*}{\le}}
\def\multiset#1#2{\ensuremath{\prn*{\kern-.3em\prn*{\genfrac{}{}{0pt}{}{#1}{#2}}\kern-.3em}}}
\newcommand{\alg}{\textup{alg}}
\newcommand{\Decompalg}{\mathcal{D}\textup{ecomp}_{\alg}}
\newcommand{\DecompB}{\mathcal{D}\textup{ecomp}_{B}}

\newcommand{\conj}{\textup{conj}}
\newcommand{\Hom}{\textup{Hom}}

\newcommand{\moccFr}{\mathcal{MOCC}\prn*{F_r}}

\newcommand{\mucg}{\mathcal{M}u\mathcal{CG}}
\newcommand{\mucgBFr}{\mathcal{M}u\mathcal{CG}_B\prn*{F_r}}

\newcommand{\Bconv}{\underset{B}{*}}
\newcommand{\convalg}{\underset{\alg}{*}}
\newcommand{\mualg}{\mu^{\alg}}
\newcommand{\chars}{\textup{char}}

\newcommand{\sgn}{\textup{sgn}}
\newcommand{\std}{\textup{std}}
\newcommand{\Ind}{\textup{Ind}}
\newcommand{\VecLambda}{\Vec{\lambda}}
\newcommand{\IndVecLambda}{\Ind\prn*{\VecLambda}}
\newcommand{\s}{\mathfrak{s}}
\newcommand{\FG}{\C}
\newcommand{\PnGhat}{\mathscr{P}_n\prn*{\hat{G}}}
\newcommand{\PdGhat}{\mathscr{P}_d\prn*{\hat{G}}}
\newcommand{\PledGhat}{\mathscr{P}_{\le d}\prn*{\hat{G}}}
\newcommand{\PleDGhat}{\mathscr{P}_{\le D}\prn*{\hat{G}}}
\newcommand{\PGhat}{\mathscr{P}\prn*{\hat{G}}}
\newcommand{\FIG}{\textbf{FI}_G}
\newcommand{\AGled}{\mathcal{A}(G)^{\le d}}
\newcommand{\AtagGled}{\mathcal{A}'(G)^{\le d}}
\newcommand{\Indphi}{\Ind\phi}
\newcommand{\IndnVecLambda}{\Ind_n\prn*{\VecLambda}}
\newcommand{\Indnphi}{\Ind_n\phi}
\newcommand{\FGconjG}{\FG^{\conj(G)}}
\newcommand{\pilab}{\pi_1^{\textup{lab}}}
\newcommand{\Cphipi}{\mathscr{C}_{\phi}^{\pi}}

\newcommand{\bigast}{%
  \mathop{\vcenter{\hbox{\scalebox{1.6}{$\ast$}}}}\limits
}

\newcommand{\C}{\mathbb{C}}

\newcommand{\F}{\mathbb{F}}

\newcommand{\N}{\mathbb{N}}
\newcommand{\Q}{\mathbb{Q}}
\newcommand{\R}{\mathbb{R}}
\newcommand{\Ss}{\mathbb{S}}

\newcommand{\Z}{\mathbb{Z}}

\newtheorem{theorem}{Theorem}[section]
\newtheorem{corollary}[theorem]{Corollary}
\newtheorem{fact}[theorem]{Fact}
\newtheorem{lemma}[theorem]{Lemma}

\newtheorem{proposition}[theorem]{Proposition}
\newtheorem{definition}[theorem]{Definition}
\newtheorem{example}[theorem]{Example}
\newtheorem{remark}[theorem]{Remark}
\newtheorem{conjecture}[theorem]{Conjecture}
\newtheorem{observation}[theorem]{Observation}
\numberwithin{equation}{section}

\title{Word Measures on Wreath Products II}
\author{Yotam Shomroni }
\date{August 2023}

\begin{document}

\maketitle

\begin{abstract}
    Every word $w$ in $F_r$, the free group of rank $r$, induces a probability measure (the $w$-measure) on every finite group $G$, by substitution of random $G$-elements in the letters. 
    This measure is determined by its Fourier coefficients: the $w$-expectations $\EX_w[\chi]$ of the irreducible characters of $G$.
    For every finite group $G$, every stable character $\chi$ of $G\wr S_n$ (trace of a finitely generated $\FIG$-module), and every word $w\in F_r$, we approximate $\EX_w[\chi]$ up to an error term of $O\prn*{n^{-\pi(w)}}$, where $\pi(w)$ is the primitivity rank of $w$. 
    This generalizes previous works by Puder, Hanany, Magee and the author.
    As an application we show that random Schreier graphs of representation-stable actions of $G\wr S_n$ are close-to-optimal expanders.
    The paper reveals a surprising relation between stable representation theory of wreath products and not-necessarily connected Stallings core graphs.
\end{abstract}


\tableofcontents 

\section{Introduction}
\label{section_intro}

In this paper we study word measures 
on wreath products of finite groups 
with the symmetric group $S_n$,
continuing the work of \cite{Sho23I}.
Explicitly, we bound the 
$w$-expectations of stable irreducible 
characters. Then we give applications 
to expansion of random Schreier graphs. 

We start by explaining the notions of wreath product, word measures and stable characters.
The \textbf{wreath product} of a group $G$ with the symmetric group $S_n$ is $G\wr S_n\defeq G\wr_{[n]} S_n \defeq G^n\rtimes S_n$, where $S_n$ acts on $G^n$ by permuting the indices $[n]\defeq \{1, \ldots, n\}$. 
The elements are $\braces*{(v, \sigma): v\in G^n, \, \sigma\in S_n}$, with coordinates $v = (v(1), \ldots, v(n))$, and the product is $\prn*{v_1, \sigma_1} \cdot \prn*{v_2, \sigma_2} = \prn*{v_1 \cdot (\sigma_1.v_2), \sigma_1\cdot \sigma_2}$. 
An element can also be thought of as a monomial matrix (every row and column has a unique non-zero entry) with $G$-elements, e.g.\
\begin{equation*}
    \begin{pmatrix}
    0 & i & 0 \\
    0 & 0 & 1 \\
    -i & 0 & 0
    \end{pmatrix} \in \{\pm 1, \pm i\} \wr S_3, \quad
    \begin{pmatrix}
    g_1 & 0 & 0 \\
    0 & 0 & g_2 \\
    0 & g_3 & 0
    \end{pmatrix} \in G\wr S_3 \quad (g_i\in G). 
\end{equation*}

Throughout the paper, $G$ denotes a finite group, $\conj(G)$ is its set of conjugacy classes, and $\hat{G}$ is its set of irreducible characters.
We also denote the set of all characters by $\textup{char}(G)$, and the trivial character (which is constant $1$) by $\textbf{1}\colon G\to \{1\}$.

\subsection*{What are word measures?}

Given a word $w\in F_r$ in a free group $F_r = \textup{Free}\prn*{\{b_1, \ldots, b_r\}}$, and a finite group $G$, we get a map (not necessarily a homomorphism)\footnote{Even though not being homomorphisms themselves, word maps commute with homomorphisms. In the language of categories, word maps are precisely the natural transformations $\textbf{forg}^r\to \textbf{forg}$, where $\textbf{forg}\colon \textbf{Grp}\to \textbf{Set}$ is the forgetful functor and $\textbf{forg}^r(G) = G^r$ gives the $r$-fold Cartesian product, and similarly $\textbf{forg}^r(f) = (f, \ldots, f)$ for every homomorphism $f$. } 
$w\colon G^r\to G$, called a word map. 
For example, $b_1 b_2^2 b_1^{-1} b_2^{-1}$ maps $(g, h)\in G^2 \mapsto gh^2 g^{-1}h^{-1}\in G$.
We denote by $U, U^{\times r}$ the uniform measures on $G, G^r$ respectively. 
The pushforward measure $\mu_w \defeq w_*(U^{\times r})$ on $G$ is called the $w$-measure on $G$.
Equivalently, by the universal property of the free group, homomorphisms $F_r\to G$ correspond to functions $\{b_1, \ldots, b_r\}\to G$, which can be identified with $G^r$, so $\Hom(F_r, G)$ is finite and for a random homomorphism $\alpha\sim U(\Hom(F_r, G))$, the image $\alpha(w)$ is distributed according to the $w$-measure. 
Alternatively, 
\[\forall g\in G:\quad  \mu_w(g) = \frac{1}{|G|^r} \abs*{\braces*{(g_1, \ldots, g_r)\in G^r: w(g_1, \ldots, g_r) = g}}. \]
For example, the word $b_1\in F_r$ induces the uniform probability measure on any finite group.

Given a function $f\colon G\to \C$, its expectation according to the $w$-measure is
\[ \EX_w[f] \defeq \EX_{\alpha}[f(\alpha(w))] = \sum_{g\in G} f(g) \mu_w(g). \]

This expectation captures interesting information about both $w$ and $f$.
As a toy example, consider $G = C_m$, the cyclic group of order $m$, and $f\colon C_m \hookrightarrow \Ss^1$ the standard embedding $f(x) = e ^ {\frac{2\pi i x}{m}}$.
Given a basis $B = \{b_1, \ldots, b_r\}$ of the free group $F_r$, define $\nu_i\in \Hom(F_r, \Z)$ for every $i\in [r]$ by letting $\nu_i(b_j)\defeq \bbone_{i=j}$ for every $j\in [r]$. 
(This determines $\nu_i$ uniquely).
Then 
\begin{equation}
\label{equation_expectation_of_cyclic_embedding}
\EX_w[f] =
    \begin{cases}
    1 & \textrm{ if } \nu_i(w)\equiv 0 \textup{ (mod }m\textup{)} \textrm{ for every }i,\\
    0 & \textrm{otherwise.}
    \end{cases}
\end{equation}
Indeed, by independence of the letters, and since the sum of a non-trivial subgroup of the $m^{th}$-roots of unity vanishes,
\[\EX_w[f] = \prod_{i=1}^r \EX_{b_i^{\nu_i(w)}}[f] = \prod_{i=1}^r \bbone_{\nu_i(w) = 0 \textup{ (mod }m\textup{)}}.\]

A much more interesting example arises when $G = S_n$ is the symmetric group, and $f = \#\textup{fix}$ is the natural character that counts fixed points of permutations.
By the well-known Nielsen-Schreier theorem, a subgroup of a free group is free. 
A primitive element in $H$ is a part of a basis of $H$, or equivalently, in the $\textup{Aut}(H)$-orbit of $h_1$ for some basis $\{h_1, \ldots, h_k\}$ of $H$.
In \cite[Theorem 1.8]{PP15}, Puder and Parzanchevski proved an approximation theorem for $\EX_w[f]$, relating it to the subgroup structure of $F_r$:
\begin{equation}
    \label{equation_PP15_thm_main_result}
    \EX_w[f] = 1 + |\textup{Crit}(w)|n^{1-\pi(w)} + O\prn*{n^{-\pi(w)}},
\end{equation}
where $\pi(w)$ and $\textup{Crit}(w)$ are the \textbf{primitivity rank} and the \textbf{critical subgroups} of $w$: 
\begin{definition}
\label{def_primitivity_rank}
\cite[Definition 1.7]{Puder_2014}
Let $w\in F_r$ be a word in a free group.
The \textbf{primitivity rank} of $w$ is
\[\pi(w) \defeq \min\braces*{\textup{rk}(H): w\in H\le F_r, \, w\textrm{ is a non-primitive element in }H},\]
with the convention $\min\emptyset = \infty$.
The subgroups achieving the minimum are called \textbf{critical}:
\[\textup{Crit}(w) \defeq \braces*{H\le F_r: w\in H\textrm{ is a non-primitive element in }H, \, \textup{rk}(H)=\pi(w)}.\]
\end{definition}

A word $w\neq 1$ is called a proper power if $w = u^k$ for some $u\in F_r, \,\,k\ge 2$ (equivalently, $\pi(w) = 1$), and otherwise it is called a non-power.
The following table gives examples for $\pi(w)$ and $\textup{Crit}(w)$:

\begin{table}[ht!]
\centering
\begin{tabular}{||c | c | c ||} 
 \hline
 Description of $w$ 
 & $\pi(w)$ 
 & $\textup{Crit}(w)$ \\[0.7ex] 
 \hline\hline
 $w=1$ & 0 & $\{\inner*{1}\}$ \\[0.7ex] 
 $w$ is a proper power & 1 & $\{\inner*{u}: \inner*{w} \lvertneqq \inner*{u}\}$  \\[0.7ex] 
 $[b_1, b_2]$ & 2 & $\{\inner*{b_1, b_2}\}$ \\[0.7ex] 
 $b_1^2 b_2^2$ & 2 & $\{\inner*{b_1, b_2}\}$ \\[0.7ex] 
 $\vdots$ & $\vdots$ & $\vdots$ \\[0.7ex] 
 $b_1^2 \ldots b_k^2$ & k & $\{\inner*{b_1, \ldots, b_k}\}$ \\[0.7ex] 
 $w$ is primitive & $\infty$ & $\emptyset$ \\ [1ex] 
 \hline
\end{tabular}
\caption{Primitivity rank and critical subgroups}
\label{table:primitivity_rank}
\end{table}
\FloatBarrier

\subsection*{What are stable characters?}

Informally, given a sequence $(G_n)_{n\in \N}$ of groups, a stable character is a sequence of characters $(\chi_n)_{n\in \N}$, where $\chi_n\in \chars(G_n)$, with a uniform definition that in some sense does not depend on $n$ for sufficiently large $n$.

To study stable characters on $G\wr S_{\bullet} \defeq \prn*{G\wr S_n}_{n\in \N}$, we first describe the case $G = \{1\}$.
The algebra of stable class functions on the symmetric groups $S_{\bullet}$ is $\Q[a_1, a_2, \ldots]$, where each $a_t$ is the family of class functions $\prn*{a_t \colon S_n\to \N}_{n\in \N}$ that count cycles of length $t$ in permutations. 

\begin{definition}
    \label{def_partition}
    An \textbf{integer partition} of $n$ is a sequence of natural numbers $\lambda_1 \ge \ldots \ge \lambda_{\ell} > 0$ with $|\lambda|\defeq \sum \lambda_i = n$. 
    We sometimes denote $\lambda\vdash n$ instead of $|\lambda| = n$.
    For every $d\in \N$ and a partition $\lambda\vdash d$, we define $\lambda[n] \vdash n$ for every $n\ge d + \lambda_1$ as $\lambda[n] \defeq \prn*{n - d, \lambda_1, \ldots, \lambda_{\ell}}$, where $\lambda_1 \defeq 0$ for the empty partition, so that $\emptyset[n] = (n)$.
    A \textbf{Young diagram} $\lambda$ with $n$ cells is an integer partition of $n$, visualized as a set of points in the plane (\enquote{cells}) arranged in $\ell$ rows of lengths $\lambda_1 \ge \ldots \ge \lambda_{\ell}$.
\end{definition}

There is a well-known correspondence between Young diagrams $\lambda$ with $n$ cells and irreducible representations of $S_n$.
For every partition $\lambda$, denote by $\chi_{\lambda}$ the corresponding irreducible character.
This gives us an alternative description of stable functions: a family of class functions $f = \prn*{f_n\colon S_n \to \Q}_{n\in \N}$ is stable of degree at most $d$ if there are finitely many Young diagrams $\lambda$ with at most $d$ cells and coefficients $c_{\lambda}\in \Q$ such that for every large enough $n$, we have $f_n = \sum_{\lambda} c_{\lambda} \cdot \chi_{\lambda[n]}$.
(See the appendix \enquote{The ring of class functions} in \cite{HP22}).
For example, for every $n\ge 2$ the partition $(1)[n] = (n-1, 1)$ corresponds to the standard (stable) character $\std=\#\textup{fix}-1$ of degree $1$,
and for every $n\ge 3$ the partition $(1, 1)[n] = (n-2, 1, 1)$ corresponds to the stable character $\wedge^2 \std$ of degree $2$.

Following \cite[Section 2.3.1]{ceccherini-silberstein_scarabotti_tolli_2014},
for every finite group $G$ 
we define $a_{t, c} \colon G\wr S_n\to \N$ 
for every $t\in \Z_{\ge 1}$ and every conjugacy class $c\in \conj(G)$ as follows: for every $\sigma\in S_n, v\in G^n$,
\begin{equation}
\label{eq_def_a_tc}
a_{t, c}(v, \sigma) \defeq \abs*{\braces*{t\textrm{-cycles $\prn*{i_1 \overset{\sigma}{\to} \ldots \overset{\sigma}{\to} i_t \overset{\sigma}{\to} i_1}$ of $\sigma$ satisfying } v(i_t)\cdots v(i_1) \in c}}.
\end{equation}
This is well defined, since changing $i_1$ in the $t$-cycle amounts to conjugating the product.

\begin{definition}
\label{def_AG}
We define the algebra of stable class functions on $G\wr S_{\bullet}$ as 
\[\mathcal{A}(G)\defeq \FG\brackets*{(a_{t,\, c})_{t\in \N,\,\, c\in \conj(G)}}. \]
We define the stable irreducible characters as
\[ \widehat{G\wr S_{\bullet}} \defeq \braces*{(\chi_n)_{n\in \N}\in \mathcal{A}(G): \,\,\,\,\,\forall n \textup{ large enough}, \,\,\,\chi_n\in \widehat{G\wr S_n}}. \]
\end{definition}

The stable irreducible characters form an orthonormal basis of the algebra $\mathcal{A}(G)$, and can be characterized using partitions similarly to the case of $S_n$
(see Section~\ref{section_stable_rep}).
The algebra $\mathcal{A}(G)$ is a filtered\footnote{Recall that a filtered algebra $\mathcal{A}$ is an algebra together with a degree map $\deg\colon \mathcal{A}\to \N$ satisfying $\deg(x + y) \le \max(\deg(x), \deg(y))$ and $\deg(x\cdot y) = \deg(x) + \deg(y)$ for every $x, y\in \mathcal{A}$.} algebra:

\begin{definition}
We define a degree function $\deg\colon \mathcal{A}(G)\to \Z_{\ge 0}$ as follows.
Every $f\in \mathcal{A}(G)$ can be written uniquely as a (finite) linear combination of finite products of $\braces*{a_{t, c}}_{t\in \N,\,\, c\in \conj(G)}$:
\[ f = \sum_{\mathcal{I}} \beta_{\mathcal{I}} \prod_{i\in \mathcal{I}} a_{t_i, c_i} ^ {r_i}, \quad \quad r_i, t_i\in \Z_{\ge 1}, \,\, c_i\in \conj(G), \,\, \mathcal{I}\subseteq \N \textup{ finite},\,\, \beta_{\mathcal{I}}\in \FG-\{0\}.\]
We define $\deg(f) \defeq \max_{\mathcal{I}:\,\, \beta_{\mathcal{I}}\neq 0} \sum_{i\in \mathcal{I}} t_i\cdot r_i.$
Note that
$\deg(f) = 0$ if and only if $f$ is constant (that is, only $\beta_{\emptyset}\neq 0$), 
that for every $t\in \N$ and $c\in \conj(G)$ we have $\deg(a_{t, c}) = t$, 
and for every $f, g\in \mathcal{A}(G)$ we have $\deg(f \cdot g) = \deg(f) + \deg(g)$ and $\deg(f + g) \le \max(\deg(f), \deg(g))$. 
(These are standard facts in polynomial rings).
\end{definition}

For background about stability, see Appendix~\ref{appendix_stable_algebra}.

\noindent Now we relate stable characters and word measures.
For a stable character $\chi = (\chi_n)_{n\in\N}$, we omit $n$ from the notation of the dimension $\dim(\chi) = \prn*{\dim(\chi_n)}_{n\in \N}$ and the $w$-expectation $\EX_w[\chi] = \prn*{\EX_w[\chi_n]}_{n\in \N}$.
The following conjecture is central in the investigation of $w$-measures:
\begin{conjecture}
\label{conj_great_HP22}
(\cite[Conjecture 1.13]{HP22}) 
For every stable irreducible character $\chi$ and $w\in F_r$,
\[ \EX_w[\chi] = O\prn*{\dim\prn*{\chi}^{1-\pi(w)}} \]
where $n\to\infty$ and the implied constant depends on $\chi, w$.
\end{conjecture}

This conjecture is known to be true for surface words - see \cite[Theorem 1.1]{magee2021surface}, 
and for proper powers.
For example, $w = [a, b]\in \textup{Free}\braces{a, b}$ has $\pi(w)=2$ and $\EX_w[\chi] = \chi(1)^{-1}$.
In \cite{Sho23I}, the case $\deg(\chi) = 1$ of this conjecture was proved for the group sequence $G\wr S_{\bullet}$, for every compact $G$.
In this paper we prove for every finite $G$ that if $\deg(\chi)\ge 2$ and $\pi(w)\ge 2$, 
\[ \EX_w[\chi] = O\prn*{n^{-\pi(w)}}. \]

\subsection*{Statement of the main result}

In the following table, we summarize the progress that was already made\footnote{The papers are ordered chronologically by the year they were produced, not published.} in the analysis of word measures on wreath products $G_n = G\wr S_n$.
Every paper in the table approximates  $\EX_w[\chi]$ for every word $w\in F_r$, up to an error term of order $O\prn*{n^{-\pi(w)}}$, where $\chi \in \chars\prn*{G_n}$.
\begin{table}[ht!]
\centering
\begin{tabular}{|| c | c | c | c | c||} 
\hline
Paper
& Group $G$
& Stable Character $\chi$ 
& $\dim(\chi) $\\ [0.5ex] 
\hline\hline
\cite{PP15}
& $\{1\}$
& Trace of  $G_n\hookrightarrow GL_n(\C)$
& $n$\\  [1ex] \hline

\cite{magee2021surface}
& $\Ss^1$ or finite cyclic
& Trace of  $G_n\hookrightarrow GL_n(\C)$
& $n$ \\  [1ex]\hline

\cite{HP22}
& $\{1\}$
& \textbf{Every} stable character
& Arbitrary polynomial\\  [1ex]\hline

\cite{Ord20}
& $\Z/2\Z$
& A specific char' $\chi_{(n-2);(2)}$
& $\binom{n}{2}$\\  [1ex]\hline

\cite{Sho23I}
& \textbf{Any} compact group
& $\chi:\, \deg(\chi)\le 1$
& $O(n)$\\  [1ex]\hline

This paper
& \textbf{Any} finite group
& \textbf{Every} stable character
& Arbitrary polynomial\\  [1ex]\hline

 \hline
\end{tabular}
\caption{Comparison to related papers about wreath products with $S_n$}
\label{table:wreath_product_papers}
\end{table}
\FloatBarrier

The notation $G_n\hookrightarrow GL_n(\C)$ indicates the natural embedding into $GL_n(\C)$ with the corresponding representation $\C^n$.
We already described the approximation of \cite{PP15} in Equation~\eqref{equation_PP15_thm_main_result}. 
Each of the other results is a special case of Conjecture~\ref{conj_great_HP22}:
\cite{magee2021surface} and \cite{Sho23I} proved\footnote{In fact, a stronger bound is proved, and the leading coefficient is also computed (at least in the \enquote{generic} case).} $\EX_w[\chi] = O\prn*{n^{1-\pi(w)}}$ whenever $\deg(\chi) = 1$,
while \cite{Ord20} and \cite{HP22} proved 
$\EX_w[\chi] = O\prn*{n^{-\pi(w)}}$ whenever $w$ is a non-power 
and either $\chi = \chi_{(n-2);(2)}$ 
(in the case of \cite{Ord20}) or 
$\deg(\chi)\ge 2$ (in the case of \cite{HP22}).
In the more recent \cite{puder2023stable} and 
\cite{puder2025stable}, 
the precise asymptotics of $\EX_w[\chi]$ is 
computed for many more stable characters $\chi$ of $G\wr S_n$
for compact groups $G$.

For every stable class function $f = (f_n)_{n\in\N} \in \mathcal{A}(G)$ and $k\in \N$, we denote $\prn*{f^{(k)}}_n(x) \defeq f_n\prn*{x^k}$.
We denote the length of a word $w\in F_r$ by $|w|$.
\begin{theorem}
\label{thm_main_result} (Main Result)
Let $w\in F_r$, $G$ a finite group, and $\forall n\in \N: G_n \defeq G\wr S_n$. 
Let $\chi\in \widehat{G_{\bullet}}$ be a stable irreducible character.
\begin{enumerate}
    \item For every $n\ge \deg(\chi)\cdot |w|$, the expectation $\EX_w[\chi]$ coincides with some rational function in $\FG(n)$.
    Moreover, if $\EX_w[\chi]\neq 0$ then 
    $\EX_w[\chi] = \Omega\prn*{n^{-\deg(\chi)\cdot |w|}} = \Omega\prn*{\chi(1)^{-|w|}}$.
    \item If $\deg(\chi) \ge 2$, and $w$ is a non-power, then $ \EX_w[\chi] = O\prn*{n^{-\pi(w)}}. $
    \item If $w$ is a power, say $w = u^k$ for a non-power $u\in F_r$ and $k\ge 2$, then $\pi(w)=1$ and
    \[ \EX_w[\chi] 
    = \EX_{x^k}[\chi] + O\prn*{n^{-1}} 
    = \inner*{\chi^{(k)}, \mathbf{1}} + O\prn*{n^{-1}}. \]
\end{enumerate}
\end{theorem}

For the third case, where $w = u^k$, we give in Corollary~\ref{corollary_approx_power} a detailed approximation, up to an error of $O\prn*{n^{-\pi(u)}}$.
Note that together with \cite[Theorem 1.9]{Sho23I} and the vacuous statement $\forall w: \EX_w[\textbf{1}] = 1$, the theorem always gives an approximation up to a $O\prn*{n^{-\pi(w)}}$-error, for every word and every stable irreducible character of $G\wr S_{\bullet}$.
Note also that this theorem (combined with \cite{Sho23I}) 
generalizes all the previous results from Table~\ref{table:wreath_product_papers}.
Since $\EX_w$ is a linear operator, this theorem gives an approximation up to $O\prn*{n^{-\pi(w)}}$ of $\EX_w[f]$ for every stable class function $f$: for example, recall $a_{t, c}$ from Equation~\eqref{eq_def_a_tc}.

\begin{corollary}
\label{corollary_a_tc_approx}
Let $w\in F_r$ be a non-power. Then for every $t\ge 2, c\in \conj(G)$:
\[ \EX_w\brackets*{a_{t, c}} = \frac{|c|}{t|G|} + O\prn*{n^{-\pi(w)}}. \]
\end{corollary}

We will show in Equation~\eqref{equation_proved_a_tc_approx} how this corollary follows from the theorem. \\

\subsection*{Application: expansion of random Schreier graphs}
We call an action that induces a stable permutation representation \textbf{representation-stable}, or just \textbf{rep-stable}.
In this paper we handle only finite groups and actions on finite sets.

\begin{wrapfigure}{}{0.25\textwidth}
\captionsetup{labelformat=empty}
    \centering
        \includegraphics[scale=0.05]{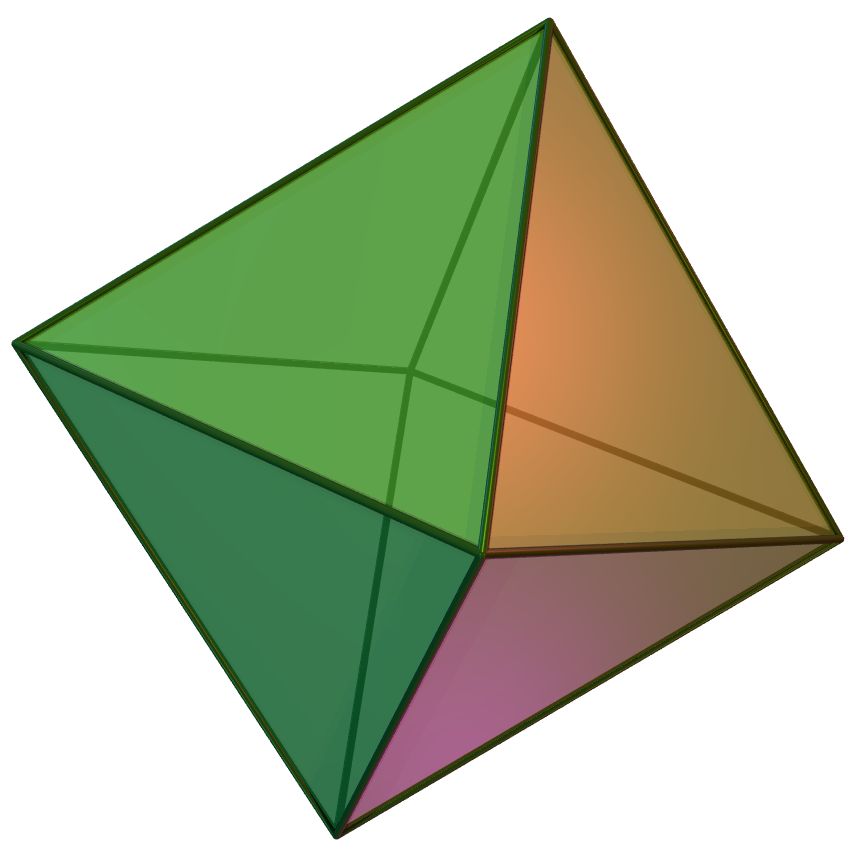}
    \includegraphics[scale=0.05]{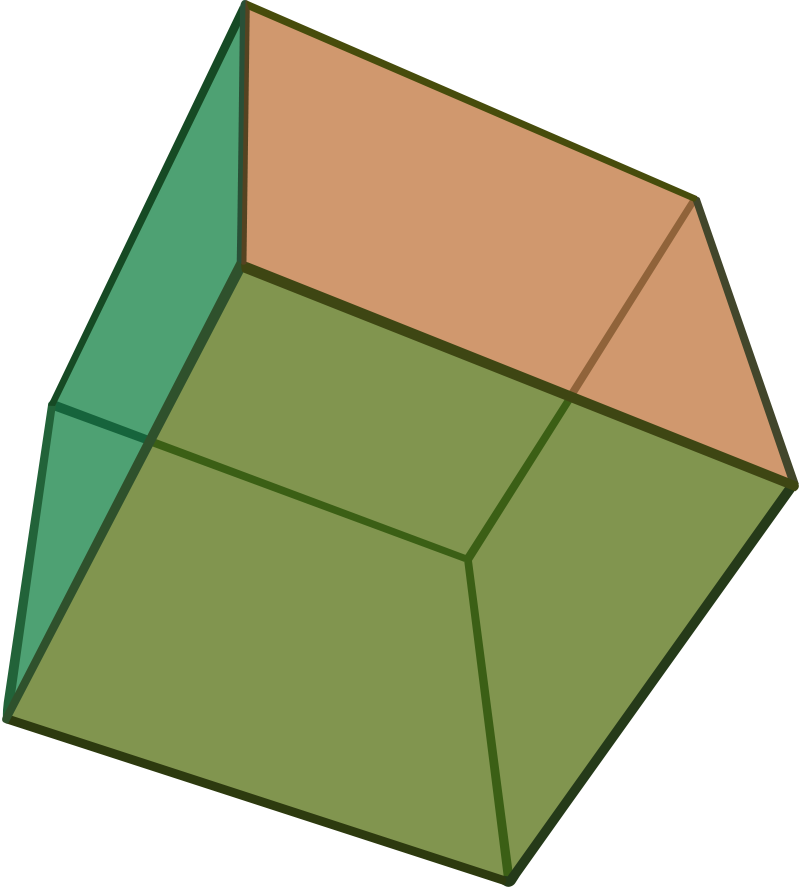}
    \caption{\footnotesize octahedron and cube: \\dual polytopes}
    \label{fig: cube_and_octahedron}
\end{wrapfigure}

For example, for every $k\in \N$, 
if $H$ is a simple non-abelian finite group 
and 
$G = \textup{Aut}(H)$, then 
$G\wr S_n = \textup{Aut}(H^n)$, 
and its action on the elements of 
$H^n$ with exactly $k$ 
nontrivial coordinates is rep-stable of 
degree $k$.
As another example, if 
$G = \Z/2\Z$, then $G\wr S_n$ is the 
hyperoctahedral group, i.e.\ the isometry group
of the $\ell^1$-ball (the cross-polytope) 
in $\R^n$, and
its action on $(k-1)$-dimensional faces is 
rep-stable of degree $k$. 
By duality, 
$\Z/2\Z \wr S_n$ is also the isometry 
group of the hypercube 
(the $\ell^{\infty}$-ball in $\R^n$), 
and its action on $(n-k)$-dimensional faces
is isomorphic to the previous action.
Together with $\{1\}\wr S_{n+1} = S_{n+1}$
acting on $k$-faces of the $n$-simplex,
these actions
cover all rep-stable symmetries of 
$n$-dimensional regular convex polytopes for 
$n\ge 5$.\footnote{
    Regular convex polytopes
    are classified in \cite[\S 7.8]{Cox73}.
    Stable actions of their isometry groups
    are classified in \cite{WILSON2014269}.
}

Every group action $G\acts X$ together with $r$ group elements $g_1, \ldots, g_r\in G$ gives rise to a $2r$-regular multi-graph, called a \textbf{Schreier graph}, with vertex set $X$ and edge multi-set $\braces*{\prn*{x, g_i(x)}:\,\,\, x\in X,\, i\in [r]}$.
We regard edges $\prn*{v, u}$ coming from different $i\in [r]$ as different. 
Note that a self-loop contributes $2$ to the degree of a vertex.
\begin{definition}
    Given a finite group $G$ and an action $\rho\colon G\to \textup{Sym}(X)$, we denote by $\textup{Sch}_r(\rho)$ the $2r$-regular random Schreier graph obtained from $\rho$ and $r$ random elements $g_1, \ldots, g_r\sim U(G)$.
\end{definition}
For example, for $n, r\in \N$, $\textup{Sch}_{r}\prn*{S_n\acts [n]}$ is just a random covering of degree $n$ of the bouquet with $r$ petals, which is closely related to the random $2r$-regular multi-graph on $n$ vertices.
Given a $d$-regular graph $\Gamma$ on $n$ vertices, denote the eigenvalues of its adjacency matrix by $d = \lambda_1 \ge \lambda_2 \ge \ldots \ge \lambda_n$, and denote $\mu_{\Gamma}\defeq \max(\lambda_2, -\lambda_n)$.
In \cite{NILLI1991207}, Alon proved that for every $d$-regular graph on $n$ vertices, $\lambda_2 \ge 2\sqrt{d-1}-o(1)_{n\to \infty}$. 
Hence a sequence of $d$-regular graphs satisfying $\lambda_2 \le 2\sqrt{d-1}$ is considered an optimal expander.
In \cite{Alon86}, Alon also conjectured that for every $\varepsilon > 0$, a random $d$-regular graph on $n$ vertices satisfies $\lambda_2 \le 2\sqrt{d-1} + \varepsilon$ with high probability as $n\to \infty$, namely, is a.a.s.\ (asymptotically almost surely) an almost-optimal expander. 
This conjecture was proved in \cite{Friedman04}.
A slightly weaker version of this conjecture was proven in \cite{Expansion_Puder_2015}, using word measures to analyze the model $\textup{Sch}_{d/2}\prn*{S_n\acts [n]}$.
In \cite{HP22}, Puder and Hanany generalized the result proved by Puder in \cite{Expansion_Puder_2015} to any rep-stable action of $S_n$, demonstrated with a specific action: for every $s \in \N$ and sufficiently large $r$, $\textup{Sch}_r\prn*{S_n\acts \braces*{\textup{injective functions }f\colon [s]\hookrightarrow [n]}}$ is a.a.s.\ a close-to-optimal expander.
The following theorem generalizes this result to wreath products $G\wr S_n$:

\begin{theorem}
\label{thm_expander_wreath}
Let $G$ be a finite group, $(X_n)_{n=1}^{\infty}$ a sequence of sets, and $\prn*{\rho_n\colon G\wr S_n\to \textup{Sym}(X_n)}_{n=1}^{\infty}$ a rep-stable sequence of transitive actions of degree $k$. Let $r\ge 2$ and $\Gamma_n\defeq \textup{Sch}_r\prn*{\rho_n}$.
Then
\[ \mu_{\Gamma_n} \le 2\sqrt{2r - 1}\cdot \exp\prn*{ \frac{2k^2}{e^2(2r-1)} } \quad\quad \textup{with probability } 1 - o(1) \textup{ as }n\to\infty. \]
In particular for every fixed $k, \varepsilon > 0$, if $r$ is large enough, $\mu_{\Gamma_n} \le 2\sqrt{2r - 1} + \varepsilon$ a.a.s.\ 
\end{theorem}

\subsection*{Related works}

To state the approximation of $\EX_w[f]$ for every stable class function $f$ of $G\wr S_{\bullet}$ in a unified formulation, we need to combine the result of \cite{Sho23I} with Theorem~\ref{thm_main_result}.

\subsubsection*{The main result of \cite{Sho23I}}

For a class function $f\colon G\to \FG$ and $w\in H\le F_r$ (here $H$ is a finitely generated subgroup), denote 
$\EX_{w\to H}[f]\defeq \EX_{\alpha\sim U(\Hom(H, G))}[f(\alpha(w))]$.
Recall that for every finite group $G$ we have the standard normalized inner product on $\FG^{G}$, that is $\inner*{\phi, \psi}\defeq \frac{1}{|G|}\sum_{g\in G} \phi(g) \overline{\psi}(g)$, and that the irreducible characters form an orthonormal basis for class functions.
For stable class functions $f, g\in \mathcal{A}(G)$ we write $\inner*{f, g}\defeq \inner*{f_n, g_n}_{G\wr S_n}$, which does not depend on $n$ once $n$ is large enough, by Corollary~\ref{corollary_stabilization}.
So in other words, if $\mu_{w\to H}\colon G\to [0, 1]$ is the $w$-measure on $G$ when we think of $w$ as an element of $H$, then 
$\EX_{w\to H}[f]\defeq |G|\cdot \inner*{f, \mu_{w\to H}} = \sum_{g\in G} f(g) \mu_{w\to H}(g)$.
Note that if $w$ is primitive in $H$ then $\EX_{w\to H}[f] = \inner*{f, \textbf{1}} = \frac{1}{|G|}\sum_{g\in G}f(g)$.

There are situations where free groups $H \le J$ are \enquote{equivalent} with respect to word measures:
A subgroup $H\le F$ of a free group is 
called a (proper) \textbf{free factor} if 
there exists a (non-trivial, respectively) 
$J\le F$ 
such that $H * J = F$.
In this case we denote $H\leff{F}$, 
and say that $F$ is a 
\textbf{free extension} of $H$.
It is easy to see that for every word $w\in H\leff J$ and for every finite group $G$, the word measures $\mu_{w\to H}, \mu_{w\to J}$ on $G$ are the same\footnote{Indeed, taking some basis of $J$ that contains a basis of $H$, every basis element is sent under a random homomorphism $J\to G$ to a random element, independent of the other basis elements. Thus the word measures $\mu_{w\to H}, \mu_{w\to J}$ are the same.}, so for every $\phi\in \textup{char}(G)$, $\EX_{w\to H}[\phi] = \EX_{w\to J}[\phi]$.

\begin{definition}[Algebraic Extension]
\label{def_algebraic_extension}
This concept goes back to \cite{Takahasi1951NoteOC}, and was further studied in \cite{KAPOVICH2002608}, \cite{MVW07}.
Given $H\le J\le F_r$, we say that $J$ is an \textbf{algebraic extension} of $H$ if $H$ is not contained in any proper free factor of $J$, and then denote $H\le_{\alg} J$. If $H\neq J$, $J$ is called a \textbf{proper} algebraic extension. For a word $w\in H$ we say that $H$ is an algebraic extension of $w$ if $\inner*{w} \le_{\alg} H$.
Every finitely generated $H$ has only a finite number of algebraic extensions (see for example \cite[Corollary 4.4]{Puder_2014}).
\end{definition}

\begin{definition}
\label{def_witnesses_full}
Let $G$ be a finite group and $F_r$ the free group of rank $r$.
For every word $1\neq w\in F_r$ and an irreducible character $\phi\in \hat{G}$, we define the $\phi$\textbf{-witnesses} of $w$ as the set of subgroups 
    \[ \textup{Wit}_{\phi}(w) \defeq \braces*{H\le F_r: \quad \inner*{w} \lneqq H \textup{ is a proper algebraic extension, and } \,\EX_{w\to H}[\phi] \neq 0}.\]
    We define the $\phi$-primitivity rank of $w$ as the minimal rank in $\textup{Wit}_{\phi}(w)$:
    \[ \pi_{\phi}(w) \defeq \min\{\textup{rk}(H): \quad H\in \textup{Wit}_{\phi}(w)\}, \]
    with the convention $\min\emptyset = \infty$.
    We also define the $\phi$-critical groups of $w$: 
    \[ \textup{Crit}_{\phi}(w) \defeq \{H\in \textup{Wit}_{\phi}(w): \quad \textup{rk}(H) = \pi_{\phi}(w)\},\]
    and the $\phi$-critical value of $w$: $\mathscr{C}_{\phi}(w) \defeq \sum_{H\in \textup{Crit}_{\phi}(w)} \EX_{w\to H}[\phi]. $
\end{definition}

Note that for every witness subgroup $H\in \textup{Wit}_{\phi}(w)$, $w$ is not primitive in $H$ (as $H$ is a proper algebraic extension of $\inner{w}$), so $\pi_{\phi}(w) \ge \pi(w)$.
The condition $\EX_{w\to H}[\phi]\neq 0$ is equivalent to the appearance of $\phi$ in the decomposition of $\mu_{w\to H}\colon G\to [0, 1]$ into a $\FG$-linear combination of irreducible characters.
By \cite[Proposition 2.3]{Sho23I},
if $\phi = \textbf{1}$ then $\pi_{\phi}(w) = \pi(w)$, $\textup{Crit}_{\phi}(w) = \textup{Crit}(w)$ and $\mathscr{C}_{\textbf{1}}(w) = |\textup{Crit}(w)|$,
and otherwise ($\phi \neq \textbf{1}$) we have the simple descriptions
\begin{equation*}
    \begin{split}
        \pi_{\phi}(w) &= \min\braces*{\textup{rk}(H): \quad w\in H \le F_r, \,\,\,\EX_{w\to H}[\phi] \neq 0}, \\
        \textup{Crit}_{\phi}(w) &= \braces*{H\le F_r: \quad w\in H,\,\,\, \,\EX_{w\to H}[\phi] \neq 0, \,\,\, \textup{rk}(H) = \pi_{\phi}(w)}.
    \end{split}
\end{equation*}
Note that $\pi_{\phi}$ is automorphism-invariant: if $\xi\in \textup{Aut}(F_r), \psi\in \textup{Aut}(G)$ then 
\[\textup{Wit}_{\phi\circ \psi}(\xi(w)) = \xi\prn*{\textup{Wit}_{\phi}(w)},\quad \pi_{\phi\circ \psi}(\xi(w)) = \pi_{\phi}(w). \]

As we explain in Fact~\ref{fact_stable_irr_chars}, irreducible stable characters of $G\wr S_{\bullet}$ of degree $1$ are in natural bijection with $\hat{G}$, so we denote them by $\braces*{\chi_{\phi}}_{\phi\in \hat{G}}$.
The main result of \cite{Sho23I} is:
\begin{theorem}
\label{thm_Sho23I_main} 
(\cite[Theorem 1.9]{Sho23I})
Let $1\neq w\in F_r$, $G$ a compact group and $\phi\in \hat{G}$. 
Then    
\[ \EX_w\brackets*{\chi_{\phi}} = \mathscr{C}_{\phi}(w) \cdot n^{1-\pi_{\phi}(w)} + O\prn*{n^{-\pi_{\phi}(w)}}.\]
\end{theorem}
\noindent This theorem generalizes 
\cite[Theorem 1.8]{PP15} and 
\cite[Theorem 1.11]{magee2021surface}.

\noindent In the current paper we are interested in the case $\pi_{\phi}(w) = \pi(w)$, so we define for every $\phi\in \hat{G}$:

\begin{equation}
\label{eq_def_crit_phi_pi}
   \Cphipi(w)\defeq \sum_{H\in \textup{Crit}(w)} \EX_{w\to H}[\phi] = \begin{cases}
   \mathscr{C}_{\phi}(w) & \textup{ if } \pi_{\phi}(w) = \pi(w),\\
   0 & \textup{ if } \pi_{\phi}(w) > \pi(w).\\
   \end{cases}
\end{equation}
(Note that for $\phi\in \hat{G}-\{\textbf{1}\} $ and $H \in \textup{Wit}_{\phi}(w)$, $H\in \textup{Crit}(w) $ if and only if $\textup{rk}(H)=\pi(w)$). 

\noindent Now we can state the main result of the current paper in a unified formulation:
\begin{theorem}[Unified Main Result]
\label{thm_main_result_unified}
Let $w\in F_r$ be a non-power, $G$ a finite group, and $f\in \mathcal{A}(G) = \textup{span}\braces*{\widehat{G\wr S_{\bullet}}}$ a stable class function. Then
\[\EX_w[f] = \inner*{f, \mathbf{1}} + \sum_{\phi\in \hat{G}} \Cphipi(w) \inner*{f, \chi_{\phi}} n^{1-\pi(w)} + O\prn*{n^{-\pi(w)}}. \]
\end{theorem}

\begin{corollary}
\label{corollary_approx_power}
    Let $w = u^k$ where $u$ is a non-power and $k\in\Z_{\ge 1}$. Then for every $f\in \mathcal{A}(G)$:
    \[ \EX_w[f] = \EX_u\brackets*{f^{(k)}} = \inner*{f^{(k)}, \mathbf{1}} + \sum_{\phi\in \hat{G}} \Cphipi(u) \inner*{f^{(k)}, \chi_{\phi}} n^{1-\pi(u)} + O\prn*{n^{-\pi(u)}}.\]
\end{corollary}

\subsubsection*{Similar phenomena in the unitary groups $U(n)$}

We have mentioned the papers \cite{PP15}, \cite{HP22}, \cite{magee2021surface}, \cite{Ord20}, \cite{Sho23I} that the current paper generalizes.
Here we present some more related papers, all with the same phenomena connecting word measures and the primitivity rank and its generalizations.
The commutator length $\textup{cl}(w)$ of a word $w\in F_r$ is defined as the minimal $g\in \N$ such that there exist words $u_1, v_1, \ldots, u_g, v_g\in F_r$ satisfying $\brackets*{u_1, v_1}\cdots \brackets*{u_g, v_g} = w$.
In \cite{unitary_Magee_2019} it was shown that for every (stable) polynomial irreducible character of the unitary group $\mu\in \widehat{\mathcal{U}(n)}$ and every word $w\in F_r$,
\begin{equation}
\label{equation_ref_unitary_magee_19}
    \EX_w[\mu] = O\prn*{\dim(\mu)^{-2\textup{scl}(w)}}_{n\to \infty}
\end{equation}
where $s\textup{cl}(w)\defeq \lim_{n\to \infty} \frac{\textup{cl}\prn*{w^n}}{n}$ is the stable commutator length (see \cite{Calegari_2009}). 
Similar results were proven for orthogonal and symplectic groups in \cite{orthogonal_magee2019matrix}. 
It is conjectured in \cite[Conjecture 6.3.2]{nicolaus2019a} (and also in \cite[Conjecture 1.14]{HP22}) that 
\begin{equation}
\label{conjecture_scl_pi}
    \pi(w) - 1\le 2\cdot \textup{scl}(w),
\end{equation}
based on numerical computer experiments.

Theorem~\ref{thm_main_result_unified} looks very similar to the recent result by Brodsky (\cite{Brodsky22}) about word measures on unitary groups:
for a unitary matrix $A$, define $\xi_m(A)\defeq tr\prn*{A^m}$ for every $m\in \Z-\{0\}$.
Define the algebra of stable class functions on $\prn*{U(n)}_{n=1}^{\infty}$ as the infinite-dimensional polynomial ring $\mathcal{A}_U \defeq \C\brackets*{\xi_m}_{m\in \Z-\{0\}}$.
This algebra has a well defined inner product: $\inner*{f, g} \defeq \lim_{n\to\infty} \inner*{f, g}_{U(n)}$, as this inner product stabilizes for large enough $n$ (as in $\mathcal{A}(G)$).

For every $w\in F_r$, define the \textbf{commutator-witnesses} $\textup{CommWit}(w)$ of $w$ as the set of all algebraic extensions $H\in F_r$ that contain $w$ as a standard (orientable) surface-word of genus $\textup{rk}(H) / 2$, i.e.\ there exists a basis $\{u_1, v_1, \ldots, u_{\textup{rk}(H)/2}, v_{\textup{rk}(H)/2}\}$ of $H$ such that 
$w = \brackets*{u_1, v_1}\cdots \brackets*{u_{\textup{rk}(H)/2}, v_{\textup{rk}(H)/2}}. $
Similarly to the definition in equation~\eqref{eq_def_crit_phi_pi}, Brodsky defined the set of commutator-critical subgroups as
$\textup{CommCrit}(w) \defeq \textup{CommWit}(w) \cap \textup{Crit}(w).$
Since every surface-word is not primitive, $\textup{CommCrit}(w) = \braces*{H\in \textup{CommWit}(w):\quad \textup{rk}(H) = \pi(w)}$.
In fact, if $\textup{CommCrit}(w)$ is not empty, then $\pi(w) = 2\textup{cl}(w)$.
Now the analog of Theorem~\ref{thm_main_result_unified} is \cite[Corollary 1.9]{Brodsky22}:
let $f\in \mathcal{A}_U$ be a stable rational class function. 
Then for any non-power $w\in F_r$,
\[\EX_w[f] = \inner*{f, \mathbf{1}} + \inner*{f, \xi_1 + \xi_{-1}}\cdot \abs*{ \textup{CommCrit}(w) } \cdot n^{1-\pi(w)} + O\prn*{n^{-\pi(w)}}. \]
Compare this to \cite[Corollary 1.4]{HP22},
\[\EX_w[f] = \inner*{f, \mathbf{1}} + \inner*{f, \std}\cdot \abs*{ \textup{Crit}(w) } \cdot n^{1-\pi(w)} + O\prn*{n^{-\pi(w)}}, \]
and to the unified result of this paper when $G = C_m$:
\[\EX_w[f] = \inner*{f, \mathbf{1}} + \sum_{\substack{\phi\in \widehat{C_m}:\\ \pi_{\phi}(w) = \pi(w)}} \inner*{f, \chi_{\phi}}\cdot \abs*{ \textup{Crit}_{\phi}(w) } \cdot n^{1-\pi(w)} + O\prn*{n^{-\pi(w)}}. \]
The main technical tool in this paper is the (general) induction-convolution lemma, which plays a similar role in the proof of Theorem~\ref{thm_main_result_unified} as \cite[Theorem 1.7, main result]{unitary_Magee_2019} plays in the proof of \cite[Corollary 1.9]{Brodsky22}.
There are also similar results about word measures on general linear groups over finite fields, see \cite{ernstwest2021word}.

\subsection*{Overview of the paper}

\begin{enumerate}
    \item [Sec~\ref{section_stable_rep}:] 
    We investigate the filtered algebra $\mathcal{A}(G)$ of stable class functions, and give in Corollary~\ref{corollary_sInd_is_basis} a basis for it, which is suitable for the $w$-expectation operator.
    \item [Sec~\ref{section_multisets}:] 
    We present the equivalent categories $\moccFr$ (multisets of conjugacy classes of subgroups of $F_r$) 
    and $\mucgBFr$ (multi core graphs) that were defined in \cite{HP22}.
    \item [Sec~\ref{section_morphisms_free_alg}:] 
    We define free and algebraic morphisms, the lattices $\DecompB(\eta)$, $\Decompalg(\eta)$,  and the Möbius inversions on these lattices.
    \item [Sec~\ref{section_ICL}:] We state the Induction-Convolution Lemma, which is the most fundamental tool in this paper, and some of its consequences. 
    The lemma's proof is postponed to Section~\ref{section_proof_of_ICL}, since it is used as a black box in the rest of the paper. 
    \item [Sec~\ref{section_thm_main_result}:] We prove the main result of the paper - Theorem~\ref{thm_main_result}.
    \item [Sec~\ref{section_expansion}:] We show the main application: expansion of random Schreier graphs - Theorem~\ref{thm_expander_wreath}. 
    \item [Sec~\ref{section_proof_of_ICL}:] We prove the induction-convolution lemma.
    \item [App~\ref{appendix_stable_algebra}: ] We give some background about stability.
\end{enumerate}

\subsubsection*{Acknowledgments}
This project has received funding from the European Research
Council (ERC) under the European Union’s Horizon 2020 research and innovation programme
(grant agreement No 850956).

I am deeply indebted to Professor 
Doron Puder, who guided me throughout 
the research and gave helpful advice.
Special thanks to Nate Harman and Inna Entova-Eizenbud for useful discussions about stable representation theory.

\section{Stable Representations}
\label{section_stable_rep}
As described in Table~\ref{table:categoriesFI}, stable representations are defined for various classes of groups, and in this paper we are interested in the stable representations of $(G\wr S_n)_{n=0}^{\infty}$, where $G$ is an arbitrary finite group, or equivalently in $\FIG$-modules.

We start with the description of general properties, conjugacy classes and irreducible representations of wreath products with $S_n$.     
For further explanations about representation theory of wreath products, 
consult the book \cite{ceccherini-silberstein_scarabotti_tolli_2014}.
\begin{fact}
\cite[Proposition 2.1.3]{ceccherini-silberstein_scarabotti_tolli_2014}
    Let $G$ be a finite group and $n\in \N$.
    Identify $S_{n-1}$ with the stabilizer of the point $n$ 
    inside $S_n$. 
    Then $S_{n-1}$ acts on $[n]$, the wreath product $G\wr_{[n]} S_{n-1} \defeq G^n\rtimes S_{n-1} \cong G\times (G\wr_{[n-1]} S_{n-1})$ is a subgroup of $G\wr_{[n]} S_n$, and the quotient $(G\wr_{[n]} S_n) / (G\wr_{[n]} S_{n-1})$ has size $n$ and can be naturally identified with $S_n / S_{n-1} \cong [n]$.
\end{fact}

\noindent Recall $a_{t, c}$, defined in Equation~\eqref{eq_def_a_tc}. 
The following fact is \cite[Theorem 2.3.11]{ceccherini-silberstein_scarabotti_tolli_2014}:
\begin{fact}
The conjugacy class of an element $(v, \sigma)\in G\wr S_n$ is determined by $ \braces*{a_{t, c}(v, \sigma)}_{t\in \N, c\in \conj(G)}$.
\end{fact}

\subsection*{Multi partitions and irreducible characters}
\label{subsection_multi_partitions}

Recall the concept of partitions (or Young Diagrams) from Definition~\ref{def_partition}.
Note that a partition can be identified with a multi-set of positive integers.
\begin{definition}
We denote the set of all partitions by $\mathscr{P}$, and denote $\mathscr{P}_n \defeq \{\lambda\in \mathscr{P}: \lambda\vdash n\}$. 
For a partition $\lambda = \prn*{\lambda_1, \ldots, \lambda_{\ell}}$, we denote the number of parts by $\ell(\lambda)\defeq \ell$.
Moreover, for every $r\in \N$, we denote by $\#_r(\lambda)$ the number of parts in $\lambda$ of size $r$, so $\sum_{r} \#_r(\lambda) = \ell(\lambda)$.
Note that the empty partition $\emptyset\vdash 0$ is allowed, and has $\ell(\emptyset) = 0$.
The size $p(n)\defeq |\mathscr{P}_n|$ is called the partition function. 
\end{definition}

A \textbf{sub-partition} $\mu$ of a partition $\lambda$ is a sub-multiset, and is denoted by $\mu\le \lambda$. 
Recall from Definition~\ref{def_partition} that for $\lambda = (\lambda_1, \ldots, \lambda_{\ell}) \vdash k$ and $n\ge k + \lambda_1$ we denote $\lambda[n]\defeq (n-k, \lambda_1, \ldots, \lambda_{\ell})$.
\begin{definition}
    Let $\lambda$ be a Young diagram (or equivalently an integer partition) and $S$ a set. A \textbf{function} $\zeta\colon \lambda\to S$ is a function from the rows of $\lambda$ to $S$. We declare two such functions $\zeta_1, \zeta_2$ as equal if one is obtained from the other by permuting rows of the same length in $\lambda$.
\end{definition}

\begin{definition}
\label{def_multi_partitions}
Let $S$ be a set and $n\in \N$.
We define the sets of \textbf{multi-partitions} as
\[\mathscr{P}_n(S)\defeq \braces*{\VecLambda\colon S\to \mathscr{P} \quad | \quad \sum_{s \in S} |\VecLambda(s)| = n} \quad \textup{and} \quad \mathscr{P}(S) \defeq \biguplus_{n=0}^{\infty} \mathscr{P}_n(S).\]
We also denote $\mathscr{P}_{\le n}(S) \defeq \biguplus_{k=0}^{n} \mathscr{P}_k(S)$.
For $\VecLambda\in \mathscr{P}_n(S)$ we define $\abs*{\abs*{\VecLambda}} = n$. 
Note that $|\mathscr{P}_n(S)| = p^{\convolve{|S|}}(n)$
where $\convolve{}$ denotes convolution, e.g.\ 
$|\mathscr{P}_n(\{1, 2\})| = (p\convolve p)(n) = \sum_{k=0}^n p(k)p(n-k)$. 

\noindent Given a multi-partition $\VecLambda\in \mathscr{P}_n(S)$, we can \enquote{forget} the $S$-values, and get a partition $\floor{\VecLambda}\in \mathscr{P}_n$ which is the disjoint union (or sum of multi-sets) of $\VecLambda(s)$ over all $s \in S$.  
For every $\VecLambda\in \mathscr{P}(S)$ there is a unique function $\zeta \colon \floor{\VecLambda}\to S$ (when $\floor{\VecLambda}$ is regarded as a multi-set) satisfying $\VecLambda(s) = \zeta^{-1}(s)$ for every $s \in S$. We denote the unique such $\zeta$ by $\zeta_{\VecLambda}$: it maps every row of $\floor{\VecLambda}$ to the element $s\in S$ whose partition $\VecLambda(s)$ contains that row.
This lets us identify
\[ \mathscr{P}_n(S) \cong \braces{\zeta \colon \lambda \to S \quad | \quad \lambda\in \mathscr{P}_n}.\]
For example, if $\VecLambda(s_1) = (2, 1)$ and $\VecLambda(s_2) = (1)$ then $\floor{\VecLambda} = (2, 1, 1)$, and $\zeta_{\VecLambda}$ maps the row of length $2$ and one row of length $1$ to $s_1$, and the other to $s_2$.
\end{definition}

\begin{definition}
\label{def_vec_lambda_stable}
Let $G$ be a finite group, $\VecLambda = (\VecLambda(\phi))_{\phi\in \hat{G}}\in \mathscr{P}(\hat{G})$ a multi partition, and let $n \ge \abs*{\abs*{\VecLambda}} + \VecLambda(\textbf{1})_1$. Define $\VecLambda[n] \in \mathscr{P}_n(\hat{G})$ via
\begin{equation}
\begin{split}
    \VecLambda[n](\phi) \defeq &
    \begin{cases}
    \VecLambda(\phi) & \textrm{ if } \phi \neq \textbf{1}, \\
    \prn*{n-\abs*{\abs*{\VecLambda}}, \VecLambda(\textbf{1})_1, \ldots, \VecLambda(\textbf{1})_{\ell({\VecLambda(\textbf{1})})}} & \textrm{ if } \phi = \textbf{1}. \\
    \end{cases}
\end{split}
\end{equation}
\end{definition}

Now we characterize the irreducible characters of 
$G\wr S_n$ using multi-partitions. 
\begin{fact}
    [See {\cite[Theorem 2.6.1]{ceccherini-silberstein_scarabotti_tolli_2014}}]
    Let $G$ be a finite group. 
    Similarly to the famous bijection between 
    $\mathscr{P}_n$ and $\widehat{S_n}$, 
    there is also a natural bijection between 
    $\PnGhat$ and $\widehat{G\wr S_n}$. 
\end{fact}
We denote the irreducible character 
corresponding to $\VecLambda\in \PnGhat$ by 
$\chi_{\VecLambda}\in \widehat{G\wr S_n}$.
The following fact is proven in Theorem~\ref{thm_irred_basis_of_AG}:
\begin{fact}
\label{fact_stable_irr_chars}
    For every $d\in \N$ there is a bijection between $\PdGhat$ and $\braces*{\chi\in \widehat{G\wr S_{\bullet}}: \deg(\chi) = d}$, that maps $\VecLambda\in \PdGhat$ to the unique stable character that coincides with the sequence $\prn*{\chi_{\VecLambda[n]}}_{n=\abs*{\abs*{\VecLambda}} + \VecLambda(\textbf{1})_1}^{\infty}$.
    In particular, since we may identify $\mathscr{P}_1\prn*{\hat{G}}$ with $\hat{G}$, there is a bijection between $\hat{G}$ and irreducible stable characters of $G\wr S_{\bullet}$ of degree $1$, which are denoted in Theorem~\ref{thm_Sho23I_main} as $\braces*{\chi_{\phi}}_{\phi\in \hat{G}}$.
\end{fact}

\begin{example}
    Let $\phi, \psi\in \hat{G}-\{\textbf{1}\}$ be different (i.e.\ $\phi\neq \psi$), and let $(v, \sigma)\in G\wr S_n$. 
    
\begin{table}[ht!]
\centering
\begin{tabular}{|| c | c | c | c ||} 
\hline
$\VecLambda$
& $\VecLambda[n]$
& $\chi_{\VecLambda[n]}(v, \sigma)$
& Degree\\ [0.5ex] 
\hline\hline

\rule{0pt}{5ex}    
$\displaystyle{ \prn*{\ontop{(1)}{\phi}}}$
& $\displaystyle{ \prn*{\ontop{(n-1)}{\textbf{1}}; \ontop{(1)}{\phi}}}$ 
& $\displaystyle{ \sum_{i:\, \sigma(i) = i} \phi(v(i))}$
& $ {1}$ \\  [1ex] \hline

\rule{0pt}{5ex}    
{$\displaystyle{ \prn*{ \ontop{(1)}{\phi}; \ontop{(1)}{\psi}}}$}
& $\displaystyle{ \prn*{\ontop{(n-2)}{\textbf{1}}; \ontop{(1)}{\phi}; \ontop{(1)}{\psi}}}$ 
& $\displaystyle{ \sum_{\substack{i\neq j: \\\sigma(i) = i, \,\,\sigma(j) = j}} \phi(v(i))\psi(v(j))}$
& $ {2}$ \\  [1ex] \hline
 \hline
\end{tabular}
\label{table:example_irreps_wreath}
\caption{Examples of formulas of stable irreducible characters.}
\end{table}
\FloatBarrier
\end{example}

For the next example, we denote 
the characters of $C_2 = \{\pm 1\}$ by $\textbf{1}, \textbf{-1}$, and denote $\binom{[n]}{2} \defeq \{A\subseteq \{1,\ldots, n\}: |A|=2\}$.
\noindent Recall from Table~\ref{table:wreath_product_papers} that the result of \cite{Ord20} is that for every non-power $w\in F_r$, $\EX_w\brackets*{\chi_{\prn*{\ontop{(n-2)}{\textbf{1}};\ontop{(2)}{\textbf{-1}}}}} = O\prn*{n^{-\pi(w)}}$. 

\begin{example}
\label{Ordos_char}
Let $H_n\defeq C_2\wr S_n$ be the Hyperoctahedral 
group, and let 
$\VecLambda = \prn*{\ontop{(n-2)}{\textbf{1}};\ontop{(2)}{\textbf{-1}}}\in \mathscr{P}_n\prn*{\hat{C_2}}$. 
The irrep 
$\rho\in \Hom\prn*{H_n, \textup{GL}_{\binom{n}{2}}(\Q)}$ 
that maps every 
$v\in \{\pm 1\}^n, \,\,\sigma\in S_n$ to
\[ \rho(v, \sigma) = 
\prn*{\bbone_{\sigma(\{i_2, j_2\}) = \{i_1, j_1\}}
v(i_1) v(j_1)}_{{\{i_1, j_1\}, \{i_2, j_2\}}
\in \binom{[n]}{2}} \]
yields the character 
$\chi_{\VecLambda} = 
\textup{tr}\circ \rho$, 
and we get the formula
\[ \chi_{\VecLambda}(v, \sigma) = 
\sum_{\{i, j\}: \,\,\sigma(\{i, j\}) = \{i, j\}} 
v(i)v(j). \]
\end{example}

\subsection*{$\FIG$-modules}

Recall that $G$ is a fixed finite group.
The algebra $\mathcal{A}(G)$ is closely related to $\FIG$-modules.
To use some known results about $\FIG$-modules, we first explain what $\FIG$-modules are.
We define $\FIG$-modules only over $\FG$, and besides that we follow \cite[\S 1.1]{Sam_2019}.

\begin{definition}
For every $n\in \N$, we fix an embedding $S_n\overset{\cong}{\to} \textup{Stab}_{S_{n+1}}(n+1)\le S_{n+1}$ of symmetric groups, which induces an embedding $\iota_n\colon G\wr S_n\hookrightarrow G\wr S_{n+1}$.    
\end{definition}

\begin{definition}
\noindent An $\FIG$-module is a sequence $(V_n)_{n=0}^{\infty}$ of $G\wr S_n$-representations over $\FG$ with linear maps $\psi_n\colon V_n\to V_{n+1}$, 
such that for every $g\in G_n$, the following diagram commutes
\[\begin{tikzcd}
	{V_n} & {V_{n+1}} \\
	{V_n} & {V_{n+1}}
	\arrow["{\psi_n}", from=1-1, to=1-2]
	\arrow["{\iota_n(g)}", from=1-2, to=2-2]
	\arrow["g"', from=1-1, to=2-1]
	\arrow["{\psi_n}", from=2-1, to=2-2]
\end{tikzcd}\]
and for every $k < n$, the subgroup $\braces*{(v, \sigma)\in G\wr S_n: \,\, \sigma(i) = i \textup{ and } v(i) = 1 \textup{ for every } i\le k}$ acts trivially on the image $\textup{Im}\prn*{\psi_{k\to n}} \subseteq V_n$, where $\psi_{k\to n}\defeq \psi_{n-1}\circ \cdots \circ \psi_k$ (and $\psi_{n\to n}\defeq \textup{id}$). 
We omit $\psi_n$ from the notation when there is no risk of ambiguity.
An $\FIG$-submodule of $(V_n)_{n=0}^{\infty}$ is a sequence $(W_n)_{n=0}^{\infty}$ of $G\wr S_n$-subrepresentations, with the restricted maps $\psi_n\restriction_{W_n}\colon W_n\to W_{n+1}$.
The \textbf{character} of an $\FIG$-module $(V_n)_{n=0}^{\infty}$ is the sequence $\prn*{\chi_n\colon G\wr S_n\to \FG}_{n=0}^{\infty}$ of characters of $V_n$.
\end{definition}

\begin{remark}
Define $\FIG$ as the following category:
\begin{itemize}
    \item An object is a finite set $X$ with a free action $G\acts X$ that has finitely many orbits.
    \item Given objects $X$ and $Y$, morphisms $X\to Y$ are $G$-equivariant injections.
\end{itemize}
The original (equivalent!) definition of an $\FIG$-module over $\FG$ is a functor from the category $\FIG$ to the category $\FG$\textbf{-Vec} of vector spaces over $\FG$.
Naturally, $\FIG$-modules form a category themselves: a \textbf{morphism} between $\FIG$-modules is a natural transformation of functors. 
This gives rise to natural definitions of a $\FIG$-submodule and a quotient $\FIG$-module.
\end{remark}

\begin{definition}[Generation of $\FIG$-module]
Let $V=(V_n)_{n=0}^{\infty}$ be an $\FIG$-module, and let $B\subseteq \biguplus_{n=0}^{\infty} V_n$ be a set of vectors. 
We define the $\FIG$-module \textbf{generated} by $B$ as 
\[\biguplus_{n=0}^{\infty} \textup{span}_{\FG} \prn*{\bigcup_{k=0}^n  \prn*{G\wr S_n}\cdot \psi_{k\to n} \prn*{B\cap V_k}}.\]
This is the minimal $\FIG$-submodule containing $B$.
An $\FIG$-module is called \textbf{finitely generated} if it has a finite generating set, and called \textbf{generated in degree $\le d$} if it is generated by $\biguplus_{q=0}^d V_q$.
\end{definition}

\begin{definition}
\label{def_AtagG}
    We define $\mathcal{A}'(G)$ as the set of sequences of class functions $\prn*{f_n\colon G\wr S_n\to \FG}_{n=0}^{\infty}$ that coincide with a sequence from $\textup{span}_{\FG} \braces*{\prn*{\chi_{\VecLambda[n]}}_{n=\abs*{\abs*{\VecLambda}} + \VecLambda(\textbf{1})_1}^{\infty}}$ for every large enough $n$.
\end{definition}

We denote $f_n\to 0$ whenever $f_n\colon G\wr S_n\to \FG$ is identically zero for every large enough $n$.

\begin{theorem}
\label{thm_GL15_fin_gen}
\cite[Theorem 1.12]{gan2016coinduction} 
Let $V = (V_n)_{n=0}^{\infty}$ be an $\FIG$-module with $\dim(V_n) < \infty$ for every $n$.
Then $V$ is finitely generated if and only if for every large enough $n$, the $\FG$-span of the union of all $G\wr S_{n+1}$-orbits of $\psi_n(V_n)$ is all of $V_{n+1}$. 
Moreover, if $V$ is finitely generated, then the character of $V$ is an element of $\mathcal{A}'(G)$, and $\psi_n$ is injective for every large enough $n$.
\end{theorem}

It also follows from \cite{gan2016coinduction} that for every $d\in \N$ and $\VecLambda\in \PdGhat$, $\prn*{\chi_{\VecLambda[n]}}_{n=\abs*{\abs*{\VecLambda}} + \VecLambda(\textbf{1})_1}^{\infty}$ is the character of a finitely generated $\FIG$-module, generated in degree $\le d$.

\begin{definition}
By definition, for every $f\in \mathcal{A}'(G)$ there are $\VecLambda_1, \ldots, \VecLambda_k \in \PGhat$ and coefficients $\beta_1, \ldots, \beta_k\in \FG-\{0\}$ such that 
$\abs*{f_n - \sum_{i=1}^{k} \beta_i \cdot \chi_{\VecLambda_i [n]}}\to 0$. 
We define $\deg\colon \mathcal{A}'(G)\to \Z_{\ge 0}$ by $\deg(f)\defeq \max_{1\le i\le k} \abs*{\abs*{\VecLambda_i}}$.
Note that $\deg(f) = 0$ if and only if $f_n$ is constant for every large enough $n$.
We also denote 
\[\AtagGled \defeq \braces*{f\in \mathcal{A}'(G): \, \deg(f)\le d}, \quad\quad \AGled \defeq \braces*{f\in \mathcal{A}(G): \, \deg(f)\le d}.\]
\end{definition}

\begin{theorem}
    \cite[Part of Theorem 7.15, Example 7.16]{ryba2019stable} 
    The pointwise addition and multiplication give $\mathcal{A}'(G)$ the structure of a filtered algebra with respect to $\deg$.
\end{theorem}

By \cite[Theorem 2.2]{casto2016fig} and the paragraph after Theorem~\ref{thm_GL15_fin_gen}, $\AtagGled$ coincides with the $\FG$-span of the characters of finitely generated $\FIG$-modules over $\FG$ generated in degree $\le d$. This also follows from \cite[Theorem 7.15]{ryba2019stable}.

\subsection*{Bases for the stable algebra}

Recall from Definition~\ref{def_AG} that $\mathcal{A}(G)\defeq \FG\brackets*{(a_{t,\, c})_{t\in \N,\,\, c\in \conj(G)}}$.

\begin{theorem}
\cite[Theorem 2.2]{casto2016fig}
Let $G$ be a finite group, and $V$ a finitely generated $\FIG$-module over $\FG$, generated in degree $\le d$.
Then there is a polynomial $P\in \mathcal{A}(G)$ of degree $\le d$, such that for every large enough $n$, the character of $V$ coincides with $P$ on $G\wr S_n$.
\end{theorem}

\begin{corollary}
There is a function $\xi_d \colon \AtagGled \to \AGled$ such that for every $f\in \AtagGled$, $\abs*{f - \xi_d(f)}\to 0$.
Note that by definition, $\xi_d$ maps sequences of irreducible characters to $\widehat{G\wr S_{\bullet}}$.
\end{corollary}

\begin{theorem}
\label{thm_irred_basis_of_AG}
For every $d\in \N$, $\widehat{G\wr S_{\bullet}}\cap \AGled$ is a linear basis of the space $\AGled$.
Moreover, for every $\VecLambda\in \PdGhat$ there is a unique polynomial $\chi_{\VecLambda[\bullet]}\in \mathcal{A}(G)$ of degree exactly $d$ that coincides with the sequence $\chi_{\VecLambda[n]}$ for every large enough $n$, and $\widehat{G\wr S_{\bullet}}\cap \AGled = \braces*{\chi_{\VecLambda[\bullet]}}_{\VecLambda\in \PledGhat}$.
\end{theorem}

\begin{proof}
    Let $\VecLambda\in \PledGhat$. Choose an arbitrary prefix for the sequence $\prn*{\chi_{\VecLambda[n]}}_{n=\abs*{\abs*{\VecLambda}} + \VecLambda(\textbf{1})_1}^{\infty}$ to get an element $f_{\VecLambda}\in \mathcal{A}'(G)$.
    The set $\braces*{\xi_d\prn*{f_{\VecLambda}}}_{\VecLambda\in \PledGhat} \subseteq \widehat{G\wr S_{\bullet}}\cap \AGled$ is linearly independent, as for sufficiently large $n$ it coincides with a set of distinct irreducible characters of $G\wr S_n$.
    Thus it suffices to prove $\abs*{\PledGhat} = \dim\prn*{\AGled}$. 
    Indeed, $\dim\prn*{\AGled}$ is the number of monomials in $\prn*{a_{t, c}}_{t\in \N, c\in \conj(G)}$ with total degree at most $d$, which is $\sum_{q=0}^d p^{\convolve{|\conj(G)|}}(q) = \abs*{\PledGhat}$. The rest of the claims follows.
\end{proof}

\begin{corollary}
\label{corollary_dim_poly}
    Let $\VecLambda\in \PdGhat$. Then there is a polynomial $P\in \Q[X]$ of degree $d$ such that $\dim\prn*{\chi_{\VecLambda[n]}} = P(n)$ for every large enough $n$.
\end{corollary}

We introduce yet another basis, which is the most convenient for our needs.
We have already discussed $\{\chi_{\phi}\}_{\phi\in \hat{G}}$, the irreducible stable characters of degree 1, that we defined by $ \chi_{\phi} \defeq \prn*{\chi_{\binom{(1)}{\phi}[n]}}_{n=1}^{\infty}$.
It turns out (as a special case of \cite[Theorem 2.6.1]{ceccherini-silberstein_scarabotti_tolli_2014}) that $\{\chi_{\phi} + \bbone_{\phi = \textbf{1}}\}_{\phi\in \hat{G}}$ have a uniform formula, which is how we compute $\chi_{\phi}$:
\begin{definition}
\label{def_Ind_phi}
\cite[Definition 1.7]{Sho23I}
Let $\phi\in \hat{G}$. 
We give two equivalent definitions for the $\phi(1)\cdot n$-dimensional character $\Indnphi\in \textup{char}(G\wr_{[n]} S_n)$ and the corresponding stable character $\Indphi \defeq \prn*{\Indnphi}_{n\in \N}$:
\begin{itemize}
    \item A formula: $\Indphi = \sum_{c\in \conj(G)} \phi(c)\cdot a_{1, c}$. That is, for every $v\in G^n, \sigma\in S_n$:
    \[ \Indnphi(v, \sigma) \defeq \sum_{i: \,\,\sigma(i)=i} \phi(v(i)). \]
    \item Character description: extend $\phi$ to be defined over $G\wr_{[n]} S_{n-1}$, where $S_{n-1}\le S_n$ is the stabilizer of the point $n$, via $\phi(v, \sigma) = \phi(v(n))$.
    Then 
    \[\Indnphi \defeq \Ind_{G\wr_{[n]} S_{n-1}}^{G\wr_{[n]} S_n} \phi. \]
\end{itemize}
\end{definition}

In \cite[Proposition 3.2]{Sho23I}, it is shown that the two definitions coincide and give a well defined character, and that the stable irreducible character $\chi_{\phi}$ is $\prn*{\Indnphi- \bbone_{\phi = \textbf{1}}}_{n=1}^{\infty}$.
When $\phi = \textbf{1}$, $\chi_{\textbf{1}}\defeq \prn*{\std(S_n)}_n$ is the standard $(n-1)$-dimensional irreducible character of $S_n$, inflated to $G\wr S_n$ via the epimorphism $G\wr S_n\twoheadrightarrow S_n$; explicitly $\forall v\in G^n, \sigma\in S_n: \,\,\, \chi_{\textbf{1}}(v, \sigma) = \#\{\textup{fixed points of }\sigma\} - 1$.\\

\begin{definition}[$\FGconjG, f^{(n)}$]
\label{def_summary_func}
Let $G$ be a finite group.
Its algebra of $\FG$-valued class functions is $\FGconjG \cong Z\prn*{\FG\brackets*{G}}$, the center of the group algebra $\FG[G]$.
For any $f\in \FGconjG$ and $n\in \Z$, denote $f^{(n)}(x)\defeq f\prn*{x^n}$. 
\end{definition}

To get our desired basis for $\mathcal{A}(G)$ we need to analyze $(\Indphi)^{(k)}$ and show that 
\[ \textup{span}_{\FG}\prn*{\braces*{a_{t, c}}_{t\le k, \,\,c\in \conj(G)}} 
= \textup{span}_{\FG}\prn*{\braces*{(\Indphi)^{(m)}}_{m\le k, \,\,\phi\in \hat{G}}}. \]
This is the content of the following two propositions:

\begin{proposition}
\label{prop_Indphi_m_is_stable}
Let $\phi\in \hat{G}$ and $k\in \Z_{\ge 1}$. Then 
$(\Indphi)^{(k)}\in \textup{span}_{\FG}\prn*{\braces*{a_{t, c}}_{t\le k,\,\, c\in \conj(G)}}$.
Explicitly,
\[ (\Indphi)^{(k)} = \sum_{t\divides k} t \sum_{c\in \conj(G)} \phi\prn*{c^{k/t}}\cdot a_{t, c}.\]
Note that $\phi\prn*{c^{k/t}}$ is well defined, i.e.\ does not depend on a representative of $c$. 
\end{proposition}

\begin{proof}
Since for every $\sigma\in S_n, v\in G^n$ we have 
$(v, \sigma)^k = \prn*{\prod_{\ell=0}^{k-1}\sigma^{\ell}.v, \,\sigma^k}$, we get
\begin{equation*}
\begin{split}
    \prn*{\Indphi}^{(k)}(v, \sigma)
    &= \Indphi\prn*{\prod_{\ell=0}^{k-1}\sigma^{\ell}.v, \,\sigma^k}
    = \sum_{i\in \textup{fix}\prn*{\sigma^k}} \phi\prn*{\prod_{\ell=0}^{k-1} v\prn*{\sigma^{-\ell}.i}}\\
    &= \sum_{t\divides k} \sum_{s \in t\textrm{-cycles}(\sigma)} \sum_{i\in s} \phi\prn*{\prn*{\prod_{j\in s} v(j)}^{k/t}}
    = \sum_{t|k} t \sum_{c\in \conj(G)} a_{t, c}(v, \sigma) \phi\prn*{c^{k/t}}.
\end{split}
\end{equation*}
\end{proof}   

\begin{example}
As we prove in Equation~\eqref{equation_proved_a_tc_approx}, $\inner*{a_{t, c}, \textbf{1}} = \frac{|c|}{|G|\cdot t}$. Therefore 
\[ \inner*{\prn*{\Indphi}^{(k)}, \textbf{1}}
= \frac{1}{|G|} \sum_{t\divides k} \sum_{c\in \conj(G)} \phi\prn*{c^{k/t}} |c|
= \sum_{t\divides k} \inner*{ \phi^{(t)}, \textbf{1}}
.\]
For example, if $G = C_m$ and $\phi\colon G\hookrightarrow \Ss^1$, then $\phi^{(t)}\in \widehat{C_m}$ and $\phi^{(t)} = \textbf{1}$ if and only if $m\divides t$, so
\[\inner*{\prn*{\Indphi}^{(k)}, \textbf{1}}
= \bbone_{m\divides k}\cdot \#\{\textup{divisors of } k/m\}.\]
\end{example}

\begin{proposition}
\label{prop_Indphi_span_AG}
For every $k\in \N, c\in \conj(G)$, we 
have
$ a_{k, c}\in \textup{span}_{\FG}\prn*{\braces*{(\Indphi)^{(m)}}_{m\le k, \,\,\phi\in \hat{G}}}. $
\end{proposition}

\begin{proof}
By induction on $k$.
For $k=0$ we 
have\footnote{We have defined $a_{k,c}$ only for 
$k\in \Z_{\ge 1}$, but it is natural to define 
$a_{0,c} = 0$ for every $c\in \conj(G)$.}
 $a_{0, c} = 0$ so there is nothing to prove.
Assume that we have proved for every $k' < k$. Then, due to the previous proposition, 
\[ (\Indphi)^{(k)} - k\sum_{c\in \conj(G)}\phi(c)\cdot a_{k, c} 
        = \sum_{\substack{t\divides k,\\t\neq k}} t \sum_{c\in \conj(G)} \phi\prn*{c^{k/t}}\cdot a_{t, c} \]
which is inside $\textup{span}_{\FG}\prn*{\braces*{(\Indphi)^{(m)}}_{m\le k, \,\,\phi\in \hat{G}}}$ by the induction hypothesis, so 
\begin{equation}
\label{eq_proof_a_tc_in_span}
    \sum_{c\in \conj(G)}\phi(c)\cdot a_{k, c} \in \textup{span}_{\FG}\prn*{\braces*{(\Indphi)^{(m)}}_{m\le k, \,\,\phi\in \hat{G}}}.
\end{equation}
By Schur orthogonality, for every $c'\in \conj(G)$ we have
\[ \sum_{\phi\in \hat{G}}\phi(c)\overline{\phi(c')} = |G|/|c|\cdot\bbone_{c=c'}. \]
Now we replace $c$ by $c'$ in Equation~\eqref{eq_proof_a_tc_in_span} to get
\begin{equation*}
    \begin{split}
        a_{k, c}
        &= \sum_{c'\in \conj(G)} \bbone_{c=c'}\cdot a_{k, c'}\\
        &= \sum_{c'\in \conj(G)} \prn*{\frac{|c|}{|G|} \sum_{\phi\in \hat{G}}\overline{\phi(c)}\phi(c')}\cdot a_{k, c'}\\
        &= \frac{|c|}{|G|}\sum_{\phi\in \hat{G}}\overline{\phi(c)} \prn*{\sum_{c'\in \conj(G)} \phi(c')\cdot a_{k, c'}}\\
        &\in \textup{span}_{\FG}\prn*{\braces*{(\Indphi)^{(m)}}_{m\le k, \,\,\phi\in \hat{G}}}
    \end{split}
\end{equation*}
as needed.
\end{proof}

Now we are ready to introduce the desired basis for $\mathcal{A}(G)$: 

\begin{definition}
\label{def_summary_s}
Define the \textbf{summary} function
\[ \s \colon \mathscr{P}\prn*{\FGconjG}\to \FGconjG, \quad
\s \prn*{\VecLambda} \defeq \prod_{\phi\in \FGconjG: \,\,\VecLambda(\phi)\neq \emptyset} \prod_i \phi^{(\VecLambda(\phi)_i)}.\]
Equivalently, we may regard $\floor{\VecLambda}$ as a multi-set of natural numbers and define for every $x\in G$
\[ \s \prn*{\VecLambda}(x) \defeq \prod_{n\in \floor{\VecLambda}} \zeta_{\VecLambda}(n)\prn*{x^n}. \]
For example, let $\phi, \psi\in \hat{G}$ be different, and let $\VecLambda\in \PGhat\subseteq \mathscr{P}\prn*{\FGconjG}$ satisfy $\VecLambda(\phi) = (2, 1)$, $\VecLambda(\psi) = (1)$ and $\VecLambda(\theta) = \emptyset$ for every $\theta\in \hat{G}-\{\phi, \psi\}$. Then $\s\prn*{\VecLambda}(x) = \phi(x^2)\phi(x)\psi(x)$ for every $x\in G$.
\end{definition}

\begin{definition}[$\IndVecLambda, \s\IndVecLambda$]
\label{def_ind_lambda}
Let $\VecLambda\in \PGhat$. Define $\IndVecLambda\in \mathscr{P}(\textup{char}(G_{\bullet}))$
by replacing every $\phi\in \hat{G}$ with $\Indphi\in \textup{char}(G_{\bullet})$. 
(Recall that $\Indphi$ is indeed stable). 
Equivalently, $\IndVecLambda$ is the only multi-partition with $||\IndVecLambda|| = \abs*{\abs*{\VecLambda}}$ and
$\IndVecLambda(\Indphi) = \VecLambda(\phi)$ for every $\phi\in \hat{G}$.
To ease notation, denote $\s\IndVecLambda\defeq \s\prn*{\IndVecLambda}$. Explicitly,
\[ \s\IndVecLambda \defeq \prod_{\phi\in \hat{G}} \prod_{i} \prn*{\Indphi}^{(\VecLambda(\phi)_i)}. \]
For example, for $\VecLambda$ as in the example in Definition~\ref{def_summary_s}, we have $\s\IndVecLambda = \prn*{\Indphi}^{(2)}\cdot \Indphi\cdot \Ind\psi$, so by Proposition~\ref{prop_Indphi_m_is_stable} with $k = 2$ and by Definition~\ref{def_Ind_phi}, for every $n\ge 2$, $v\in G^n$ and $\sigma\in S_n$,
\[ \s\IndVecLambda(v, \sigma) = \prn[\Bigg]{\sum_{i:\,\,\sigma(i) = i} \phi\prn{v(i)^2} + 2\sum_{\substack{\{i, j\}:\,\, i\neq j,\\ \sigma(i) = j,\,\, \sigma(j) = i}} \phi\prn{v(i)v(j)}} \cdot \sum_{i:\,\,\sigma(i) = i} \phi(v(i)) \cdot \sum_{i:\,\,\sigma(i) = i} \psi(v(i)). \]
\end{definition}

We conclude now that $\braces*{\s\IndVecLambda}_{\VecLambda\in \PledGhat}$ is a filtered linear basis\footnote{Recall that a filtered basis $\mathcal{B}$ for a filtered algebra $\mathcal{A}$ is a linear basis such that for every $d\in \N$, $\braces*{b\in \mathcal{B}: \deg(b)\le d}$ is a linear basis of the space $\braces*{a\in \mathcal{A}: \deg(a)\le d}$. } for $\AGled$:

\begin{corollary}
\label{corollary_sInd_is_basis}
The functions $\braces*{\prn*{\Indphi}^{(m)}}_{m=1}^{\infty}$ are stable class functions with $\deg\prn*{\prn*{\Indphi}^{(m)}} = m$ (by Proposition~\ref{prop_Indphi_m_is_stable}), and they generate the algebra $\mathcal{A}(G)$ (by Proposition~\ref{prop_Indphi_span_AG}).
Consequently, since
$\braces*{\s\IndVecLambda}_{\VecLambda\in \PGhat}$
are the monomials in $\braces*{\prn*{\Indphi}^{(m)}}_{m=1}^{\infty}$,
and there are $\abs*{\PledGhat} = \dim\prn*{\AGled}$ such monomials,
we get that $\braces*{\s\IndVecLambda}_{\VecLambda\in \PGhat}$
is a filtered linear basis for $\mathcal{A}(G)$, with $\deg\prn*{\s\IndVecLambda} = \abs*{\abs*{\VecLambda}}$.
\end{corollary}

\section{Multisets of Free Groups and Graphs}
\label{section_multisets}

Similarly to the case of $S_n$, our analysis of word measures on $G\wr S_n$ uses the categories $\moccFr, \mucgBFr$ defined in \cite{HP22}. 

\subsection*{Core graphs}
This subsection is based on \cite[Sections 3.1-3.2]{HP22}.
Let $B = \{b_1, \ldots, b_r\}$ be a basis of $F_r$, and consider the bouquet 
$\Omega_B$ on $r$ circles with distinct labels
from $B$ and arbitrary orientations and with wedge point $\otimes$. 
Then $\pi_1(\Omega_B, \otimes)$ is naturally identified
with $F_r$.
For example,
\[\pi_1\prn*{\begin{tikzcd}
	\otimes \arrow[loop left]{l}{b_1} \arrow[loop right]{r}{b_2}
\end{tikzcd}, \,\otimes} = \textup{Free}(\{b_1, b_2\}).\]
The notion of ($B$-labeled) core graphs, introduced in \cite{Stallings1983TopologyOF}, refers to finite, connected graphs with every vertex having degree at least two (so no leaves and no isolated vertices), 
that come with a graph morphism to
$\Omega_B$ which is an immersion, namely, locally injective. 
In other words, this is a finite connected graph with at least one edge and no leaves, with edges that are directed and labeled by the elements of $B$, such that for every vertex $v$ and every $b\in B$, there is at most one incoming $b$-edge and at most one outgoing $b$-edge at $v$. We stress that multiple edges between two vertices and loops at vertices are allowed.

There is a natural one-to-one correspondence 
between finite $B$-labeled core graphs and 
conjugacy classes of non-trivial finitely generated 
subgroups of $F_r$.
Indeed, given a core graph $\Gamma$ as above, 
pick an arbitrary vertex $v$ and consider the 
\enquote{labeled fundamental group} 
$\pilab(\Gamma, v)$: closed paths in a graph 
with oriented and $B$-labeled edges correspond to 
words over the alphabet $B$. 
In other words, if 
$p\colon \Gamma\to \Omega_B$ is the immersion, 
then $\pilab(\Gamma, v)$ is the subgroup 
$p_*(\pi_1(\Gamma, v))$ of 
$\pi_1(\Omega_B, \otimes) = F_r$.
The conjugacy class of $\pilab(\Gamma, v)$ 
is independent of the choice of $v$ and it 
corresponds to $\Gamma$. 
We denote it by $\pilab(\Gamma)$. 
Conversely, if $H\le F_r$ is a non-trivial 
finitely generated subgroup, the conjugacy class $H^{F_r}$ corresponds to a finite core graph, denoted $\Gamma_B(H^{F_r})$, which can be obtained in different manners - see \cite{HP22} and the references therein. 
\begin{example}
For $F_3 = \textup{Free}(\{a, b, c\})$, the 
core graph of $\inner*{c, aca, a^{-1}ba}^{F_3}$ is
\[\begin{tikzcd}
	\bullet \arrow[loop left]{l}{c} & \bullet \\
	& \bullet \arrow[loop right]{r}{b}
	\arrow["a", from=1-1, to=1-2]
	\arrow["a", from=2-2, to=1-1]
	\arrow["c", from=1-2, to=2-2]
\end{tikzcd}\]
\end{example}
\noindent Here, we consider core graphs which are not necessarily connected:
\begin{definition}
\cite[Definition 3.1]{HP22}
Let $B$ be a basis of $F_r$.
A $B$-labeled multi core graph is a disjoint union of finitely many core graphs. 
In other words, this is a finite graph, not necessarily connected, with no leaves and no isolated vertices, 
and which comes with an immersion to $\Omega_B$. 
We denote the set of $B$-labeled multi core graphs by $\mucgBFr$.
\end{definition}

\begin{definition}
\cite[Definition 3.2]{HP22}
A morphism $\eta\colon \Gamma\to \Delta$ between $B$-labeled multi core graphs is a graph-morphism which 
commutes with the immersions to $\Omega_B$.
\[\begin{tikzcd}
	\Gamma && \Delta \\
	& \Omega_B
	\arrow["\eta", from=1-1, to=1-3]
	\arrow["{i_{\Gamma}}"', from=1-1, to=2-2]
	\arrow["{i_{\Delta}}", from=1-3, to=2-2]
\end{tikzcd}\]
In particular, the morphism $\eta$ is an immersion, and it preserves orientations and edge-labels.
\end{definition}

Because a core graph corresponds to a conjugacy class of non-trivial finitely generated subgroups of $F_r$, a multi core graph corresponds to a multiset of such objects. 
This motivates the following definition:

\begin{definition}
\label{def_mocc}
\cite[Definition 2.1]{HP22}
For a free group $F_r$ of rank $r$, we define the category $\moccFr$, where objects are \textbf{m}ultisets  \textbf{o}f \textbf{c}onjugacy \textbf{c}lasses of non-trivial finitely generated subgroups of $F_r$.
Given 2 objects
$\vec{\mathcal{H}}= \{H_1^{F_r}, \ldots, H_{\ell}^{F_r}\}$ 
and 
$\vec{\mathcal{J}}= \{J_1^{F_r}, \ldots, J_{m}^{F_r}\}$
of $\moccFr$, a \textbf{morphism} 
$\eta\colon \mathcal{H}\to \mathcal{J}$ consists of a map 
$f\colon [\ell]\to [m]$
and a choice of representatives 
$\overline{H_1}\in H_1^{F_r}, \ldots, \overline{H_{\ell}}\in H_{\ell}^{F_r}, \overline{J_1}\in J_1^{F_r}, \ldots, \overline{J_{m}}\in J_{m}^{F_r}$
so that
$\overline{H_i}\le \overline{J_{f(i)}}$
for all
$i\in [\ell]$.
\end{definition}

Every basis $B$ of $F_r$ gives rise to a one-to-one correspondence $\mucgBFr\leftrightarrow \moccFr$, that is between $B$-labeled multi core graphs and finite multisets of conjugacy classes of non-trivial finitely generated subgroups of $F_r$.
For a multi core graph $\Gamma\in \mucgBFr$ we let $\pilab(\Gamma)$ denote the corresponding multiset in $\moccFr$, and for a multiset $\mathcal{H}\in \moccFr$ we let $\Gamma_B(\mathcal{H})$ denote the corresponding multi core graph.

Every morphism $\eta\colon \Gamma\to \Delta$ of multi core graphs has also a description in terms of subgroups à la Definition~\ref{def_mocc}:
Assume that $\Gamma$ consists of $\ell$ components $\Gamma_1, \ldots, \Gamma_{\ell}$ and that $\Delta$ consists of $m$ components $\Delta_1, \ldots, \Delta_m$.
Let $f\colon [\ell]\to [m]$ be the induced map on connected components, so $\eta(\Gamma_i)\subseteq \Delta_{f(i)}$.
For every $i\in [\ell]$, pick an arbitrary vertex $v_i\in \Gamma_i$ and let $H_i\defeq \pilab(\Gamma_i, v_i)$.
As $\eta$ is an immersion, it induces injective maps at the level of fundamental groups: indeed, any non-backtracking cycle in $\Gamma$ is mapped to a non-backtracking cycle in $\Delta$. 
Therefore, $\eta$ can be thought of as the embedding, for all $i\in [\ell]$,
\begin{equation}
\label{equation_def_morphism}
    H_i\hookrightarrow \pilab\prn*{\Delta_{f(i)}, \eta(v_i)}.
\end{equation}

We still need to conjugate the images in equation \eqref{equation_def_morphism} so that they all sit in the same subgroups in the conjugacy class of subgroups of $\Delta_j$.
Formally, pick an arbitrary vertex $p_k\in \Delta_k$ for all $k\in [m]$ and
let $J_k\defeq \pilab(\Delta_k, p_k)$.
For every $i\in [\ell]$, let $u_i\in F_r$ satisfy
$u_i\brackets*{ \pilab\prn*{ \Delta_{f(i)}, \eta(v_i) } }u_i^{-1} = J_{f(i)}$.
So now $u_i H_i u_i^{-1}\le J_{f(i)}$, and we get a morphism as in Definition~\ref{def_mocc}.

Conversely, every embedding of subgroups $H\hookrightarrow J$ of $F_r$ gives rise to a morphism of core graphs $\Gamma_B(H)\to \Gamma_B(J)$ (see \cite{HP22}).
Thus any morphism in $\moccFr$ as in Definition~\ref{def_mocc}, gives rise to a morphism of the corresponding multi core graphs.
We say that two morphisms in $\moccFr$ are identical if they induce the same morphism of multi core graphs, up to a post-composition by an isomorphism of the co-domain. 
This equivalence of morphisms in $\moccFr$ can also be defined in completely algebraic terms, and, in particular, it does not depend on the basis B.

\begin{definition}
\cite[Definition 3.3]{HP22}
Let $\Gamma\in \mucgBFr$ be a multi core graph and 
$\vec{\mathcal{H}}= \pilab(\Gamma)$ 
the corresponding object in $\moccFr$.
We denote by $\chi(\Gamma) = \chi(\mathcal{H}) \defeq \#V(\Gamma) - \#E(\Gamma)$ the Euler characteristic of $\Gamma$.
Note that as we excluded the trivial group, $\chi(\Gamma)\le 0$ for all $\Gamma\in \mucgBFr$.
\end{definition}

Note also that despite discarding the trivial group in the definition of $\moccFr$, the object $\emptyset$ is valid, and corresponds to the empty multi core graph $\Gamma_B(\emptyset)$.

\subsection*{Expectation of class function via morphisms}
\label{subection_groups_graphs_expectation}

From now on, we fix an ambient group $F_r$ with a basis $B = \{b_i\}_{i=1}^r$.

\begin{definition}[Multi-word, Cyclic Object, Cyclic Graph]
\label{def_cyclic_cat}
A \textbf{cyclic object} of $\moccFr$ is a multiset 
consisting only of cyclic subgroups.
Equivalently, a \textbf{cyclic graph} is a multi 
core graph with Euler characteristic 0. 
Note that if there exists a morphism 
$\Vec{\mathcal{H}}\to \Vec{\mathcal{J}}$ and 
$\Vec{\mathcal{J}}$ is cyclic, then 
$\Vec{\mathcal{H}}$ is cyclic as well 
(as every non-trivial subgroup of $\Z$
is isomorphic to $\Z$). 
The category of cyclic objects is a full sub-category 
of $\moccFr$.
Note also that a morphism of cyclic graphs is a 
covering map onto its image.

A \textbf{multi-word} is a multi-set $\vec{w}$ of conjugacy classes of words $w_1, \ldots, w_k\in F_r$.
Equivalently, a multi-word is the data of a cyclic object, together with the generators $w_i$ of the groups $\inner{w_i}$, thus making $w, w^{-1}$ inequivalent (even though they generate the same cyclic group). 
Given a $\moccFr$-object $\Vec{\mathcal{H}}$, a morphism $\eta\colon \Vec{w}\to \Vec{\mathcal{H}}$ is just an ordinary $\moccFr$-morphism $\eta\colon \{\inner{w_1}, \ldots, \inner{w_k}\}\to \Vec{\mathcal{H}}$, i.e.\ it conjugates the words $w_i$ into the subgroups in $\Vec{\mathcal{H}}$, but the image $\eta(\Vec{w})$ is a multi-word.
In the language of graphs, $\Vec{w}$ \enquote{remembers} the orientation of every cycle $\Gamma_B(w_i)$ (clockwise or anti-clockwise).
Given a multi core graph $\Delta$, a morphism $\eta\colon \Gamma_B(\Vec{w})\to \Delta$ is just an ordinary morphism of multi core graphs, but $\eta(\Vec{w})$ is a multi-set of directed paths $\eta(w_i)$.
\end{definition}

We abbreviate notation and denote a multi word by 
$\braces*{w_1, \ldots, w_k}$ instead 
of\footnote{Note that repetition is allowed.} 
$\braces*{w_1^{F_r}, \ldots, w_k^{F_r}}$.

\begin{definition}[$\EX_{\eta}\brackets*{\zeta}$]
\label{def_ex_eta_zeta}
Let $\Vec{w} = \braces*{w_1, \ldots, w_k}$ be a multi-word, $\Vec{\mathcal{J}}\in \moccFr$ 
and let $\eta \in \Hom(\Vec{w}, \Vec{\mathcal{J}})$ be a choice of representatives $w_1, \ldots, w_k$ of $\Vec{w}$ and $J_1, \ldots, J_{\ell}$ of $\Vec{\mathcal{J}}$, together with a function $f\colon [k]\to [\ell]$ satisfying $w_j\in J_{f(j)}$ for every $j\in [k]$.
Also let $G$ be a finite group, and $\zeta\colon \Vec{w} \to \FGconjG$ a function assigning each word $w_j$ a class function $\zeta_j\colon G\to \FG$. Define
\[ \EX_{\eta}[\zeta] \defeq \prod_{i=1}^{\ell} \EX_{\alpha_i\sim U(\Hom(J_i, G))} \brackets*{\prod_{j\in f^{-1}(i)} \zeta_j(\alpha_i(w_j))}. \]
It is well defined, i.e.\ does not depend on the representatives of the conjugacy classes, since $\zeta_j$ are class functions.
Note that $\EX_{\eta}[\zeta]$ is multi-linear in $\zeta$, as a function $\EX_{\eta}\colon (\FGconjG)^k\to \FG$ (that is, for every $j\in [k]$, it is linear in $\zeta_j$).
\end{definition}

This quantity also has a geometric interpretation, 
which generalizes \cite[Proposition 2.9]{Sho23I}.
Here $\Vec{E}(\Delta)$ denotes the set of oriented edges
of a graph $\Delta$: every edge of $\Delta$ occurs in it twice, once with each orientation.

\begin{proposition}
\label{prop_geometric_interpretation_of_E_eta}
In the notation above, denote $\Delta \defeq \Gamma_B(\Vec{\mathcal{J}})$. 
Consider the set of anti-symmetric functions $\beta\colon \Vec{E}(\Delta)\to G$, that is $\beta(u\to v) = \beta(v\to u)^{-1}$, and draw a uniformly random element $\beta$.
Then
\[ \EX_{\eta}[\zeta] 
= \EX_{\beta} \brackets*{\prod_{j=1}^k 
\zeta_j\prn*{\prod_{e\in \Vec{E}(\Gamma_B(w_j))}
\beta(\eta_*(e))}}, \]
where $\eta_*(e)\in \Vec{E}(\Delta)$ is the image of $e$ under the graph-morphism $\eta_*$, and the product is over the edges in the path $w_j$ in the order they appear in $w_j$.
(This is again defined only up to conjugation so it is important that $\zeta_j$ are class functions).
\end{proposition}

\begin{proof}
By restricting the morphism $\eta$ to pre-images of connected components of $\Delta$ we get independent morphisms, thus the expectation is multiplicative with respect to the connected components of $\Delta$ so we may assume that $\Delta$ is connected, with $\pi_1(\Delta) = J\le F_r$ (for some basepoint of $\Delta$, which corresponds to some group $J$ from the conjugacy class). 
We can fix a basis of $J$ by choosing a spanning tree $T\subseteq E(\Delta)$ and orientation for the rest of the edges. 
Regardless of the values $\{\beta(e)\}_{e\in T}$ (which we can assume are all $1$), the random variables $\{\beta(e)\}_{e\not\in T}$ are independent and uniform in $G$, thus defining a random homomorphism $\alpha \colon J\to G$. 
We finish the proof by observing that the $w_j$-path in $\Delta$ is evaluated as $\alpha(w_j)$ when we write $w_j$ in the basis corresponding to $T$.
\end{proof}

\begin{remark}
\label{remark_2complex}
In \cite[Definition 2.6]{louder2021uniform}, Louder and Wilton defined a 2-complex by what we call a morphism (in $\mucgBFr$) from a cyclic object. 
\end{remark}

\section{Morphisms: Free, Algebraic, and Möbius Inversions}
\label{section_morphisms_free_alg}

We start by citing some definitions and theorems 
from \cite{HP22}. 

\begin{definition}[Free morphism] 
    \cite[Definition 4.2]{HP22}
If $H_1, \ldots, H_{\ell}$ are subgroups of the free group $J$, we say that $J$ is a free extension of the multiset $\braces*{H_1, \ldots, H_{\ell}}$, denoted $\braces*{H_1, \ldots, H_{\ell}}\leff J$, if $J$ decomposes as a free product 
\[ J = \prn*{\underset{i=1}{\overset{\ell}{\mathop{\bigast}}} j_i H_i j_i^{-1}} * K\]
for some conjugate subgroups $j_i H_i j_i^{-1}$ of $H_i$ (so $j_i\in J$) and some subgroup $K\le J$.

Now let $\eta\colon \Gamma\to \Delta$ be a morphism of multi core graphs with $\Delta $ connected. One can pick an arbitrary subgroup $J$ in the single conjugacy class in $\pilab(\Delta)$, and for every component $\Gamma_1, \ldots, \Gamma_{\ell}$ of $\Gamma$, a suitable subgroup $H_i$ so that $H_i\le J$.
Say that $\eta$ is a free morphism if $\{H_1, \ldots, H_{\ell}\}\leff J$. Finally, say that a general morphism $\eta\colon \Gamma\to \Delta$ of multi core graphs is free if $\eta_{\eta^{-1}(\Delta')}\colon \eta^{-1}(\Delta')\to \Delta'$ is free for every connected component $\Delta'$ of $\Delta$.
\end{definition}

\begin{definition}
    A morphism in $\moccFr$ is called $B$-surjective if the corresponding morphism in the equivalent category $\mucgBFr$ is surjective (on both vertices and edges).
\end{definition}
We will use the term \enquote{$B$-surjective} also for morphisms in $\mucgBFr$ interchangeably. 
Recall $\EX_{\eta}[\zeta]$ from 
Definition~\ref{def_ex_eta_zeta}, 
where $\zeta\colon \Vec{w}\to \FGconjG$ 
associates each cycle a class function.

\begin{proposition}
[Generalization of {\cite[Proposition 4.3(4)]{HP22}}]
Let $\Vec{w} \overset{\eta_1}{\twoheadrightarrow} \Vec{\mathcal{H}} \overset{\eta_2}{\twoheadrightarrow} \Vec{\mathcal{J}} $
be $B$-surjective morphisms, where $\Vec{w}$ is a multi-word. 
Let $G$ be a finite group and let $\zeta\colon \Vec{w}\to \FGconjG$. 
If $\eta_2$ is free then 
\[ \EX_{\eta_1}[\zeta] = \EX_{\eta_2\circ \eta_1}[\zeta]. \]
\end{proposition}

\begin{proof}
We may assume without loss of generality that $\Vec{\mathcal{J}}$ is a single conjugacy class with a representative $J$, since both sides of the equation are multiplicative with respect to the different conjugacy classes in $\Vec{\mathcal{J}}$.
Let $\beta\sim U(\Hom(J, G))$. Let $\braces*{w_1, \ldots, w_k}$ and $\braces*{H_1, \ldots, H_{\ell}}$ be the representatives of $\vec{w}$ and $\Vec{\mathcal{H}}$ given by $\eta_1$ and $\eta_2$ respectively. 
Since the free product of $H_i$ is a free factor of $J$, the homomorphisms $\beta_i \defeq \beta\restriction_{H_i}$ are uniform and independent. Thus
\[\EX_{\eta_1}[\zeta] 
= \prod_{j=1}^{\ell} \EX\brackets*{\prod_{i\in \eta_1^{-1}(j) }\zeta_i(\beta_j(w_i))}
= \EX\brackets*{\prod_{j=1}^{\ell} \prod_{i\in \eta_1^{-1}(j) }\zeta_i(\beta(w_i))} 
=  \EX_{\eta_2\circ \eta_1}[\zeta]. \]
\end{proof}

\begin{definition}[Algebraic Morphism]
    \cite[Definition 4.6]{HP22}
Let $\eta\colon \Gamma\to \Delta$ be a morphism of multi core graphs. We say that $\eta$ is \textbf{algebraic} if whenever $\Gamma\overset{\eta_1}{\to} \Sigma \overset{\eta_2}{\to} \Delta$ is a decomposition of $\eta$ with $\eta_2$ free, we have that $\eta_2$ is an isomorphism.
\end{definition}

\begin{proposition}
\cite[Theorem 4.7]{HP22}
Some properties of algebraic morphisms.
\begin{enumerate}
    \item Every algebraic morphism of multi core graphs is surjective.
    \item The identity morphism is algebraic, and so is the composition of two algebraic morphisms.
\end{enumerate}
\end{proposition}

\begin{theorem}[Algebraic-Free Decomposition]
\label{thm_AFD}
\cite[Theorem 4.9]{HP22}
Let $\eta\colon \Gamma\to \Delta$ be a morphism of multi core graphs. Then there is a decomposition
$\Gamma \overset{\phi}{\underset{\textrm{algebraic}}{\longrightarrow}} \Sigma \overset{\psi}{\underset{\textrm{free}}{\longrightarrow}} \Delta$
with $\eta = \psi\circ \phi$ such that $\phi$ is algebraic and $\psi$ is free. 
This decomposition is unique in the following sense: if there is another decomposition 
$\Gamma \overset{\phi'}{\underset{\textrm{algebraic}}{\longrightarrow}} \Sigma' \overset{\psi'}{\underset{\textrm{free}}{\longrightarrow}} \Delta$
where $\phi'$ is algebraic and $\psi'$ is free, then there is an isomorphism $\iota$ 
which makes the following diagram commute:
\[\begin{tikzcd}
	\Gamma & \Sigma \\
	{\Sigma'} & \Delta
	\arrow["\phi", from=1-1, to=1-2]
	\arrow["\psi", from=1-2, to=2-2]
	\arrow["{\phi'}"', from=1-1, to=2-1]
	\arrow["{\psi'}"', from=2-1, to=2-2]
	\arrow["\iota"{description}, dashed, from=1-2, to=2-1]
\end{tikzcd}\]
\end{theorem}

\subsection*{Möbius inversions}
\label{subsection_mobius_inversions}

\begin{definition}[$B$-surjective Decomposition]
\label{def_decompos_B}
\cite[Definition 6.4]{HP22}
Let $\eta\in Hom_{\mucgBFr}(\Gamma, \Delta)$ be a $B$-surjective morphism. Define
\[ \DecompB(\eta) \defeq \{(\eta_1, \eta_2):
\Gamma \overset{\eta_1}{\twoheadrightarrow} \textup{Im}(\eta_1) \overset{\eta_2}{\twoheadrightarrow} \Delta, \,\,\, \eta = \eta_2\circ \eta_1 \}\]
modulo the following equivalence relation: $(\eta_1, \eta_2) \sim (\eta_1', \eta_2') $ whenever there is an isomorphism $\theta \colon \textup{Im}(\eta_1) \to \textup{Im}(\eta_1')$ such that the diagram
\[\begin{tikzcd}
	\Gamma & \textup{Im}(\eta_1) \\
	& {\textup{Im}(\eta_1')} & \Delta
	\arrow["{\eta_1}", two heads, from=1-1, to=1-2]
	\arrow["{\eta_1'}"', two heads, from=1-1, to=2-2]
	\arrow["\cong", from=1-2, to=2-2]
	\arrow["{\eta_2}", two heads, from=1-2, to=2-3]
	\arrow["{\eta_2'}", two heads, from=2-2, to=2-3]
\end{tikzcd}\]
commutes. Similarly, let $\DecompB^3(\eta)$ denote the set of decompositions $\Gamma\overset{\eta_1}{\twoheadrightarrow}\Sigma_1\overset{\eta_2}{\twoheadrightarrow}\Sigma_2\overset{\eta_3}{\twoheadrightarrow}\Delta$ of $\eta$ into three surjective morphisms. Again, two such decompositions are considered equivalent (and therefore the same element in $\DecompB^3(\eta)$) if there are isomorphisms $\Sigma_i\cong \Sigma_i', i = 1, 2$, which commute with the decompositions.\\
The set $\DecompB(\eta)$ is endowed with a partial order: $(\eta_1, \eta_2)\le (\eta_1', \eta_2') $ whenever there is a morphism $\textup{Im}(\eta_1)\twoheadrightarrow \textup{Im}(\eta_1')$ such that the diagram commutes: see \cite[Remark 6.7]{HP22}.
\end{definition}

\begin{definition}[$\textup{sur}_B, \mathcal{Q}_B(\Gamma)$]
\label{def_surB_QB}
We define $\textup{sur}_B$ to be the collection of all $B$-surjective morphisms in the category $\mucgBFr$, 
where we declare two morphisms $(\Gamma \overset{\eta}{\twoheadrightarrow} \Delta), (\Gamma' \overset{\eta'}{\twoheadrightarrow} \Delta')$ as equivalent if there are isomorphisms $\Gamma \overset{\phi}{\to} \Gamma', \Delta \overset{\psi}{\to} \Delta'$
such that the following diagram commutes:
\[\begin{tikzcd}
	\Gamma & {\Gamma'} \\
	\Delta & {\Delta'}
	\arrow["\phi", from=1-1, to=1-2]
	\arrow["{\eta'}", from=1-2, to=2-2]
	\arrow["\eta"', from=1-1, to=2-1]
	\arrow["\psi"', from=2-1, to=2-2]
\end{tikzcd}\]
Moreover, given an object $\Gamma$, define $\mathcal{Q}_B(\Gamma)$ as the set of (classes of) morphisms in $\textup{sur}_B$ which originates in $\Gamma$. 
These are called $\Gamma$\textbf{-quotients}. 
For a multi-word $\Vec{w}$, define $\mathcal{Q}_B(\Vec{w})$ similarly, and recall that a morphism $\eta\colon \Vec{w}\to \Vec{\mathcal{H}}$ remembers the orientation of $\vec{w}$ (so it does not depend only on the generated subgroups $\inner*{w_i}$). 
\end{definition}

\begin{definition}[$B$-Convolution]
Let $f, g \colon \textup{sur}_B\to \FG$.
We define their $B$-convolution as follows: for every $B$-surjective morphism $\eta \colon \mathcal{H}\to \mathcal{J}$,
\[(f \underset{B}{*} g)(\eta) \defeq \sum_{(\eta_1, \eta_2)\in \DecompB(\eta)} f(\eta_1) g(\eta_2). \]
This operation, together with the obvious vector space structure, makes $\FG^{\textup{sur}_B}$ into an associative convolution algebra.
The identity element is the function $\delta_B$ that gives $1$ to isomorphisms and $0$ to any other morphism.
\end{definition}

The following definition relies on the local finiteness 
of the poset $\textup{sur}_B$. 

\begin{definition}[Möbius Inversion]
Let $\textbf{1}\in \FG^{\textup{sur}_B}$ be the constant $1$ function. It is invertible in $\FG^{\textup{sur}_B}$, and we denote its inverse by $\mu_B$, called the \textbf{Möbius Inversion} in the basis $B$. 
\end{definition}

\begin{definition}
\cite[Definition 3.4]{HP22}\footnote{In \cite{HP22} it was denoted by $\Phi_{\eta}(n)$.}
Let $n\in \N$. We define a function
$\EX\brackets*{S_n\acts [n]} \colon \textup{sur}_B\to \FG$, and denote its value at a morphism
$\eta\colon \mathcal{H}\to \mathcal{J}$ by $\EX_{\eta}\brackets*{S_n\acts [n]}$:
\[ \EX_{\eta}\brackets*{S_n\acts [n]} \defeq \prod_{J\in \mathcal{J}} \EX_{\alpha\sim U(\Hom(J, \,S_n))} \brackets*{\prod_{H\in \eta^{-1}(J)}\#\{\textup{common fixed points of }\alpha(H)\le S_n\}}, \]    
where the product is over representatives of the 
conjugacy classes 
$J\in \mathcal{J}$.
\end{definition}

\begin{definition}[Möbius Derivations of $\EX\brackets*{S_n\acts [n]}$]
\label{def_LB}
\cite[Definition 6.6]{HP22}
We define the \textbf{L}eft, \textbf{R}ight and \textbf{C}entral Möbius derivations of $\EX\brackets*{S_n\acts [n]}$ by
\begin{align*}
    L^B(n) & \defeq \mu_B \underset{B}{*} \mathbb{E}\brackets*{S_n\curvearrowright[n]}, && \textup{i.e. } \mathbb{E}\brackets*{S_n\curvearrowright[n]} = \sum_{(\eta_1, \eta_2)\in \DecompB(\eta)} L_{\eta_2}^B(n).\\
    R^B(n) &\defeq \mathbb{E}\brackets*{S_n\curvearrowright[n]} \underset{B}{*} \mu_B, && \textup{i.e. } \mathbb{E}\brackets*{S_n\curvearrowright[n]} = \sum_{(\eta_1, \eta_2)\in \DecompB(\eta)} R_{\eta_1}^B(n).\\
    C^B(n) &\defeq \mu_B \underset{B}{*} \mathbb{E}\brackets*{S_n\curvearrowright[n]} \underset{B}{*} \mu_B, && \textup{i.e. } \mathbb{E}\brackets*{S_n\curvearrowright[n]} 
    = \sum_{(\eta_1, \eta_2, \eta_3)\in \DecompB^3(\eta)} C_{\eta_2}^B(n).
\end{align*}
Abbreviating notation via $\EX\defeq \mathbb{E}\brackets*{S_n\curvearrowright[n]}$, 
this can be summarized in the following diagram:
\[\begin{tikzcd}
	& {{ \mathbb{E}}} \\
	{{ L^B(n) \defeq \mu_B \underset{B}{*} \mathbb{E}}} && 
 {{ R^B(n) \defeq \mathbb{E} \underset{B}{*} \mu_B}} \\
	& {{C^B(n) \defeq \mu_B \underset{B}{*} \mathbb{E} \underset{B}{*} \mu_B }}
	\arrow[from=1-2, to=2-3]
	\arrow[from=1-2, to=2-1]
	\arrow[from=2-1, to=3-2]
	\arrow[from=2-3, to=3-2]
\end{tikzcd}\]
\end{definition}

\begin{proposition}[\enquote{Basis dependent Möbius inversions}]
\label{prop_LB}
\cite[Proposition 6.8]{HP22}
Let $\eta\colon \Gamma\to \Delta$ be a $B$-surjective morphism in $\mucgBFr$. 
Then for every $n\ge \abs*{E(\Gamma)}$:
\[ L_{\eta}^B(n) = \frac{\prod_{v\in V(\Delta)}(n)_{|\eta^{-1}(v)|}}
{\prod_{e\in E(\Delta)}(n)_{|\eta^{-1}(e)|}} = n^{\chi(\Gamma)} \cdot (1 + O\prn*{n^{-1}}). \]
(We use the notation $(n)_t\defeq n\cdot (n-1)\cdots (n-t+1)$). 
Moreover, it has the following properties:
\begin{itemize}
    \item $L_{\eta}^B(n)$ is equal to the average number of injective lifts from $\Gamma$ to a random $n$-cover of $\Delta$.
    \item $L_{\eta}^B(n)$ is multiplicative with respect to the connected components of $\textup{Im}(\eta)$.
    \item If there is only one letter then $L_{\eta}^B(n)=1$.
\end{itemize}
\end{proposition}

\begin{definition}[\enquote{Algebraic Möbius Inversion}]
\label{def_mobius_alg}
\cite[Definition 6.13]{HP22}
For an algebraic morphism $\eta\colon \Gamma\to \Delta$ in $\mucgBFr$, or equivalently in $\moccFr$, denote by $\Decompalg(\eta)$ and $\Decompalg^3(\eta)$ the set of decompositions of $\eta$ into two (three, respectively) algebraic morphisms, with the same identifications as in Definition~\ref{def_decompos_B}.
In analogy with the $B$-convolution, the \textbf{algebraic convolution} of two functions $f, g$ on the algebraic morphisms is
\[ (f\convalg g)(\eta) \defeq \sum_{(\eta_1, \eta_2)\in \Decompalg(\eta)} f(\eta_1)\, g(\eta_2), \]
and $\mualg$ denotes the inverse of the constant function $\textbf{1}$ with respect to $\convalg$.
We also define algebraic left, right and central Möbius inversions of $\EX[S_n\acts [n]]$, denoted $L_{\eta}^{\alg}(n), R_{\eta}^{\alg}(n), C_{\eta}^{\alg}(n)$, respectively. For instance, $L^{\alg}$ is defined by
\[ \EX_{\eta}[S_n\acts [n]] = \sum_{(\eta_1, \eta_2)\in \Decompalg(\eta)} L_{\eta_2}^{\alg}(n). \]
If $\eta$ is algebraic, then $\Decompalg(\eta)\subseteq \DecompB(\eta)$, by \cite[Theorem 4.7]{HP22}.
\end{definition}

\section{The Induction-Convolution Lemma}
\label{section_ICL}

The induction-convolution lemma provides a method 
to compute the word measure of an induced character, 
under certain independence assumptions on the group 
and subgroup used in the induction process which 
are satisfied in the case of wreath products.
Recall Definition~\ref{def_ex_eta_zeta} (expectation of class functions via a morphism) and Definition~\ref{def_ind_lambda} (induction of a multi partition). 
Let $G$ be a finite group, let $\vec{w} = \{w_1, \ldots, w_k\}$ be a multi-word, and $\zeta\colon \{w_1, \ldots ,w_k\}\to \FGconjG$ a class function on every cycle. 
We denote by $\Ind\zeta$ the coordinate-wise induction of the corresponding multi partition: for every $n\in \N$,
$\Ind\zeta\colon \{w_1, \ldots, w_k\}\to \FG^{\conj\prn*{G\wr S_n}}$ is defined by $(\Ind\zeta)(w_i) \defeq \Ind(\zeta(w_i))$.
Our objective is to express 
$\EX_{\eta}[\Ind\zeta]$ 
in terms of formulas that do not depend on the wreath product $G_n = G\wr S_n$, but instead only involve either $G$ or $S_n$ separately.
The induction-convolution lemma gives such a disentanglement, in the form of convolution in the lattice of $B$-surjective morphisms.

\noindent Recall that $\Ind\textbf{1} = \#\textup{fixed points}$ is the permutation character of $S_n\acts [n]$.
Thus it is natural to denote for every morphism $\eta\colon \vec{\mathcal{H}}\to \vec{\mathcal{J}}$:
\[\EX_{\eta}[\Ind\textbf{1}]\defeq \EX_{\eta}[S_n\acts [n]].\]
This notation is less intuitive, but it emphasizes the idea of the induction-convolution lemma:

\begin{lemma}
    [The Induction-Convolution Lemma]
\label{lemma_ICL}
Let $B\subseteq F_r$ be a basis, and $\vec{w} = \{w_1, \ldots, w_k\}$ a multi-word. 
Then
\[ \EX[\Ind\zeta] = \EX[\zeta] \underset{B}{*} \mu_B \underset{B}{*} \EX[\Ind\textbf{1}]\]
where $\EX[\zeta], \EX[\Ind\zeta]$ are elements of $\FG^{\mathcal{Q}_B(\vec{w})}$, i.e.\ map morphisms $\eta$ to their $\zeta$-expectations
$\EX_{\eta}[\zeta], \EX_{\eta}[\Ind\zeta]$.
\end{lemma}
Note that the convolution is well defined even though $\EX[\zeta], \EX[\Ind\zeta]$ are defined only for morphisms $\eta\in \FG^{\mathcal{Q}_B(\vec{w})}$. 
In contrast, $\mu_B$ and $\EX[\Ind\textbf{1}]$ are defined on every morphism in $\textup{sur}_B$, and this is crucial because the convolution operator involves many morphisms that appear in $\DecompB(\eta)$.  
An alternative formulation 
is that for every morphism $\eta\in \mathcal{Q}_B(\vec{w})$,
\begin{equation}
\label{ICL_withL}
    \EX_{\eta}[\Ind\zeta] = \sum_{(\eta_1, \eta_2)\in \DecompB(\eta)} \EX_{\eta_1}[\zeta]\cdot L_{\eta_2}^B(n).
\end{equation}
Indeed, $\mu_B \Bconv \EX[\Ind\textbf{1}] = L^B(n)$ by Definition~\ref{def_LB}, so $\EX[\Ind\zeta] = \EX[\zeta] \Bconv L^B(n)$, and evaluating this equality at $\eta$ gives \eqref{ICL_withL}, by the definition of the $B$-convolution.

\noindent Actually, the induction-convolution lemma is applicable to every wreath product $G\wr_{[n]} S$ where $G, S$ are compact groups and $S\acts [n]$ is transitive, but we do not need this in this paper.\\

Weaker versions of the induction-convolution lemma were proven before.
The simplest case 
is when $\zeta_j = \textbf{1}$ for every $j\in [k]$: then the statement of the lemma is merely the definition of the Möbius inversion $\mu_B$ as the inverse of the constant function $\textbf{1}$ in the convolution algebra $\FG^{\textup{sur}_B}$. 
One can think of this lemma as a way of lifting results from $S_n$ to $G\wr S_n$ when $G$ is not trivial.

The first non-trivial version of the induction-convolution lemma was proven in \cite[Subsection 3.1]{magee2021surface}:
For $m\in \Z_{\ge 0}$ let 
$C_m\defeq \{z\in \Ss^1: z^m = 1\}$
(so $C_0 = \Ss^1$), and let $\phi\colon C_m\hookrightarrow \Ss^1$ be the inclusion map.
Recall from Equation~\eqref{equation_expectation_of_cyclic_embedding} that $\EX_w[\phi]$ 
is the indicator of the following event: 
every letter in $w$ appears $m\Z$ times in total, 
counted with signs.
For a subgroup $H\le F_r$, compose the abelianization map $H\twoheadrightarrow H^{\textup{ab}}$ with the \enquote{modulo $m$} map to define
\begin{equation}
    \label{eq_def_KmH}
    K_m(H)\defeq \ker\prn*{H\twoheadrightarrow H^{\textup{ab}}/mH^{\textup{ab}}}.
\end{equation}
By \cite{magee2021surface}, for every $w\in F_r$, $\EX_{w\to H}[\phi] = \bbone_{w\in K_m(H)}$.
Now $\Indphi$ is the character of the standard faithful representation $C_m\wr_{[n]} S_n \le \textup{GL}_n(\C)$.
The main result of \cite{magee2021surface} is the approximation of the 
$w$-expectation of $\Indphi \colon C_m\wr_{[n]} S_n\to \C$.
They showed that for every $w\in H\le F_r$ 
and any basis $B\subseteq F_r$ 
such that the associated core-graph morphism 
$\eta\colon \Gamma_B(w)\to \Gamma_B(H)$ 
is surjective, $w\in K_m(H)$ 
if and only if the directed path 
$\eta(w)$ traverses each edge 
$e\in E(\Gamma_B(H))$ a total number of $m\Z$ times, where each crossing is counted as $\pm1$ according to the direction.
Then the corresponding version of the induction-convolution lemma is
\[ \EX_w[\Indphi] = \sum_{(\eta_1, \eta_2)\in \DecompB(\eta)} \bbone_{w\in K_m\prn*{\pilab(\textup{Im}(\eta_1))}} L^B_{\eta_2}(n), \]
where $\eta\colon \Gamma_B(w)\to \Omega_B$ is the morphism to the bouquet.
The main result of \cite{magee2021surface} follows immediately.

For the next special case of the induction-convolution lemma, let $\vec{w} = \braces*{w_1, \ldots, w_k}$ be a multi-word, and let $\zeta\colon \vec{w}\to \widehat{C_m}$ be $\zeta_j = \phi$ for every $1\le j\le k$.
For a general morphism $\eta\colon \vec{w}\to \braces*{J_1, \ldots, J_{\ell}}$ we have 
\begin{equation*}
\EX_{\eta}[\zeta] =
    \begin{cases}
    1 & \textrm{if for every } i\in [\ell]:\,\,\, \prn*{\prod_{j\in \eta^{-1}(i)} w_j} \in K_m(J_i),\\
    0 & \textrm{otherwise.}
    \end{cases}
\end{equation*}
Equivalently, $\EX_{\eta}[\zeta]$ is the indicator of the event that for each edge $e$ in the graph $\Gamma_B(\textup{Im}(\eta))$, the total (signed) number of times that $\eta_{*}(w_j)$-paths cross $e$ is $0$ mod $m$. 
If this happens we say that $\eta$ is $m$\textbf{-balanced}.
The induction-convolution lemma then reads 
\begin{equation*}
    \EX_{\eta}[\Ind\zeta] = \sum_{\substack{(\eta_1, \eta_2)\in \DecompB(\eta):\\ \eta_1 \textrm{ is }m\textrm{-balanced}}} L_{\eta_2}^B(n).
\end{equation*}
This was proven in \cite[Claim 3.5]{Ord20}, but only for $m=2$. 
Finally, the induction-convolution lemma was proven for every compact group $G$ in \cite[Lemma 3.5]{Sho23I}, but only for $k=1$ (a single word $w$).

\subsection*{Algebraic induction-convolution lemma}

We will soon state an algebraic version of the same lemma, but first we have to understand the connection between $L^{\alg}$ and $L^B$.
In \cite[Proposition 6.15]{HP22}, Puder and Hanany proved the following identity for $C^{\alg}$ and $C^B$; the same proof gives it for $L^{\alg}$ and $L^B$, and so does part (1) of Lemma~\ref{lemma_mobius_alg_vs_B} below with $f = \EX\brackets*{\Ind\textbf{1}}$. For every algebraic morphism $\eta$ and every basis $B\subseteq F_r$,
\begin{equation}
\label{eq_Lalg_vs_LB}
    L^{\alg}_{\eta} = \sum_{\substack{(\eta_1, \eta_2)\in \DecompB(\eta):\\ \eta_1 \textrm{ is free}}} L^B_{\eta_2}.
\end{equation}

\begin{corollary}
\label{corollary_L_alg_approx}
[Generalization of {\cite[Corollary 3.7]{Sho23I}}]
For every basis $B\subseteq F_r$, algebraic morphism 
$\eta\colon \Gamma\to \Omega_B$ and 
$n\ge \abs*{E(\Gamma)}$, $L^{\alg}_{\eta}(n)$ coincides with a rational function in $\Q(n)$ and 
\[ L^{\alg}_{\eta}(n) = n^{\chi(\Gamma)} \cdot \prn*{1 + O\prn*{n^{-1}}}. \]
\end{corollary}

The short proof is the same as in 
\cite[Corollary 3.7]{Sho23I}; 
we give it here for completeness:

\begin{proof}
The rationality follows from 
Proposition~\ref{prop_LB} and 
Equation~\eqref{eq_Lalg_vs_LB}. Moreover,
\begin{equation*}
    \begin{split}
    (\textup{Equation}~\eqref{eq_Lalg_vs_LB})\quad\quad
    L^{\alg}_{\eta} (n)
    &= \sum_{\substack{(\eta_1, \eta_2)\in \DecompB(\eta):\\ \eta_1 \textrm{ is free}}} L^B_{\eta_2}(n)\\
    (\textup{Proposition}~\ref{prop_LB})\quad\quad\quad\quad\quad &= 
    \sum_{\substack{(\eta_1, \eta_2)\in \DecompB(\eta):\\ \eta_1 \textrm{ is free}}} n^{\chi(\textup{Im}(\eta_1))} \cdot \prn*{1 + O\prn*{n^{-1}}}\\
    &= n^{\chi(\Gamma)} \cdot \prn*{1 + O\prn*{n^{-1}}}
    \end{split}
\end{equation*}
where the last step is obtained by splitting the summation to the case where $\eta_1$ is an isomorphism, which contributes the dominant term, and every free morphism that is not an isomorphism, which must decrease Euler characteristic.
\end{proof}

However, we prove a slightly more general version of this connection.

\begin{definition}
\label{def_free_inv_func}
(Free-Invariant Function)
A function $g\in \C^{\textup{sur}_B}$ is called \textbf{free-invariant} if it is invariant with respect to post-composition with a free morphism, i.e.\ if 
$\eta_{0\to 2} = \eta_{1\to 2}\circ \eta_{0\to 1}$
and $\eta_{1\to 2}$ is free, 
as in the following diagram
\[\begin{tikzcd}
	{\Gamma_0} & {\Gamma_1} & {\Gamma_2}
	\arrow[from=1-1, to=1-2]
	\arrow["free", from=1-2, to=1-3]
\end{tikzcd}\]
then $g(\eta_{0\to 1}) = g(\eta_{0\to 2})$.\\
It is a known fact that if a function depends only on the distribution of words under random homomorphisms to finite groups, then it is free-invariant. 
Indeed, if $F=H*J$ is a free decomposition, and $G$ is any group, then there is a natural identification
$\Hom(F, G) \cong \Hom(H, G) \times \Hom(J, G) $.
In particular if $G$ is finite and $\alpha\sim U(\Hom(F, G))$ is a random homomorphism, then $\alpha\restriction_H\sim U(\Hom(H, G))$ has uniform distribution.
\end{definition}

The connection between the $B$-surjective category 
and the algebraic category 
is described in the following straight-forward 
generalization of both \cite[Proposition 6.15]{HP22} and \cite[Lemma 3.9]{Sho23I}:

\begin{lemma}[$\mualg$ v.s. $\mu_B$]
\label{lemma_mobius_alg_vs_B}
Let $f\in \C^{\textup{sur}_B}$. Then for every algebraic morphism $\eta$,
\begin{enumerate}
    \item 
    \[ (\mualg \convalg f)(\eta) = \sum_{\substack{(\eta_1, \eta_2)\in \textup{Decomp}_{B}(\eta):\\ \eta_1 \textrm{ is free}}} (\mu_B \Bconv f)(\eta_2). \]
    In particular, since the operators $\Bconv, \convalg$ share the same identity element, we have
    \[ \mualg(\eta) = \sum_{\substack{(\eta_1, \eta_2)\in \textup{Decomp}_{B}(\eta):\\ \eta_1 \textrm{ is free}}} \mu_B (\eta_2).\]
    \item For every free-invariant $g\in \C^{\textup{sur}_B}$, 
    \[ g \Bconv \mu_B \Bconv f = g \convalg \mualg \convalg f. \]
\end{enumerate}
\end{lemma}

\begin{proof}
    Although this lemma is a generalization of 
    \cite[Lemma 3.9]{Sho23I}, the proof in 
    \cite{Sho23I} was intentionally formulated in 
    the language of morphisms so that it could be 
    applied without any modifications to the 
    generalized lemma as well. 
\end{proof}

\begin{corollary}[Algebraic Induction-Convolution Lemma]
\label{corollary_AICL}
Let $\Vec{w}$ 
be a multi-word, $\Vec{\mathcal{J}}\in \moccFr$, and $\eta \in \Hom(\Vec{w}, \Vec{\mathcal{J}})$ an algebraic morphism. 
Let $G$ be a finite group, and $\zeta\colon \Vec{w}\to \FGconjG$.
Then
\[ \EX_{\eta}[\Ind\zeta] 
= \sum_{(\eta_1, \eta_2)\in \Decompalg(\eta)} \EX_{\eta_1}[\zeta] \cdot L_{\eta_2}^{\alg}(n). \]
\end{corollary}

\begin{proof}
By Lemma~\ref{lemma_ICL}, $ \EX[\Ind\zeta] = \EX[\zeta] \underset{B}{*} \mu_B \underset{B}{*} \EX[\Ind\textbf{1}]$.
By part (2) of Lemma~\ref{lemma_mobius_alg_vs_B}, and since $\EX[\zeta] \in \C^{\mathcal{Q}_B(\vec{w})}$ is free-invariant,
$ \EX[\Ind\zeta] = \EX[\zeta] \convalg \mualg \convalg \EX[\Ind\textbf{1}]$, as needed.
\end{proof}

A special case of the algebraic induction-convolution lemma, where $G$ is cyclic and $\vec{w} = \{w\}$ consists of a single word, was proven in \cite[equation (3.7)]{magee2021surface}, where $L^{\alg}$ is called \enquote{contrib}.

\section{Proof of the Main Result}
\label{section_thm_main_result}

In this section we prove our main result, Theorem~\ref{thm_main_result}. 
Recall that the third part of Theorem~\ref{thm_main_result}, which approximates $\EX_w[\chi]$  up to an $O\prn*{n^{-1}}$-error in the case where $w$ is a proper power, is a special case of Corollary~\ref{corollary_approx_power}, which relies only on the second part of Theorem~\ref{thm_main_result} that handles non-power words, and on \cite{Sho23I}.
Therefore we need to prove only the first and the second parts.

We first summarize the proof, and then prove everything in detail.
Both parts are proved for the basis $\braces*{\s\IndVecLambda}_{\VecLambda\in \PGhat}$ of $\mathcal{A}(G)$ from Corollary~\ref{corollary_sInd_is_basis}: the first part as it stands, and the second part in the equivalent form \eqref{equation_equivalent_thm_main_result} below.
Let $\eta_{\VecLambda}\colon w^{\VecLambda}\to \Omega_B$ be the morphism from the multi-word $w^{\VecLambda}$ (Definition~\ref{def_w_VecLambda}) to the bouquet, which we may assume to be algebraic.
By Observation~\ref{observe_s} and the algebraic induction-convolution lemma (Corollary~\ref{corollary_AICL}),
\[ \EX_w\brackets*{\s\IndVecLambda} = \sum_{(\eta_1, \eta_2)\in \Decompalg(\eta_{\VecLambda})} \EX_{\eta_1}\brackets*{\VecLambda} \cdot L_{\eta_2}^{\alg}(n), \]
where $\EX_{\eta_1}\brackets*{\VecLambda}\in \FG$ does not depend on $n$.
The first part of Theorem~\ref{thm_main_result} follows by Corollary~\ref{corollary_L_alg_approx} and Proposition~\ref{prop_alg_no_id_from_words_is_at_least_2_to_1_on_edges}.
For the second part we refine this sum, by Definition~\ref{def_mobius_alg}, into
\[ \EX_w\brackets*{\s\IndVecLambda} = \sum_{(\eta_1, \eta_2, \eta_3)\in \Decompalg^3(\eta_{\VecLambda})} \EX_{\eta_1}\brackets*{\VecLambda} \cdot C_{\eta_2}^{\alg}(n), \]
and split this sum according to $\chi\prn*{\textup{Im}(\eta_2)}$.
The terms with $\chi\prn*{\textup{Im}(\eta_2)} = 0$ sum to $\inner*{\s\IndVecLambda, \textbf{1}}$, by \cite[Lemma 6.17]{HP22} and the inner product formula (Corollary~\ref{corollary_inner_product_formula}), and the terms with $\chi\prn*{\textup{Im}(\eta_2)} < 0$ where $\eta_2$ is not an isomorphism, or with $\chi\prn*{\textup{Im}(\eta_2)} \le -\pi(w)$, are $O\prn*{n^{-\pi(w)}}$ by \cite[Corollary 6.16]{HP22}.

In the remaining terms $\eta_2$ is an isomorphism, so $\chi\prn*{\textup{Im}(\eta_1)} = \chi\prn*{\textup{Im}(\eta_2)} = 1 - \pi(w)$ and $C_{\eta_2}^{\alg}(n) = n^{1-\pi(w)}$.
The main step is Proposition~\ref{prop_key}, which computes the sum of $\EX_{\eta_1}\brackets*{\VecLambda}$ over these decompositions:
\[ \sum_{\substack{(\eta_1, \eta_2)\in \Decompalg(\eta_{\VecLambda}):\\ \chi(\textup{Im}(\eta_1)) = 1 - \pi(w)}} \EX_{\eta_1}[\VecLambda]
= \sum_{\substack{\phi\in \hat{G}:\\ \pi_{\phi}(w) = \pi(w)}} \inner*{\s\IndVecLambda, \chi_\phi} \cdot \mathscr{C}_{\phi}(w). \]
Its proof has two ingredients.
The first is the classification of the algebraic morphisms from $w^{\VecLambda}$ with $\chi\prn*{\textup{Im}} = 1 - \pi(w)$, which we cite from \cite[Propositions 7.7 and 7.8]{HP22} in Theorem~\ref{thm_classification}: each of them is a cyclic quotient $\eta_0\colon w^{\VecLambda}\twoheadrightarrow w^{\vec{\mu}}$ with one marked cycle $w^1$, followed by a critical morphism $w\to H$ on that cycle.
Consequently, the left-hand side is a sum over the triples $(\eta_0, \vec{\tau}, H)$, where $\vec{\tau}\le \VecLambda$ is the pre-image of the marked cycle and $H\in \textup{Crit}(w)$.
The second ingredient is the computation of the inner products $\inner*{\s\IndVecLambda, \chi_{\phi}}$ in Proposition~\ref{prop_FrobRecip}, by which, for every $\phi\in \hat{G}$,
\[ \sum_{\emptyset\neq \vec{\tau}\le \VecLambda} \inner*{\s\Ind\prn*{\VecLambda \setminus \Vec{\tau}}, \textbf{1}}\cdot \inner*{\s\prn*{\vec{\tau}}, \phi} = \inner*{\s\IndVecLambda, \chi_{\phi}} \]
(Corollary~\ref{corollary_nonempty_irr}).
Finally, by Proposition~\ref{prop_approx_lin_dim_char}, a consequence of Theorem~\ref{thm_Sho23I_main}, and since $\Cphipi(w) = 0$ whenever $\pi_{\phi}(w) > \pi(w)$ (Equation~\eqref{eq_def_crit_phi_pi}),
\[ n^{1-\pi(w)} \sum_{\substack{\phi\in \hat{G}:\\ \pi_{\phi}(w) = \pi(w)}} \inner*{\s\IndVecLambda, \chi_\phi} \cdot \mathscr{C}_{\phi}(w)
= \sum_{\phi\in \hat{G}} \inner*{\s\IndVecLambda, \chi_{\phi}} \EX_w\brackets*{\chi_{\phi}} + O\prn*{n^{-\pi(w)}}, \]
which is the second term in Theorem~\ref{thm_main_result_unified}.
This ends the summary, and the rest of this section is the proof itself.

Recall from Definition~\ref{def_Ind_phi} that we denote the irreducible stable characters of $G\wr S_{\bullet}$ of degree $1$ by
$ \braces*{\chi_{\phi}}_{\phi\in \hat{G}}$, where $\chi_{\phi} = \chi_{\binom{(1)}{\phi}[\bullet]} = \Indphi - \bbone_{\phi = \textbf{1}}$ has dimension $\phi(1)\cdot n - \bbone_{\phi = \textbf{1}}$.

\begin{definition}
\label{def_w_object}
Let $w\in F_r$ be a non-power.
For every partition $\lambda = (\lambda_1, \ldots, \lambda_k) \in \mathscr{P}$, we define a multi-word $w^{\lambda} \defeq \prn*{w^{\lambda_1}, \ldots, w^{\lambda_k}}$.
We call such a multi-word a $w$\textbf{-object}.
Equivalently, a cyclic multi core graph $\Gamma$ with orientation on every cycle
is called a $w$\textbf{-graph} if every connected component (which is a cycle) is the cycle $w^{\lambda_i}$ for some $\lambda_i\in \N$.
Note that a cyclic quotient graph of a $w$-graph is also a $w$-graph (as explained in Definition~\ref{def_cyclic_cat}).
Given a $w$-graph $\Gamma$, its $w$\textbf{-roots} are all the vertices $v\in V(\Gamma)$ which indicate the starting point (which is also an end point) of a $w$-path.
By definition, the number of $w$-roots in a connected component $\Gamma_j$ with $\pilab(\Gamma_j) = \inner*{w^{\lambda_i}}$ is $\lambda_i$.
\end{definition}

\begin{example}
Let $\lambda = (\lambda_1, \ldots, \lambda_k)\in \Z^k$. Then $w^{\lambda} \defeq \prn*{w^{\lambda_1}, \ldots, w^{\lambda_k}}$ is a $w$-object if and only if $\lambda_i > 0$ for every $i\in [k]$ (as a free group has no torsion).
In $F_1 = \textup{Free}(\{b\})$, $(b^4, b^6)$ is a $b$-object in $F_1 = \inner*{b}$, but is not a $b^2$-object in $\inner*{b}$, since $b^2$ is a proper power. 
It is a $b^2$-object in $\inner*{b^2}$, however, as $b^2$ is a non-power in $\inner*{b^2}$.
\end{example}

\begin{definition}[$w^{\VecLambda}$]
\label{def_w_VecLambda}
Let $w\in F_r$ be a non-power and let $\VecLambda\in \mathscr{P}(\FGconjG)$ (see Definition~\ref{def_multi_partitions}).
We define $w^{\VecLambda}$ to be the $w$-object $w^{\floor{\VecLambda}}$ together with the function 
$\zeta_{\VecLambda}\colon w^{\floor{\VecLambda}}\to \FGconjG$ that gives every cycle $p\in w^{\floor{\VecLambda}}$ the corresponding class function $\zeta_{\VecLambda}(p)$.
For every $\Vec{\mathcal{J}}\in \moccFr$ and every $\eta\in \Hom\prn*{w^{\VecLambda}, \Vec{\mathcal{J}}}$ we denote $\EX_{\eta}[\VecLambda]\defeq \EX_{\eta}[\zeta_{\VecLambda}]$ (see Definition~\ref{def_ex_eta_zeta}).
\end{definition}

\begin{observation}
\label{observe_s}
Let $w\in F_r$ be a non-power, $B\subseteq F_r$ a basis, $\VecLambda\in \mathscr{P}(\FGconjG)$ a multi partition, and let $p\colon \Gamma_B\prn*{w^{\VecLambda}} \twoheadrightarrow \Gamma_B(w)$ be the unique covering map. 
Let $\Gamma$ be a multi core graph, $\eta_0\colon \Gamma_B(w)\to \Gamma$ a $B$-surjective morphism, and let $\eta \defeq \eta_0\circ p$
(see the diagram below).
Then
\[ \EX_{\eta}[\VecLambda] = \EX_{\eta_0}[\s(\VecLambda)].\]
Indeed, $\Gamma$ is connected, being the image of the connected graph $\Gamma_B(w)$.
Let $J$ be a representative of $\pilab(\Gamma)$ containing the image of $\inner*{w}$; by Definition~\ref{def_ex_eta_zeta} the choice of the representative does not matter.
Every cycle $w^m$ of $w^{\VecLambda}$ is mapped by $\eta$ into $J$, and $\alpha\prn*{w^m} = \alpha(w)^m$ for every $\alpha\in \Hom(J, G)$, so for $\alpha\sim U(\Hom(J, G))$,
\[ \EX_{\eta}[\VecLambda]
\overset{\textup{Definition}~\ref{def_ex_eta_zeta}}{=} \EX_{\alpha}\brackets*{\prod_{\phi:\,\, \VecLambda(\phi)\neq \emptyset} \prod_i \phi\prn*{\alpha(w)^{\VecLambda(\phi)_i}}}
\overset{\textup{Definition}~\ref{def_summary_s}}{=} \EX_{\alpha}\brackets*{\s(\VecLambda)\prn*{\alpha(w)}}
\overset{\textup{Definition}~\ref{def_ex_eta_zeta}}{=} \EX_{\eta_0}[\s(\VecLambda)]. \]
If $w$ is a proper power, the map $p$ is not unique anymore, but the result is still valid (for all covering maps $p$).
\[\begin{tikzcd}
	{\Gamma_B\prn*{w^{\VecLambda}}} \\
	\Gamma_B(w) & \Gamma
	\arrow["p"', two heads, from=1-1, to=2-1]
	\arrow["{\eta_0}"', from=2-1, to=2-2]
	\arrow["\eta", from=1-1, to=2-2]
\end{tikzcd}\]
\end{observation}

The following proposition generalizes 
\cite[Lemma 4.1]{Expansion_Puder_2015},
but the proof is the same.
\begin{proposition}
\label{prop_alg_no_id_from_words_is_at_least_2_to_1_on_edges}
    Let $\Gamma$ and $\Delta$ be non-isomorphic multi core graphs, where $\Gamma$ is cyclic (i.e.\ $\chi(\Gamma) = 0$), and let $\eta\colon \Gamma\to \Delta$ be an algebraic morphism. Then for every edge $e\in \Delta$, $\abs*{\eta^{-1}(e)} \ge 2$.
\end{proposition}

Now we prove the first part of Theorem~\ref{thm_main_result}:
\begin{theorem}
Let $w\in F_r$ be a word, $G$ a finite group, 
and $f\in \mathcal{A}(G)$ a stable class function.
Then $\EX_w[f]$ coincides with a rational function in $\FG(n)$ for every $n\ge \deg(f)\cdot \abs{w}$.
Moreover, if $\EX_w[f]\neq 0$, then
\[\EX_w[f] = \Omega\prn*{n^{-\deg(f)\cdot \abs{w}}} = \Omega\prn*{\dim(f)^{-\abs{w}}},\]
where $\dim(f) \defeq \max_{\chi\in \widehat{G_{\bullet}}, \,\,\inner*{f, \chi}\neq 0} \dim(\chi)$ (and recall Corollary~\ref{corollary_dim_poly}).
\end{theorem}

\begin{proof}
The proof is very similar to the proof of \cite[Proposition 4.1]{Sho23I}.
Let $\VecLambda\in \PGhat$, fix a basis $B\subseteq F_r$, let 
$\eta\colon w^{\VecLambda}\to \Omega_B$ 
be the unique morphism to the bouquet and let $\zeta = \zeta_{\VecLambda}$.
By Observation~\ref{observe_s} and the induction-convolution lemma (Lemma~\ref{lemma_ICL}),
\[ \EX_w\brackets*{\s\IndVecLambda} = \EX_{\eta}[\Ind\zeta] 
= \sum_{(\eta_1, \eta_2)\in \DecompB(\eta)} \EX_{\eta_1}[\zeta] \cdot L_{\eta_2}^{B}(n). \]
For $b\in B$ denote by $c_b$ the number of $b$-edges in $\Gamma_B(w)$, so that $\sum_{b\in B} c_b = \abs*{E(\Gamma_B(w))}\le \abs{w}$.
Every $\textup{Im}(\eta_1)$ is a quotient of $w^{\VecLambda}$, hence has at most $\abs*{\abs*{\VecLambda}}\cdot c_b$ edges labeled $b$ and at most $\abs*{\abs*{\VecLambda}}\cdot \abs{w}$ vertices, so by Proposition~\ref{prop_LB}, for $n\ge \abs*{\abs*{\VecLambda}}\cdot \abs{w}$ the denominator of $L^B_{\eta_2}(n)$ divides $\prod_{b\in B}(n)_{\abs*{\abs*{\VecLambda}}\cdot c_b}$.
Since every $f\in \mathcal{A}(G)$ is a linear combination of such $\s\IndVecLambda$ with $\abs*{\abs*{\VecLambda}}\le \deg(f)$ (Corollary~\ref{corollary_sInd_is_basis}), the product $\EX_w[f]\cdot \prod_{b\in B}(n)_{\deg(f)\cdot c_b}$ is a polynomial, which proves the rationality.
If $\EX_w[f]\neq 0$ this polynomial is nonzero, so the degree of $\EX_w[f]$ is at least $-\deg(f)\sum_{b\in B} c_b\ge -\deg(f)\cdot \abs{w}$.
\end{proof}

It is only left to prove the second part of Theorem~\ref{thm_main_result}, regarding non-power words $w$.

\begin{definition}
Denote by $P_1\colon \mathcal{A}(G)\to \mathcal{A}(G)^{\le 1}$ the orthogonal projection. 
Since $\widehat{G\wr S_{\bullet}}$ is a filtered orthonormal basis for $\mathcal{A}(G)$, for every $f\in \mathcal{A}(G)$ we have
\[ P_1\prn*{f} = \sum_{\substack{\chi\in \widehat{G\wr S_{\bullet}}:\\ \deg(\chi)\le 1}} \inner*{f, \chi}\cdot \chi. \]
\end{definition}

Recall the definition $\mathscr{C}_{\phi}(w) \defeq \sum_{H\in \textup{Crit}_{\phi}(w)} \EX_{w\to H}[\phi]$ from Definition~\ref{def_witnesses_full} and $\Cphipi(w) \defeq \sum_{H\in \textup{Crit}(w)} \EX_{w\to H}[\phi]$ from Equation~\eqref{eq_def_crit_phi_pi}.
Our next goal is to finish the proof of Theorem~\ref{thm_main_result}.
The following proposition is an immediate 
consequence of the definitions and 
Theorem~\ref{thm_Sho23I_main}:
\begin{proposition}
\label{prop_approx_lin_dim_char}
For every $w\in F_r$ and $\phi\in \hat{G}$,
$ \EX_w[\chi_{\phi}] = \Cphipi(w)\cdot n^{1-\pi(w)} + O\prn*{ n^{-\pi(w)} }. $
\end{proposition}

Instead of tackling irreducible characters of degree $\ge 2$ directly, 
our objective is  the following formula: for every $f\in \mathcal{A}(G)$ and a non-power $w\in F_r$,
\begin{equation}
\label{equation_equivalent_thm_main_result}
   \begin{split}
       \EX_w[f]
       &= \EX_w\brackets*{P_1(f)} + O\prn*{n^{-\pi(w)}}\\
       (\textrm{Proposition}~\ref{prop_approx_lin_dim_char}) \quad\quad\quad\quad &\overset{}{=} \inner*{f, \textbf{1}} + \sum_{\phi\in \hat{G} } \inner*{f, \chi_{\phi}}\cdot  \Cphipi(w) n^{1-\pi(w)} + O\prn*{n^{-\pi(w)}}.\quad\quad\quad
   \end{split} 
\end{equation}

Recall the second part of Theorem~\ref{thm_main_result}: for every stable irreducible character $\chi$ of $G\wr S_{\bullet}$ of degree $\deg(\chi)\ge 2$ and a non-power $w\in F_r$, we have $\EX_w[\chi] = O\prn*{n^{-\pi(w)}}$. 
This is equivalent to Equation~\eqref{equation_equivalent_thm_main_result}, as $\braces*{\chi\in \widehat{G\wr S_{\bullet}}:\,\, \deg(\chi)\ge 2}$ is a basis for the linear subspace $\ker(P_1)$.

\subsection*{The inner product formula and Frobenius reciprocity}

Until the end of this section, we denote $G_n\defeq G\wr S_n$.

\begin{definition}
    Let $\lambda\vdash n$ be a partition. 
    Let $\sigma\in S_n$ be a permutation with cycle-structure $\lambda$.
    We denote $z_{\lambda}\defeq \abs*{\textup{Cntr}_{S_n}(\sigma)} \defeq \abs*{\braces*{\tau\in S_n: \,\, \sigma\tau=\tau\sigma}}$.
    Note that
    \[z_{\lambda} = \prod_{j=1}^{\infty} j^{\#_j(\lambda)} \#_j(\lambda)!,\]
    as we can permute cycles of the same length and then rotate each of them separately.
\end{definition}

Lemma~\ref{lemma_ICL} gives us a tool for computing inner products of stable functions.
In the following corollary, recall $\s\IndVecLambda$ from Definition~\ref{def_ind_lambda}, the $w$-objects $w^{\tau}$ from Definition~\ref{def_w_object}, and $w^{\VecLambda}$ and $\EX_{\eta}\brackets*{\VecLambda}$ from Definition~\ref{def_w_VecLambda}.
Recall also the set $\mathcal{Q}_B$ of quotients from Definition~\ref{def_surB_QB}, taken here with the basis $B = \{w\}$ of the free group $\inner*{w}$.
\begin{corollary}
\label{corollary_inner_product_formula}
Let $\VecLambda\in \PGhat$.
Then for every non-power word $w\in F_r$,
\[\inner*{\s\IndVecLambda, \textbf{1}} = \sum_{k=0}^n \sum_{\tau \vdash k} \frac{1}{z_{\tau}} 
\sum_{\eta\in \mathcal{Q}_{\{w\}}\prn*{w^{\VecLambda}}: \,\,\,\textup{Im}(\eta) = w^{\tau}} \EX_{\eta}\brackets*{\VecLambda}. \]
\end{corollary}

\begin{proof}
The following computation takes place in the ambient group $\inner*{w}$.
We have 
\[\inner*{\s\IndVecLambda, \textbf{1}} 
= \EX_{w\to w}[\s\IndVecLambda]
= \EX_{w^{\VecLambda} \twoheadrightarrow w}[\Ind \VecLambda]
= \EX_{\eta}[\Ind \VecLambda]\]
where $\eta\colon w^{\VecLambda} \twoheadrightarrow w$ is the unique covering map.
Now the induction-convolution lemma tells us
\[\EX_{\eta}[\Ind \VecLambda] 
= \sum_{(\eta_1, \eta_2)\in \textup{Decomp}_{\{w\}}(\eta)} \EX_{\eta_1}\brackets*{\VecLambda} \cdot L_{\eta_2}^{\{w\}}(n), \]
but $L_{\eta_2}^{\{w\}}(n) = 1$ for every morphism $\eta_2$ since there is only one letter in the basis $\{w\}$.
Moreover, every quotient $\textup{Im}(\eta_1)$ is cyclic so it is of the form $w^{\tau}$ for some $\tau\vdash k, \,\,k\le n$.
The division by $z_{\tau}$ comes from the fact that generally, $\DecompB(\eta)$ consists of equivalence classes: 
each graph $w^{\tau}$ has exactly $z_{\tau}$ automorphisms. Indeed, $z_{\tau}$ is the number of permutations that commute with $\tau$, which correspond to automorphism of the cycle graph of $\tau$.
\end{proof}

\begin{corollary}[Stabilization of Characters]
\label{corollary_stabilization}
For $n\ge \abs*{\abs*{\VecLambda}}$, the inner product 
\[ \inner*{\s\IndnVecLambda, \textbf{1}} = \inner*{\s\IndVecLambda, \textbf{1}} \]
stabilizes (i.e.\ does not depend on $n$ anymore).
\end{corollary}

\begin{proof}
In Corollary~\ref{corollary_inner_product_formula}, we sum over all $k\le n$ and all partitions $\tau\vdash k$ the expectations of (surjective) quotient maps. 
Every quotient is of size (i.e.\ number of vertices = number of edges) at most $\abs*{\abs*{\VecLambda}}$, so for $n\ge \abs*{\abs*{\VecLambda}}$ there are no new quotients.
Note that we must count every equivalence class in $\textup{sur}_B$ only once.
\end{proof}

A useful tool for our computation is Frobenius reciprocity, both the classical one, which expresses the adjunction between induction and restriction, and a version for $\s\IndVecLambda$.
We start with an example of a simple usage of Frobenius reciprocity.
Recall Corollary~\ref{corollary_a_tc_approx}:
Let $w\in F_r$ be a non-power. Then for every $t\ge 2,\,\, c\in \conj(G)$:
$\EX_w[a_{t, c}] = \frac{|c|}{t|G|} + O\prn*{n^{-\pi(w)}}. $

\begin{proof}
[Proof of Corollary~\ref{corollary_a_tc_approx} assuming Theorem~\ref{thm_main_result}]
For a uniformly random $(v, \sigma)\sim U\prn*{G\wr S_n}$, $\EX[a_{t, c}] = \frac{|c|}{t|G|}. $
Indeed, for every $\sigma\in S_n, v\in G^n$: $a_{t, c}(v, \sigma) = \frac{1}{t} \sum_{i_1, \ldots, i_t \in [n] \textup{ distinct}} \bbone_{\sigma(i_1)=i_2, \ldots, \sigma(i_t)=i_1} \bbone_{v(i_t)\cdots v(i_1)\in c}. $
Thus 
\begin{equation}
\label{equation_proved_a_tc_approx}
    \begin{split}
        \EX[a_{t, c}] 
        &= \frac{1}{t} \sum_{i_1, \ldots, i_t \in [n] \textrm{ distinct}} \PR_{\sigma}(\sigma(i_1)=i_2, \ldots, \sigma(i_t)=i_1)\cdot \PR_{v}(v(i_t)\cdots v(i_1)\in c)\\
        & = \frac{1}{t}\cdot (n)_t \cdot \frac{1}{(n)_t} \cdot \frac{|c|}{|G|} = \frac{|c|}{t|G|}.
    \end{split}
\end{equation}
where $(n)_t\defeq n\cdot (n-1)\cdots (n-t+1)$ is the falling factorial.
Now by Theorem~\ref{thm_main_result}, it is enough to check that $\inner*{a_{t, c}, \chi_{\phi}} = 0$ for every $\phi\in \hat{G}$. 
By Frobenius reciprocity,
\begin{equation*}
    \begin{split}
        \inner*{\Indphi, a_{t, c}}_{G_n} 
        &= \inner*{\phi, \textup{Res}_{G\times G_{n-1}} a_{t, c}}_{G\times G_{n-1}} \\
        &= \frac{1}{|G|^n (n-1)!} \sum_{g\in G} \sum_{v\in G^{n-1}} \sum_{\sigma\in S_{n-1}} \phi(g) \#\braces*{t\textrm{-cycles }s\textrm{ of }\sigma\textrm{ with }\prod_{i\in s}v(i)\in c}\\
        (\textup{since }t\ge 2)&\overset{}{=} \prn*{\frac{1}{|G|} \sum_{g\in G} \phi(g) } \prn*{\frac{1}{|G|^{n-1} (n-1)!}  \sum_{v\in G^{n-1}} \sum_{\sigma\in S_{n-1}}  \#\braces*{t\textrm{-cycles }s\textrm{ of }\sigma\textrm{ with }\prod_{i\in s}v(i)\in c}}\\
        &= \inner*{\phi, \mathbf{1}}_G \cdot \inner*{a_{t, c}, \mathbf{1}}_{G_{n-1}}
    \end{split}
\end{equation*}
So for $\phi\in \hat{G}-\{\textbf{1}\} $ we have $\inner*{\Indphi, a_{t, c}} = 0$ and for $\phi = \textbf{1}$ we have $\inner*{\Indphi, a_{t, c}} = \inner*{a_{t, c}, \textbf{1}}$ for $n$ large enough, so $\inner*{\chi_{\phi}, a_{t, c}} = 0$. This shows the desired implication.
\end{proof}

Recall from Proposition~\ref{prop_Indphi_m_is_stable} that for $\phi\in \hat{G},\,\, k\in \N$,
$ (\Indphi)^{(k)} = \sum_{t\divides k} t \sum_{c\in \conj(G)} \phi\prn*{c^{k/t}}\cdot a_{t, c}.$
Now we generalize the number $k$ to a 
multi-partition $\VecLambda$.

\begin{definition}
Given two Young diagrams $\lambda, \tau\in \mathscr{P}$, we define 
\[ \Hom(\lambda, \tau) \defeq \braces*{\eta\colon \lambda\to \tau\middle\vert\,\, \textup{ for every part } p\in \lambda,\, |\eta(p)|\,\divides\, |p|}. \]
For $\VecLambda\in \PGhat$ we define $\Hom(\VecLambda, \tau)\defeq \Hom(\floor{\VecLambda}, \tau)$.
\end{definition}

\begin{corollary}
\label{corollary_sIndFormula}
Let $\VecLambda\in \PGhat$. For every $p = \VecLambda(\phi)_i\in \VecLambda$ denote $\zeta_p = \phi$.
Then for every $v\in G^n$ and $\sigma\in S_n$ with cycle-structure $\tau\in \mathscr{P}_n$:
\[\s\IndVecLambda(v, \sigma)  
= \sum_{\eta\in \Hom(\VecLambda, \tau)} \prod_{p\in \VecLambda} |\eta(p)| \cdot \zeta_p \prn*{ \prn{\prod_{k\in \eta(p)} v(k)}^{\frac{|p|}{|\eta(p)|}} }. \]
\end{corollary}

\begin{proof}
Just substitute Proposition~\ref{prop_Indphi_m_is_stable} in the definition $\s\IndVecLambda(v, \sigma) 
        = \prod_{\phi} \prod_{i} \Indphi\prn*{(v,\sigma)^{\VecLambda(\phi)_i}}$ and expand brackets.
\end{proof}

\begin{definition}
Let $\VecLambda\in \PGhat$. A \textbf{multi-subpartition} $\Vec{\tau} \le \VecLambda$ is a function $\Vec{\tau} \colon \hat{G}\to \mathscr{P}$ such that for every $\phi\in \hat{G}$, $\Vec{\tau}(\phi) \le \VecLambda(\phi)$.
Parts of the same size are distinguished, so that the multi-subpartitions of $\VecLambda$ correspond to the subsets of the set of parts of $\VecLambda$.
The \textbf{difference} $\VecLambda \setminus \Vec{\tau}$ between two such multi-partitions is defined coordinate-wise using difference of multi-sets: $\prn*{\VecLambda \setminus \Vec{\tau}}(\phi) \defeq \VecLambda(\phi)\setminus \Vec{\tau}(\phi)$.
\end{definition}

In the following proposition, we denote the inner product in $G_n = G\wr S_n$ by $\inner*{\cdot, \cdot}_{G_n}$.
Recall $\Indnphi$ from Definition~\ref{def_Ind_phi}, the summary function $\s$ from Definition~\ref{def_summary_s}, and $\s\IndVecLambda$ from Definition~\ref{def_ind_lambda}, whose $n$-th entry we denote by $\s\IndnVecLambda$.
Note that $\s\prn*{\vec{\tau}}$ is a class function on $G\cong G_1$.

\begin{proposition}
\label{prop_FrobRecip}
For every $\phi\in \hat{G}$, $\VecLambda\in \mathscr{P}(\textup{char}(G))$ and $n\in \N$, we have
\[ \inner*{\s\IndnVecLambda, \Indnphi}_{G_n} = \sum_{\vec{\tau}\le \VecLambda} 
        \inner*{\s\Ind_{n-1}\prn*{\VecLambda \setminus \Vec{\tau}}, \textbf{1}}_{G_{n-1}}\cdot
        \inner*{\s\prn*{\vec{\tau}}, \phi}_{G_1}. \]
In particular, by Corollary~\ref{corollary_stabilization}, if $n\ge \abs*{\abs*{\VecLambda}} + 1$ the terms do not depend on $n$:
\[ \inner*{\s\IndVecLambda, \Indphi} = \sum_{\vec{\tau}\le \VecLambda} 
        \inner*{\s\Ind\prn*{\VecLambda \setminus \Vec{\tau}}, \textbf{1}}\cdot
        \inner*{\s\prn*{\vec{\tau}}, \phi}. \]
\end{proposition}

\begin{proof}
Embed $\iota \colon S_{n-1}\hookrightarrow S_n$ as the permutations which fix $n$. This gives rise to the embedding $G\times G_{n-1}\hookrightarrow G_n$, by which we induce characters.
Let $\sigma'\in S_{n-1}$ and denote $\sigma\defeq \iota(\sigma')$. (Equivalently, take $\sigma\in S_n$ with $\sigma(n)=n$ and denote 
$\sigma'\defeq \sigma\restriction_{[n-1]}$). 
Denote the cycle-structure of $\sigma'$ by $\tau'$.
Similarly let $v\in G^n$ and denote 
$v'\defeq v\restriction_{[n-1]}$. 
In Corollary~\ref{corollary_sIndFormula}, we can divide the sum according to the pre-image $\vec{\tau} = \eta^{-1}(n)$ to get
\begin{equation*}
    \begin{split}
        \s\IndVecLambda(v, \sigma) 
        &= \sum_{\vec{\tau}\le \VecLambda} \sum_{\eta \in \Hom(\VecLambda \setminus \Vec{\tau}, \tau')} \prod_{p\in \VecLambda \setminus \Vec{\tau}} |\eta(p)| \zeta_p \prn*{ \prn{ \prod_{k\in \eta(p)} v(k)}^{\frac{|p|}{|\eta(p)|}}} 
        \prod_{p\in \vec{\tau}} \zeta_p\prn*{v(n)^{|p|}} \\
        &= \sum_{\vec{\tau}\le \VecLambda} \s\Ind_{n-1}\prn*{\VecLambda \setminus \Vec{\tau}}(v', \sigma') \prod_{p\in \vec{\tau}} \zeta_p\prn*{v(n)^{|p|}}\\
        &= \sum_{\vec{\tau}\le \VecLambda} \s\Ind_{n-1}\prn*{\VecLambda \setminus \Vec{\tau}}(v', \sigma') \cdot \s\prn*{\vec{\tau}}(v(n)).
    \end{split}
\end{equation*}
(This is because for $p\in \vec{\tau}$ we have $|\eta(p)| = 1$).
By Frobenius reciprocity,
\begin{equation*}
    \begin{split}
        \inner*{\s\IndVecLambda, \Indphi} 
        &= \inner*{\textup{Res}_{G\times G_{n-1}} \s\IndVecLambda, \phi}_{G\times G_{n-1}} \\
        &= \EX_{\substack{v\sim U(G^n)\\ \sigma'\sim U(S_{n-1})}} [\Bar{\phi}(v(n))\cdot \s\IndVecLambda(v, \sigma)] \\
        &= \EX_{\substack{v\sim U(G^n)\\ \sigma'\sim U(S_{n-1})}} \brackets*{\Bar{\phi}(v(n)) \sum_{\vec{\tau}\le \VecLambda} \s\Ind_{n-1}\prn*{\VecLambda \setminus \Vec{\tau}}(v', \sigma') \cdot \s\prn*{\vec{\tau}}(v(n))} \\
        &= \sum_{\vec{\tau}\le \VecLambda} \EX_{\substack{v'\sim U(G^{n-1})\\ \sigma'\sim U(S_{n-1})}} \brackets*{ \s\Ind_{n-1}\prn*{\VecLambda \setminus \Vec{\tau}}(v', \sigma')} \cdot \EX_{g\sim U(G)} \brackets*{ \Bar{\phi}(g) \cdot \s\prn*{\vec{\tau}}(g)} \\
        &= \sum_{\vec{\tau}\le \VecLambda} 
        \inner*{\s\Ind_{n-1}\prn*{\VecLambda \setminus \Vec{\tau}}, \textbf{1}}_{G_{n-1}} \cdot
        \inner*{\s\prn*{\vec{\tau}}, \phi}_G.
    \end{split}
\end{equation*}
\end{proof}

\subsubsection*{Examples}

Proposition~\ref{prop_FrobRecip} may seem complicated and technical, so we give some usage examples.

\begin{example}
Consider $\phi = \textbf{1}$ and $\VecLambda = \binom{(1, 1)}{\textbf{1}}$.
Then $\Ind\textbf{1} = \#\textup{fixed points} = \std + \textbf{1} = \chi_{(n-1, 1)} + \chi_{(n)}$ 
and $\s\IndVecLambda$ decomposes as 
$\prn*{\Ind\textbf{1}}^2 = 2 + 3\cdot \std + \chi_{(n-2, 2)} + \chi_{(n-2, 1, 1)}$
so $\inner*{\prn*{\Ind\textbf{1}}^2, \Ind\textbf{1}} = 5$, and indeed there are $4$ sub-partitions of $\VecLambda$, which correspond to taking the empty partition (that contributes $\inner*{\Ind\textbf{1}^2, \textbf{1}} = 2$ to the sum), the first cycle ($\inner*{\Ind\textbf{1}, \textbf{1}} = 1$), the second (contributes $1$), or both ($\inner*{\textbf{1}, \textbf{1}} = 1$). Indeed $2 + 1 + 1 + 1 = 5$.
Moreover we see that the non-empty sub-partitions give the coefficient of the standard character ($3$).
This suits the result in Corollary~\ref{corollary_nonempty_irr}.
\end{example}

We continue to more complicated examples:

\begin{example}
For $G = C_2$ denote $\hat{C_2} = \{\textbf{1}, \textbf{-1}\}$.
For $\VecLambda = \prn*{\ontop{(1)}{\textbf{1}}; \ontop{(1)}{\textbf{-1}}} $, 
we have\\
$ \s\IndVecLambda 
= \Ind\textbf{1}\cdot \Ind(\textbf{-1}) 
= \chi_{\prn*{\ontop{(n-2, 1)}{\textbf{1}}; \ontop{(1)}{\textbf{-1}}}} 
+ 2\cdot \chi_{\prn*{\ontop{(n-1)}{\textbf{1}}; \ontop{(1)}{\textbf{-1}}}} 
$
and

\begin{table}[ht!]
\centering
\begin{tabular}{|| c | c | c | c | c ||} 
\hline
$\vec{\tau}$ 
& $\s\Ind\prn*{\VecLambda \setminus \Vec{\tau}}$
& $\inner*{\s\Ind\prn*{\VecLambda \setminus \Vec{\tau}}, \textbf{1}}$
& $\s\prn*{\vec{\tau}}$ 
& $\inner*{\s\prn*{\vec{\tau}}, \phi}$ \\ [0.5ex] 
\hline\hline
$\emptyset$
& $\Ind\textbf{1}\cdot \Ind(\textbf{-1}) $ 
& $0$
& $\textbf{1}$ 
& $\bbone_{\phi=\textbf{1}}$\\  [0.5ex] \hline

$\prn*{\ontop{(1)}{\textbf{1}}}$
& $\Ind(\textbf{-1}) $ 
& $0$
& $\textbf{1}$ 
& $\bbone_{\phi=\textbf{1}}$\\  [0.5ex]\hline

$\prn*{\ontop{(1)}{\textbf{-1}}}$
& $\Ind\textbf{1} $ 
& $1$
& $\textbf{-1}$ 
& $\bbone_{\phi=\textbf{-1}}$\\  [0.5ex]\hline

$\prn*{\ontop{(1)}{\textbf{1}}; \ontop{(1)}{\textbf{-1}}}$
& $ \textbf{1} $ 
& $1$
& $\textbf{-1}$ 
& $\bbone_{\phi=\textbf{-1}}$\\  [0.5ex]\hline

 \hline
\end{tabular}
\label{table:frobenius_reciprocity_example_1}
\end{table}
\FloatBarrier

\end{example}

\begin{example}
For $G = S_3$ denote $\hat{S}_3 = \{1, \sgn, \std\}$. For 
$\VecLambda = \prn*{\ontop{\emptyset}{\textbf{1}}; \ontop{\emptyset}{\sgn}; \ontop{(2)}{\std}}$,
the class function $\std^{(2)}$ decomposes into irreducible characters of $S_3$ as
$\std^{(2)} = \textbf{1} - \sgn + \std, $
and
\begin{table}[ht!]
\centering
\begin{tabular}{|| c | c | c | c | c ||} 
\hline
$\vec{\tau}$ 
& $\s\Ind\prn*{\VecLambda \setminus \Vec{\tau}}$
& $\inner*{\s\Ind\prn*{\VecLambda \setminus \Vec{\tau}}, \textbf{1}}$
& $\s\prn*{\vec{\tau}}$ 
& $\inner*{\s\prn*{\vec{\tau}}, \phi}$ \\ [0.5ex] 
\hline\hline
$\emptyset$
& $\textup{\Ind(std)}^{(2)} $ 
& $1$
& $\textbf{1}$ 
& $\bbone_{\phi=\textbf{1}}$\\  [1ex] \hline

$\prn*{\ontop{(2)}{\std}}$
& $\textbf{1} $ 
& $1$
& $\std^{(2)}$ 
& $\bbone_{\phi=\textbf{1}} - \bbone_{\phi=\sgn} + \bbone_{\phi=\std}$\\  [1ex]\hline

 \hline
\end{tabular}
\label{table:frobenius_reciprocity_example_4}
\end{table}
\FloatBarrier
\end{example}

\subsubsection*{Putting it all together}

\begin{corollary}
\label{corollary_nonempty_irr}
For every $\phi\in \hat{G}$, $\VecLambda\in \PGhat$ and $n\ge \abs*{\abs*{\VecLambda}} + 1$, we have
\[ \inner*{\s\IndVecLambda, \chi_{\phi}}_{G_n} = \sum_{\emptyset\neq \vec{\tau}\le \VecLambda}
        \inner*{\s\Ind\prn*{\VecLambda \setminus \Vec{\tau}}, \textbf{1}}_{G_{n-1}}\cdot
        \inner*{\s\prn*{\vec{\tau}}, \phi}_G. \]
\end{corollary}

\begin{proof}
Immediate from Proposition~\ref{prop_FrobRecip}, since $\s(\emptyset) = \textbf{1}$ so
$ \inner*{\s\prn*{\vec{\tau}}, \phi}_G = \bbone_{\phi=\textbf{1}}. $
\end{proof}

Recall that a morphism $w\to H$ of groups (and the corresponding morphism on core graphs) is said to be critical if it is not free and $\textup{rk}(H) = \pi(w)$ (in particular it is algebraic).

We cite the classification theorem from \cite[Propositions 7.7 and 7.8]{HP22}:
\begin{theorem}[Classification of Critical Morphisms]
\label{thm_classification}
Let $1\neq w\in F$ be a non power, and let $\lambda\vdash d$. The set of algebraic morphisms $\eta\colon w^{\lambda}\to \Gamma$ with $-\pi(w) < \chi(\Gamma) < 0$ 
is precisely the set of possible outcomes of the following procedure:
\begin{enumerate}
    \item Choose a cyclic quotient $w^{\mu}$ of $w^{\lambda}$ with a marked $w^1$ cycle. (Choices of different marked cycles in the same $w^{\mu}$ are considered different).
    \item Choose a critical morphism $w\to H$ of core graphs and compose it on the marked cycle, leaving the other cycles untouched.
\end{enumerate}
The image $\Gamma$ of the resulting morphism has $\chi(\Gamma) = 1 - \pi(w)$, and has a unique non-cyclic connected component.
\end{theorem}

Recall that 
$\eta_{\VecLambda}\colon w^{\VecLambda}\to \Omega_B$ 
is the unique morphism to the bouquet, which we may assume to be algebraic by the free-invariance of $\EX\brackets*{\IndVecLambda}$ (Definition~\ref{def_free_inv_func}).
The following proposition is the key for our main result.
In the proposition below, recall $\Decompalg$ from Definition~\ref{def_mobius_alg}, $\EX_{\eta}\brackets*{\VecLambda}$ from Definition~\ref{def_w_VecLambda}, $\pi_{\phi}(w)$ and $\mathscr{C}_{\phi}(w)$ from Definition~\ref{def_witnesses_full}, and $\chi_{\phi} = \Indphi - \bbone_{\phi = \textbf{1}}$ from the beginning of this section.

\begin{proposition}
\label{prop_key}
Let $w\in F_r$ be a non power, and $\VecLambda\in \PGhat$. Then
\begin{equation}
    \begin{split}
        \sum_{\substack{(\eta, \eta')\in \Decompalg(\eta_{\VecLambda}):\\ \chi(\textup{Im}(\eta)) = 1 - \pi(w)}} \EX_{\eta}[\VecLambda]
        &=
        \sum_{\substack{\phi\in \hat{G}:\\ \pi_{\phi}(w) = \pi(w)}} \inner*{\s\IndVecLambda, \chi_\phi} \cdot \mathscr{C}_{\phi}(w)
    \end{split}
\end{equation}
\end{proposition}

\begin{proof}
Fix an arbitrary basis $B\subseteq F_r$.
By the classification theorem (Theorem~\ref{thm_classification}), algebraic morphisms $\eta\colon \Gamma_B(w^{\VecLambda}) \to \Gamma$ with $\chi(\Gamma) = 1 - \pi(w)$ are exactly $\eta_1\circ \eta_0$ where $\eta_0$ is a surjective morphism of $w$-graphs with a marked $w^1$ cycle and $\eta_1$ is the identity on non-marked cycles and critical (of the form $w^1\to H$) on the marked cycle.
Denote by $w^{\vec{\tau}} \defeq \eta_0^{-1}(w^1)$ the cycles which $\eta_0$ sends to the marked cycle.
This is illustrated in the following diagram, 
where OOooo represents the graph $\Gamma_B\prn*{w^{(2, 2, 1, 1, 1)}}$ and 
$\otimes$ represents a quotient graph of $\Gamma_B(w)$. 
The cycles on which $\eta$ acts non-trivially are \textcolor{orange}{highlighted}: 

\[\begin{tikzcd}
	{\Gamma_B(w^{\VecLambda})=\textrm{O\textcolor{orange}{Oo}oo}} && {\Gamma_B(w^{\vec{\tau}})=\textrm{\textcolor{orange}{Oo}}} \\
	& {\Gamma_B(w^{\vec{\mu}})=\textrm{O\textcolor{orange}{o}oo}} & {\Gamma_B(w)=\textrm{\textcolor{orange}{o}}} \\
	{\Gamma=\textrm{O}\textcolor{orange}\otimes\textrm{oo}} && {\Gamma_B(H)=\textcolor{orange}{\otimes}}
	\arrow["{\eta_0}", two heads, from=1-1, to=2-2]
	\arrow["{\eta_1}", two heads, from=2-2, to=3-1]
	\arrow["\eta"', two heads, from=1-1, to=3-1]
	\arrow["{\eta_0\restriction_{\Gamma_B(w^{\vec{\tau}})}}"', two heads, from=1-3, to=2-3]
	\arrow["{\eta_1\restriction_{\Gamma_B(w)}}"', two heads, from=2-3, to=3-3]
	\arrow[dotted, hook', from=1-3, to=1-1]
	\arrow[dotted, hook', from=2-3, to=2-2]
	\arrow[dotted, hook', from=3-3, to=3-1]
\end{tikzcd}\]

By Observation~\ref{observe_s}, 
\begin{equation}
\label{equation_decompose_to_tau_and_rest}
    \begin{split}
        \EX_{\eta}\brackets*{\VecLambda} 
        &= \EX_{\eta_1\circ \eta_0}\brackets*{\VecLambda} \\
        &= \EX_{\eta_0 \restriction_{w^{\VecLambda \setminus \Vec{\tau}}}} 
        \brackets*{\VecLambda \setminus \Vec{\tau}} \cdot \EX_{\eta_1\circ \eta_0\restriction_{w^{\vec{\tau}}}} \brackets*{\vec{\tau}} \\
        (\textup{Observation }\ref{observe_s}) &= \EX_{\eta_0 \restriction_{w^{\VecLambda \setminus \Vec{\tau}}}}
        \brackets*{\VecLambda \setminus \Vec{\tau}} \cdot \EX_{w\to H}\brackets*{\s\prn*{\vec{\tau}}} \\
        &= \EX_{\eta_0 \restriction_{w^{\VecLambda \setminus \Vec{\tau}}}
        }\brackets*{\VecLambda \setminus \Vec{\tau}} \cdot \sum_{\phi\in \hat{G}} \inner*{\s\prn*{\vec{\tau}}, \phi} \EX_{w\to H}\brackets*{\phi}.  
    \end{split}
\end{equation}
Denote the set of cyclic quotients through which $\eta_{\VecLambda}$ factors by
\[ \textup{Quot}_0\prn*{\eta_{\VecLambda}} \defeq \braces*{\eta_0: \prn*{\eta_0, \eta'}\in \DecompB\prn*{\eta_{\VecLambda}}, \chi\prn*{\textup{Im}(\eta_0)} = 0}. \]
Given two morphisms $\eta_i\colon \Gamma_i\to \Delta_i, \, i\in \{1, 2\}$ we denote their disjoint union by $\eta_1\uplus \eta_2\colon \Gamma_1\uplus \Gamma_2 \to \Delta_1\uplus \Delta_2$.
Now we sum over all decompositions $(\eta_0, \eta_1)$ as above (and remember that $\vec{\tau}$, the pre-image of the marked $w^1$ cycle, should not be empty):
\allowdisplaybreaks
    \begin{align*}
        \sum_{\substack{(\eta, \eta')\in \Decompalg(\eta_{\VecLambda}):\\ \chi(\textup{Im}(\eta)) = 1 - \pi(w)}} \EX_{\eta}[\VecLambda] &= \\
        (\textup{Theorem }\ref{thm_classification}) \quad\quad &\overset{}{=} 
        \sum_{\substack{\eta_0\in \textup{Quot}_0(\eta_{\VecLambda}):\\ w^1\in \textup{Im}(\eta_0)}} \sum_{\substack{\vec{\tau}\le \VecLambda:\\ \eta_0\prn*{\vec{\tau}} = w^1}} \sum_{H\in \textup{Crit}(w)} \EX_{\prn*{ (w\to H)\uplus id_{w^{\VecLambda \setminus \Vec{\tau}}} }\circ \eta_0}[\VecLambda] \\
        (\textup{equation }\eqref{equation_decompose_to_tau_and_rest}) \quad\quad &\overset{}{=} 
        \sum_{\substack{\eta_0\in \textup{Quot}_0(\eta_{\VecLambda}):\\ w^1\in \textup{Im}(\eta_0)}} \sum_{\substack{\vec{\tau}\le \VecLambda:\\ \eta_0\prn*{\vec{\tau}} = w^1}} \sum_{H\in \textup{Crit}(w)} \EX_{\eta_0 \restriction_{w^{\VecLambda \setminus \Vec{\tau}}}
        }[\VecLambda \setminus \Vec{\tau}] \cdot \sum_{\phi\in \hat{G}} \inner*{\s\prn*{\vec{\tau}}, \phi} \EX_{w\to H}[\phi] \\
        (\textup{Corollary }\ref{corollary_inner_product_formula}) \quad\quad &\overset{}{=}
        \sum_{\substack{\emptyset \neq \vec{\tau}\le \VecLambda}} 
        \inner*{\s\Ind\prn*{\VecLambda \setminus \Vec{\tau}}, \textbf{1}}   
        \sum_{H\in \textup{Crit}(w)} \sum_{\phi\in \hat{G}} \inner*{\s\prn*{\vec{\tau}}, \phi} \EX_{w\to H}[\phi] \\
        (\textup{Corollary }\ref{corollary_nonempty_irr}) \quad\quad &\overset{}{=}   
        \sum_{\phi\in \hat{G}}  
        \inner*{\s\IndVecLambda, \chi_{\phi}} \sum_{H\in \textup{Crit}(w)} \EX_{w\to H}[\phi] \\
        (*) \quad\quad &\overset{}{=} 
        \sum_{\substack{\phi\in \hat{G}:\\ \pi_{\phi}(w) = \pi(w)}} 
        \inner*{\s\IndVecLambda, \chi_{\phi}} \sum_{H\in \textup{Crit}(w)} 
        \EX_{w\to H}[\phi]\\
        (\textup{Definition}~\eqref{eq_def_crit_phi_pi}) \quad\quad & \overset{}{=} 
        \sum_{\substack{\phi\in \hat{G}:\\ \pi_{\phi}(w) = \pi(w)}} 
        \inner*{\s\IndVecLambda, \chi_{\phi}} \mathscr{C}_{\phi}(w).
    \end{align*}
To justify the $(*)$ step, assume that $\phi\in \hat{G} $ has $\pi_{\phi}(w) > \pi(w) = \textup{rk}(H)$: then $H\not \in \textup{Wit}_{\phi}(w)$, i.e.\ $\EX_{w\to H}[\phi] = 0$.
\end{proof}

Now we prove formula~\eqref{equation_equivalent_thm_main_result} and finish the proof of the main result, Theorem~\ref{thm_main_result}.
Recall the formula~\eqref{equation_equivalent_thm_main_result}: For every non-power $w\in F_r$ and $\psi\in \mathcal{A}(G)$,
\begin{equation*}
    \EX_w[\psi] = \inner*{\psi, \mathbf{1}} + \sum_{\phi\in \hat{G}} \inner*{\psi, \chi_{\phi}} \EX_w[\chi_{\phi}] + O\prn*{n^{-\pi(w)}}.
\end{equation*}

The proof is similar to that of \cite[Theorem 7.2]{HP22}. Recall $\Decompalg^3$ from Definition~\ref{def_mobius_alg}.
\begin{proof}
It is sufficient to prove the formula for the basis $\braces*{\s\IndVecLambda}_{\VecLambda\in \PGhat}$, so let $\VecLambda\in \PGhat$ and denote $\psi \defeq \s\IndVecLambda$.
As above, we may assume that $\eta_{\VecLambda}$ is algebraic.
By Corollary~\ref{corollary_AICL},
\begin{equation}
\label{eq_C}
\begin{split}
    \EX_w[\psi]
    = \EX_{\eta_{\VecLambda}}[\Ind \VecLambda] 
    &= \sum_{(\eta_1, \eta_2)\in \Decompalg(\eta_{\VecLambda})} \EX_{\eta_1}\brackets*{\VecLambda} \cdot L_{\eta_2}^{\alg}(n) \\
    &= \sum_{(\eta_1, \eta_2, \eta_3)\in \Decompalg^3(\eta_{\VecLambda})} \EX_{\eta_1}\brackets*{\VecLambda} \cdot C_{\eta_2}^{\alg}(n).
\end{split}
\end{equation}
Now we split the summation into three parts: \\

\underline{\textbf{Case I:} $\chi(\textup{Im}(\eta_2)) = 0$}. 
By \cite[Lemma 6.17]{HP22}, 
$ C_{\eta_2}^{\alg}(n) = \bbone_{\eta_2 = id}$, 
and thus by Corollary~\ref{corollary_inner_product_formula}, 
\[ \sum_{\substack{(\eta_1, \eta_2, \eta_3)\in \Decompalg^3(\eta_{\VecLambda}):\\ \chi(\textup{Im}(\eta_2)) = 0}} \EX_{\eta_1}\brackets*{\VecLambda} \cdot C_{\eta_2}^{\alg}(n) = \inner*{\psi, \textbf{1}}.\] 

\underline{\textbf{Case II:} $\chi(\textup{Im}(\eta_2)) = 1-\pi(w)$ and $\eta_2$ is an isomorphism}. \\
By \cite[Corollary 6.16]{HP22}, 
$ C_{\eta_2}^{\alg}(n) = n^{\chi(\textup{Im}(\eta_1))} = n^{\chi(\textup{Im}(\eta_2))} = n^{1-\pi(w)}$.
For every $(\eta_1, \eta_3)\in \Decompalg(\eta_{\VecLambda})$ with $\chi(\textup{Im}(\eta_1)) = 1 - \pi(w)$ there is exactly one such decomposition, the one with $\eta_2 = \textup{id}$, and so the total contribution is 
\begin{equation*}
    \begin{split}
        \sum_{\substack{(\eta_1, \eta_3)\in \Decompalg(\eta_{\VecLambda}):\\ \chi(\textup{Im}(\eta_1)) = 1 - \pi(w)}} \EX_{\eta_1}\brackets*{\VecLambda} n^{1-\pi(w)} &=\\
        (\textup{Proposition}~\ref{prop_key}) \quad\quad\quad\quad &\overset{}{=}
        \sum_{\substack{\phi\in \hat{G}:\\ \pi_{\phi}(w) = \pi(w)}} \inner*{\s\IndVecLambda, \chi_\phi} \cdot \mathscr{C}_{\phi}(w) n^{1-\pi(w)} \\
        (\textup{Proposition}~\ref{prop_approx_lin_dim_char}) \quad\quad\quad\quad &\overset{}{=} \sum_{\phi\in \hat{G}} \inner*{\psi, \chi_{\phi}} \EX_w[\chi_{\phi}] + O\prn*{n^{-\pi(w)}}.
    \end{split}
\end{equation*}

\underline{\textbf{All remaining terms in } (\ref{eq_C})}: 
For each term, either $\chi(\textup{Im}(\eta_2)) \le -\pi(w)$ or $\chi(\textup{Im}(\eta_2)) = 1 - \pi(w)$ but $\eta_2$ is not an isomorphism, and \cite[Corollary 6.16]{HP22} yields that $C_{\eta_2}^{\alg}(n) = O\prn*{n^{-\pi(w)}}$.
\end{proof}

\section{Expansion of Random Schreier Graphs}
\label{section_expansion}

In this section we use our results about word measures to deduce Theorem~\ref{thm_expander_wreath}:
Let $G$ be a finite group, $(X_n)_{n=1}^{\infty}$ a sequence of sets, and $\prn*{\rho_n\colon G\wr S_n\to \textup{Sym}(X_n)}_{n=1}^{\infty}$ a rep-stable sequence of transitive actions of degree $k$. Let $r\ge 2$ and $\Gamma_n\defeq \textup{Sch}_r\prn*{\rho_n}$.
Then, a.a.s.\ (as $n\to \infty$), the largest absolute value $\mu_{\Gamma_n}$ of a non-trivial eigenvalue of the adjacency matrix of $\Gamma_n$ satisfies
\[ \mu_{\Gamma_n} \le 2\sqrt{2r - 1}\cdot \exp\prn*{ \frac{2k^2}{e^2(2r-1)} }. \]
In particular for every fixed $k, \varepsilon > 0$, if $r$ is large enough, $\mu_{\Gamma_n} \le 2\sqrt{2r - 1} + \varepsilon$ a.a.s.\

Until the end of this section, we denote for every stable class function $f\in \mathcal{A}(G)$ and $1\neq w\in F_r$,
\[ \mathscr{C}_w(f) \defeq \sum_{\substack{\phi\in \hat{G}:\\ \pi_{\phi}(w)=\pi(w)}} \inner*{f, \chi_\phi} \mathscr{C}_{\phi}(w). \]
As in \cite{HP22}, we denote by $\mathfrak{CR}_t({F}_r)$ the set of cyclically reduced words of length $t$ in $F_r$.

\begin{lemma}
    [{\cite[Lemma 8.5]{HP22}}]
\label{lemma_bound_coeff_L}
Let $\eta\colon \Gamma\to \Omega_B$ 
be the unique morphism from a multi core graph 
$\Gamma$ 
to the bouquet. 
Assume that $|V(\Gamma)|, |E(\Gamma)|\le T$.
Then the coefficients $b_p$ in the Laurent expansion 
$ L_{\eta}^B(n) = n^{\chi(\Gamma)} \sum_{p=0}^{\infty} b_p n^{-p} $ 
are bounded by $|b_p|\le T^{2p}$.
\end{lemma}

\begin{proposition}[Effective Bound on $\EX_w\brackets*{\s\IndVecLambda}$, generalizing {\cite[Lemma 8.6]{HP22}}]
\label{prop_effective_bound_sInd}
For a non-power $w\in \mathfrak{CR}_t(F_r)$ with $\pi(w) < \infty$, 
and 
$\VecLambda\in \PGhat$ 
that has 
$\ell(\VecLambda)$ 
parts overall, denote by 
$T\defeq t\cdot \abs*{\abs*{\VecLambda}}$ 
the number of vertices (and edges) in $w^{\VecLambda}$. 
Then $\forall n > T^2$,
\[ \abs*{ \EX_w\brackets*{\s\IndVecLambda} - \inner*{\s\IndVecLambda, \mathbf{1}} - n^{1-\pi(w)}\cdot \mathscr{C}_w\prn*{\s\IndVecLambda} } 
\le \frac{\s\VecLambda(1)\cdot T^{2(\pi(w) + \ell(\VecLambda))}}{n^{\pi(w) - 1}\cdot (n - T^2)}. \]
\end{proposition}

\begin{proof}
By part (1) of Theorem~\ref{thm_main_result}, 
$\EX_w\brackets*{\s\IndVecLambda}$ coincides with a rational function in $n$, and $\abs*{\EX_w\brackets*{\s\IndVecLambda}}\le \s\VecLambda(1)$, so this rational function has degree at most $0$ and a Laurent series
$ \EX_w\brackets*{\s\IndVecLambda} = \sum_{p = 0}^{\infty} a_p n^{-p}. $
By formula \eqref{equation_equivalent_thm_main_result} we know that 
\[a_0 = \inner*{\s\IndVecLambda, \mathbf{1}},\quad a_1 = \ldots = a_{\pi(w)-2} = 0, \quad a_{\pi(w)-1} = \mathscr{C}_w\prn*{\s\IndVecLambda}\] 
so the LHS we want to bound is $\sum_{p = \pi(w)}^{\infty} a_p n^{-p}$.
We will prove
$|a_p| \le \s\VecLambda(1) T^{2(\ell(\VecLambda) + p)},$
and the result follows by summation.
By the Induction convolution lemma,
\[ \EX_w\brackets*{\s\IndVecLambda} 
= \sum_{(\eta_1, \eta_2)\in 
\DecompB(w^{\VecLambda}\to \Omega_B)} 
\EX_{\eta_1}\brackets*{\VecLambda} 
L_{\eta_2}^B(n). \]
This means that if $L_{\eta}^B(n) = n^{\chi(\Gamma)} \sum_{p=0}^{\infty} b_p n^{-p}$ then
\[ a_p = \sum_{(\eta_1, \eta_2)\in \DecompB(w^{\VecLambda}\to \Omega_B)} \EX_{\eta_1}\brackets*{\VecLambda} b_{p + \chi(\textup{Im}(\eta_1))}. \]
By Lemma~\ref{lemma_bound_coeff_L}, and since every quotient graph of $w^{\VecLambda}$ has at most $T$ vertices and edges, we get the bound $|b_{p + \chi(\textup{Im}(\eta_1))}|\le T^{2(p + \chi(\textup{Im}(\eta_1)))}$.
The term $\EX_{\eta_1}\brackets*{\VecLambda}$ can be bounded by $\s\VecLambda(1) =  \prod_{\phi} \phi(1)^{\ell(\VecLambda(\phi))}$ (where $1$ is the identity element in $G$), as characters attain their maxima at $1$. 
Moreover, by \cite[Proposition 5.7 and the proof of Lemma 8.6]{HP22} we know that the number of decompositions $(\eta_1, \eta_2)\in \DecompB(w^{\VecLambda}\to \Omega_B)$ with $\chi(\textup{Im}(\eta_1)) = x$ is at most $\binom{T}{2}^{\ell(\VecLambda) - x}$, where $\ell(\VecLambda)$ is the number of parts in the partition $\floor{\VecLambda}$. 
We get the bound
\[ |a_p| 
\le \s\VecLambda(1) \sum_{x = -p}^0 \binom{T}{2}^{\ell(\VecLambda) - x} T ^ {2 (p + x)} 
\le \s\VecLambda(1) T^{2(\ell(\VecLambda) + p)} \sum_{x = -p}^0 2^{x - \ell(\VecLambda)} 
\le \s\VecLambda(1) T^{2(\ell(\VecLambda) + p)}, \]
by $\binom{T}{2}\le \frac{T^2}{2}$ and $\ell(\VecLambda)\ge 1$.
For $\VecLambda = \emptyset$ we have $\s\IndVecLambda = \textbf{1}$, so $a_0 = 1$ and $a_p = 0$ for every $p\ge 1$, and the bound holds too.
\end{proof}

\begin{corollary}
    [Generalization of {\cite[Proposition 8.4]{HP22}}]
\label{corollary_effective_bound_sInd_comb}
For every stable class function $f\in \mathcal{A}(G)$ of degree $D$, there is a constant $A\ge 1$ such that for every non-power $w$ of length $t$ with $\pi(w)<\infty$ and $n > \prn*{t\cdot D}^2$, 
\[ \abs*{ \EX_w[f] - \inner*{f, \mathbf{1}} - \mathscr{C}_w(f) n^{1-\pi(w)} } \le A \frac{\prn*{t\cdot D}^{2(\pi(w) + D)}}{n^{\pi(w) - 1}\cdot \prn*{n - \prn*{t\cdot D}^2}}. \]
\end{corollary}

\begin{proof}
By Corollary~\ref{corollary_sInd_is_basis}, every stable function $f$ is a linear combination $f = \sum_{\VecLambda\in \PleDGhat} \beta_{\VecLambda}\cdot  \s\IndVecLambda $
where the sum runs over multi-partitions $\VecLambda$ with $\abs*{\abs*{\VecLambda}}\le D$. 
Clearly every such $\VecLambda$ has at most $D$ parts. Let
$A\defeq \max\prn*{1, \sum_{\VecLambda\in \PleDGhat} |\beta_{\VecLambda}|\cdot \s\VecLambda(1) }. $
As the bounded expression in Proposition~\ref{prop_effective_bound_sInd} is linear in $\s\IndVecLambda$, a linear combination of the previous bounds gives:
\begin{equation*}
    \begin{split}
        \abs*{ \EX_w[f] - \inner*{f, \mathbf{1}} - \mathscr{C}_w(f) n^{1-\pi(w)} }
        &\le \sum_{\VecLambda\in \PleDGhat} |\beta_{\VecLambda}|\cdot  \frac{\s\VecLambda(1)\cdot \prn*{t\cdot \abs{\abs{\VecLambda}}}^{2(\pi(w) + \ell(\VecLambda))}}{n^{\pi(w) - 1}\cdot (n - \prn*{t\cdot \abs{\abs{\VecLambda}}}^2)}\\
        &\le A \frac{\prn*{t\cdot D}^{2(\pi(w) + D)}}{n^{\pi(w) - 1}\cdot \prn*{n - \prn*{t\cdot D}^2}}.
    \end{split}
\end{equation*}
\end{proof}

\begin{proposition}
\label{prop_bound_critical_value}
Let $f\in \mathcal{A}(G)$ be a stable class function. Then there exists $B > 0$ such that for every $w\in F_r-\{1\}$, $ \abs*{ \mathscr{C}_w(f) } \le B\cdot |\textup{Crit}(w)|. $
\end{proposition}

\begin{proof}
First, note that for every $\phi\in \hat{G}$ with $\pi_{\phi}(w)=\pi(w)$, we can bound
\[\abs*{ \mathscr{C}_{\phi}(w) } 
= \abs{ \sum_{H\in \textup{Crit}(w)} \EX_{w\to H}[\phi] }
\le |\textup{Crit}(w)|\cdot \phi(1). \]
Now let $B_0\defeq \max_{\phi\in \hat{G}} \abs*{ \inner*{f, \chi_\phi} }$ and choose $B\defeq \max\prn*{1, B_0\cdot |G|}$. Then
\[\abs*{ \mathscr{C}_w(f) } \le B_0 \sum_{\phi\in \hat{G}} |\textup{Crit}(w)|\cdot \phi(1) \le B_0\cdot |\textup{Crit}(w)| \cdot |G| 
= B\cdot |\textup{Crit}(w)|.\]
\end{proof}

\begin{corollary}
\label{corollary_effective_bound_f}
For every stable class function $f\in \mathcal{A}(G)$ of degree $D$, there is a constant $C\ge 1$ such that for every non-power $w$ of length $t$ with $\pi(w)<\infty$ and $n > \prn*{t\cdot D}^2$,
\[ \abs*{ \EX_w[f] - \inner*{f, \mathbf{1}}  } \le
\frac{C}{n^{\pi(w) - 1}}\cdot \prn*{|\textup{Crit}(w)| + \frac{\prn*{t\cdot D}^{2(\pi(w) + D)}}{n - \prn*{t\cdot D}^2}}.\]
\end{corollary}

\begin{proof}
Take $C\defeq\max\{A, B\}$, for the constant $A$ of Corollary~\ref{corollary_effective_bound_sInd_comb} and the constant $B$ of Proposition~\ref{prop_bound_critical_value}.
\end{proof}

Now we have finished generalizing the tools from \cite{HP22} to wreath products, and we can follow the rest of the proof from \cite{HP22} without change (even though we have a slightly different notation).

\begin{definition}
Let $\Gamma$ be a graph. Denote by $\Vec{E}(\Gamma)$ the set of oriented edges in $\Gamma$, i.e.\ for every $e\in E(\Gamma)$ we have $e, \Bar{e}\in \Vec{E}(\Gamma)$, and $|\Vec{E}(\Gamma)| = 2|E(\Gamma)|$.
The \textbf{Hashimoto} matrix of $\Gamma$, denoted
$B = B_{\Gamma}\in M_{\Vec{E}(\Gamma)\times \Vec{E}(\Gamma)}(\{0,1\})$,
is defined by
\[B_{e, f} = \bbone_{trg(e)=src(f)\textrm{ and } f\neq \Bar{e}}. \]
If $\Gamma$ is $d$-regular, we denote by $\nu(\Gamma)$ an eigenvalue of maximal absolute value among the non-trivial (i.e.\ not $\in \{1, d-1\}$) eigenvalues of its Hashimoto matrix. 
\end{definition}

We wish to bound (from above) $\abs*{\nu(\Gamma)}$, where $\Gamma$ is the Schreier graph of a representation-stable action of $G\wr S_n$. 
Due to the Ihara-Bass formula, this will imply a bound on the nontrivial spectrum of the adjacency matrix.

\begin{theorem}
\label{thm_cite_HP22_expansion_graphs}
(\cite[Section 8]{HP22})
Let $r\ge 2$, and let $(G_n\acts X_n)_{n=1}^{\infty}$ be a sequence of transitive finite group actions, with corresponding permutation character $\psi = \psi_n$, of dimension $|X_n| = O(n^s)$ for some $s > 0$.
Denote $d\defeq 2r$ and $\Gamma_n \defeq \textup{Sch}_r(\psi_n)$, a $d$-regular graph.
If there are constants $C, D \ge 1$ such that for every non-power $w\in F_r$ with $\pi(w) < \infty$, of length $t\defeq |w|$, and every $n > \prn*{t\cdot D}^2$,
$ \abs*{\EX_w[\psi] - 1} \le \frac{C}{n^{\pi(w) - 1}}\cdot \prn*{|\textup{Crit}(w)| + \frac{\prn*{t\cdot D}^{2(\pi(w) + D)}}{n - \prn*{t\cdot D}^2}}, $
then a.a.s.\ (as $n\to \infty$)
\[ \mu_{\Gamma_n} \le 2\sqrt{d - 1}\cdot \exp\prn*{ \frac{2s^2}{e^2(d-1)} }. \]
\end{theorem}

\begin{proof}
This was proved in \cite{HP22}, but with a slightly different formulation and notation, so we present some milestones along the proof to help translate the proof of \cite{HP22} to our language.
\begin{enumerate}
    \item\label{item_expansion_trace}  Let $t\in 2\N$. Then
    $\EX[Re[\nu(\Gamma_n)^t]] \le \brackets*{ \sum_{w\in \mathfrak{CR}_t(\textbf{F}_r)} (\EX_w[\psi] - 1) } + 2|X_n| (d-1)^{t/2}. $

    \item\label{item_expansion_crit} \cite[Proposition 4.3 and Theorem 8.2]{Expansion_Puder_2015} 
    For every $r\ge 2$ and $m\in \{1, \ldots, r\}$,
    \[ \limsup_{t\to\infty} \brackets*{ \sum_{w\in \mathfrak{CR}_t({F}_r): \,\pi(w) = m} |\textup{Crit}(w)| }^{\frac{1}{t}} = \max(\sqrt{2r-1}, 2m-1). \]
    Note that we consider only $m\in \{1, \ldots, r\}$ even though $\{\pi(w)\}_{w\in F_r} = \{0, \ldots, r, \infty\}$ since we use only long words (so $\pi(w)\neq 0$), and primitive words $w$ do not contribute to the summation (since $\EX_w[\psi] - 1 = 0$) so we may assume $\pi(w) < \infty$.
    
    \item The proper powers are the words with $\pi(w)=1$, and $\abs*{\textup{Crit}(w)}\ge 1$, so by item~\ref{item_expansion_crit} with $m=1$ their number in $\mathfrak{CR}_t({F}_r)$ is at most $\brackets*{(1+\varepsilon)\sqrt{d-1}}^t$ for every $\varepsilon>0$ and every large enough $t$.
    Since $\abs*{\EX_w[\psi] - 1}\le |X_n| + 1$ for every $w$, item~\ref{item_expansion_trace} gives
    \[ \EX[Re[\nu(\Gamma_n)^t]] \le \sum_{\substack{w\in \mathfrak{CR}_t({F}_r):\\ \pi(w)\ge 2}} \prn*{\EX_w[\psi] - 1} + 2|X_n| \brackets*{(1+\varepsilon)\sqrt{d - 1}}^t. \]

    \item\label{item_expansion_nonpowers} Only the words with $2\le \pi(w)\le r$ contribute to the sum above, and for them $\prn*{t\cdot D}^{2(\pi(w) + D)} \le \prn*{t\cdot D}^{2(r + D)}$ and $\abs*{\textup{Crit}(w)}\ge 1$, hence
    $\sum_{\substack{w\in \mathfrak{CR}_t({F}_r):\\ \pi(w)\ge 2}} \prn*{\EX_w[\psi] - 1} \le C\cdot \prn*{1 + \frac{\prn*{t\cdot D}^{2(r + D)}}{n - \prn*{t\cdot D}^2}} \sum_{m=2}^r \frac{1}{n^{m - 1}} \sum_{w\in \mathfrak{CR}_t({F}_r):\, \pi(w) = m} \abs*{ \textup{Crit}(w)}$
    
    \item Combining items~\ref{item_expansion_crit} and~\ref{item_expansion_nonpowers}, for every $\varepsilon > 0$ and every large enough $t = t(\varepsilon)$,
    \[\sum_{\substack{w\in \mathfrak{CR}_t({F}_r):\\ \pi(w)\ge 2}} \prn*{\EX_w[\psi] - 1} \le C\cdot \prn*{1 + \frac{\prn*{t\cdot D}^{2(r + D)}}{n - \prn*{t\cdot D}^2}} \sum_{m=2}^r \frac{1}{n^{m - 1}}  \brackets*{\max(\sqrt{2r-1}, 2m-1) +\varepsilon}^t. \]
    
    \item Choosing $t\approx log_c(n)$ (and keeping $t\in 2\N$), for every $\varepsilon > 0$ and large enough $n, t$ we get
    \[\EX[Re[\nu(\Gamma_n)^t]] \le \brackets*{ (1 + \varepsilon) \cdot \sqrt{d - 1} \cdot c ^ s}^t. \]
    If $d$ is large then the optimal value of $c$ turns out to be $ c = e^{\frac{2}{e\sqrt{d - 1}}}. $
    
    \item By using the Ihara-Bass formula one can get a connection between the spectrum of the Hashimoto matrix and the adjacency matrix, and by some more analysis the result follows.

\end{enumerate}
\end{proof}

Theorem~\ref{thm_expander_wreath} now follows.
Indeed, the permutation character $\psi$ of $(\rho_n)_{n=1}^{\infty}$ is stable of degree $k$, and $\inner*{\psi, \textbf{1}} = 1$ as $\rho_n$ is transitive.
For $k = 0$ the character $\psi$ is constant, so $\psi = \textbf{1}$, each $X_n$ is a single point and the statement is trivial; assume therefore that $k\ge 1$.
By Theorem~\ref{thm_irred_basis_of_AG}, $\psi$ is a linear combination of stable irreducible characters of degree at most $k$, so $|X_n| = \psi_n(1) = O\prn*{n^k}$ by Corollary~\ref{corollary_dim_poly}.
The first bound of Theorem~\ref{thm_expander_wreath} is therefore the case $s = k$, $d = 2r$ of Theorem~\ref{thm_cite_HP22_expansion_graphs}, whose assumption holds with $D = k$ and with the constant $C$ from Corollary~\ref{corollary_effective_bound_f}.
Its second bound follows since $2\sqrt{2r - 1}\cdot \prn*{\exp\prn*{ \frac{2k^2}{e^2(2r-1)} } - 1} \to 0$ as $r\to \infty$.

\begin{remark}
As in \cite[Remark 8.7]{HP22}, also here Conjecture~\ref{conj_great_HP22} for $(G\wr S_n)_n$ implies a uniform bound 
$ \mu_{\Gamma_n} \le 2\sqrt{d - 1}\cdot \exp\prn*{ \frac{2}{e^2(d-1)} } $ (not depending on $|X_n|$).
\end{remark}

\section{Proof of the Induction Convolution Lemma}
\label{section_proof_of_ICL}

Induced representations have the 
following convenient form:

\begin{remark}
\label{remark_explicit_general_induction}
Generally, given a representation $\rho\in \Hom(K, GL(W))$ of a subgroup $K\le G$, one can write $\Ind_K^G \rho$ explicitly as a block-matrix representation.
Given a transversal $S$ of $G/K$ and $g\in G$, the matrix $\Ind^G_K \rho(g)$ has $S\times S$ blocks, where the $(s, t)$ block is
\begin{equation*}
    \Ind^G_K \rho(g)_{s, t} = \begin{cases}
    \rho(s^{-1} g t) &\textrm{if }s^{-1}gt\in K,\\
    0& \textrm{otherwise.}
    \end{cases}
\end{equation*}
\end{remark}

\begin{definition}[Block-Trace]
Let $A\in M_{n\times n}(\FG)$ be a matrix.
For every $d \divides n$ let
\[ tr^n_d(A) \defeq \sum_{i=1}^{n/d} A_{d: (i, i)} \in M_{d\times d}(\FG) \]
where $A_{d: (i, i)}$ is the $(i, i)^{th}$ block in $A$.
Clearly $tr = tr^n_1$, and for every $a \divides b \divides n$ we have $tr^b_a \circ tr^n_b = tr^n_a$.
\end{definition}

Before we prove the induction-convolution lemma for multiple words, we need some more machinery, that will help us deal with different characters of different dimensions together: tensor product. We denote Kronecker product by $\otimes$.
\begin{definition}
\label{def_tensor_embeddings}
Let $a_1, \ldots, a_k\in \N$.
Define a matrix algebra 
\[R_{a_1, \ldots, a_k}(\FG) \defeq \bigotimes_{i=1}^k M_{a_i\times a_i}(\FG) = \textup{End}\prn*{ \bigotimes_{i=1}^k \FG^{a_i} } \]
and define for every $1\le i\le k$ a ring-homomorphism $T_i \colon M_{a_i\times a_i }(\FG) \to R_{a_1, \ldots, a_k}(\FG)$
by tensoring the identity matrix, i.e. applying the Kronecker product:
\[ T_1(A) = A \otimes I_{a_2} \otimes \ldots \otimes I_{a_k}, \quad\quad \ldots, \quad\quad  T_k(A) = I_{a_1} \otimes I_{a_2} \otimes \ldots \otimes A. \]
\end{definition}

In the proof of the induction-convolution lemma, we will need to handle a product of induced characters, that are all induced from an index $n$ subgroup. To make the different representations act on the same vector space, we use tensors.

\begin{proposition} 
\label{prop_tensors}
Let $a_1, \ldots, a_k, n\in \N$, and let $(A_i)_{i=1}^k$ be square matrices, $A_i\in M_{(a_i n)\times (a_i n)}(\FG)$, each viewed as an $n\times n$ matrix of blocks of size $a_i\times a_i$.
The tensor embeddings
$T_i \colon M_{a_i\times a_i }(\FG) \to R_{a_1, \ldots, a_k}(\FG)$
behave nicely with trace and product:
\begin{enumerate}
    \item The product of traces is $\prod_{i=1}^k tr^{a_i n}_1 (A_i) = tr^{a_1\cdots a_k}_1 \prn*{\prod_{i=1}^k T_i\prn*{tr^{a_i n}_{a_i} A_i}}. $
    \item For every $i \neq j\in [k]$ and all matrices $X\in M_{a_i\times a_i}(\FG), Y\in M_{a_j\times a_j}(\FG)$, $T_i(X) T_j(Y) = T_j(Y) T_i(X).$
\end{enumerate}
\end{proposition} 

\begin{proof}
We use standard properties of Kronecker product: 
\begin{itemize}
    \item For every 2 matrices $A, B$, we have $tr(A\otimes B) = tr(A)\cdot tr(B)$.
    \item For every 4 matrices $A_{a\times a'}, B_{a'\times b}, C_{c\times c'}, D_{c', d}$, we have
    $(A\cdot B)\otimes (C\cdot D) = (A\otimes C)\cdot (B\otimes D)$
\end{itemize}
For simplicity of notation, we prove the proposition only for $k=2$:
\begin{enumerate}
    \item Product of Traces:
    \begin{equation*}
    \begin{split}
        \prod_{i=1}^2 tr^{a_i n}_1 (A_i)
        &= tr^{a_1 a_2}_1 \prn*{\prn*{ tr^{a_1 n}_{a_1}A_1}
        \otimes \prn*{tr^{a_2 n}_{a_2}A_2 }} = tr^{a_1 a_2}_1 \prn*{ \prn*{tr^{a_1 n}_{a_1}A_1\cdot I_{a_1}} 
        \otimes \prn*{ I_{a_2}\cdot \prn*{tr^{a_2 n}_{a_2}A_2 }} }\\
        &= tr^{a_1 a_2}_1\prn*{ \prn*{tr^{a_1 n}_{a_1}A_1 \otimes I_{a_2}} \cdot \prn*{I_{a_1}\otimes tr^{a_2 n}_{a_2}A_2} } = tr^{a_1 a_2}_1\prn*{ T_1\prn*{tr^{a_1 n}_{a_1}A_1}\cdot T_2\prn*{tr^{a_2 n}_{a_2}A_2} }.
    \end{split}
    \end{equation*}
    
    \item Different embeddings commute:
    \begin{equation*}
        \begin{split}
            T_1(X) \cdot T_2(Y) 
            &= \prn*{X\otimes I_{a_2}}\prn*{I_{a_1}\otimes Y} = \prn*{X \cdot I_{a_1}}\otimes \prn*{I_{a_2}\cdot Y}\\
            &= \prn*{I_{a_1}\cdot X}\otimes \prn*{Y\cdot I_{a_2}} = \prn*{I_{a_1}\otimes Y}\prn*{X\otimes I_{a_2}} = T_2(Y)\cdot T_1(X).
        \end{split}
    \end{equation*}
\end{enumerate}
\end{proof}

Recall the notations $G_n\defeq G\wr S_n$ and $(\Ind\zeta)_j \defeq \Ind_{G\times G_{n-1}}^{G_n} (\zeta_j)$, and recall Lemma~\ref{lemma_ICL}:
Let $\Vec{w} = \braces*{w_1, \ldots, w_k}$ be a multi-word, $\Vec{\mathcal{J}}\in \moccFr$ and $\eta \in \Hom(\Vec{w}, \Vec{\mathcal{J}})$ with representatives $\braces*{J_1, \ldots, J_{\ell}}$. 
Also let $G$ be a finite group, and $\zeta\colon \Vec{w}\to \FGconjG$. 
Then
\[\EX_{\eta}[\Ind\zeta] 
= \sum_{(\eta_1, \eta_2)\in \DecompB(\eta)} \EX_{\eta_1}[\zeta] \cdot L_{\eta_2}^B(n). \]

\begin{proof} [Proof of Lemma~\ref{lemma_ICL}]
We work in steps.\\
\textbf{Step 0:}
We make a reduction to the case $\zeta\colon \Vec{w}\to \PGhat$ of irreducible characters.
Recall Definition~\ref{def_ex_eta_zeta}, and note that induction is a linear operator on class functions, expectation is a linear operator, and $\EX_{\eta}[\cdot], \EX_{\eta_1}[\cdot]$ are multi-linear in $\zeta$, i.e.\ linear in every $\zeta_i$.
Thus it suffices to prove the lemma for the linear basis $\braces*{\zeta\colon \Vec{w}\to \PGhat}$.
So from now on assume $\zeta\colon \Vec{w}\to \PGhat$. \\ 

\textbf{Step 1:} 
We make a reduction to the case $\ell=1$.
$\EX_{\eta}[\Ind\zeta]$ is multiplicative with respect to $J_1,\ldots, J_{\ell}$ so we want to show that $\sum_{(\eta_1, \eta_2)\in \DecompB(\eta)} \EX_{\eta_1}[\zeta] \cdot L_{\eta_2}^B(n)$ is also multiplicative.
Indeed, for every $1\le t\le \ell$ denote $\Vec{w}_t \defeq \eta^{-1}(J_t), \,\,\eta^{t} \defeq \eta \restriction_{\Vec{w}_t}$. Then since $\EX_{\eta_1}[\zeta], L_{\eta_2}^B(n)$ are both multiplicative, 
\begin{equation*}
    \begin{split}
        & \prod_{t=1}^{\ell} \sum_{(\eta_1^t, \eta_2^t)\in \DecompB(\eta^t)} \EX_{\eta_1^t}[\zeta] \cdot L_{\eta_2^t}^B(n) \\
        &= \sum_{(\eta_1^1, \eta_2^1)\in \DecompB(\eta^1)} \ldots \sum_{(\eta_1^{\ell}, \eta_2^{\ell})\in \DecompB(\eta^{\ell})}  \prod_{t=1}^{\ell} \EX_{\eta_1^t}[\zeta] \cdot L_{\eta_2^t}^B(n)\\
        &= \sum_{(\eta_1, \eta_2)\in \DecompB(\eta)} \EX_{\eta_1}[\zeta] \cdot L_{\eta_2}^B(n).\\
    \end{split}
\end{equation*}
Thus from now on we may assume that $\Vec{\mathcal{J}}$ consists of a single conjugacy class, and, replacing the ambient group $F_r$ by a representative of that class, that $\Vec{\mathcal{J}} = \braces*{F_r}$.\\

\textbf{Step 2:} Recall Remark~\ref{remark_explicit_general_induction}. 
Denote the dimension of $\zeta_i$ by $d_i \defeq \zeta_i(1)$, and let $\rho_i\in \Hom(G_n, GL_{d_i n}(\FG))$ be a representation that yields $\Ind\zeta_i$ in the form of Remark~\ref{remark_explicit_general_induction}. 
We wish to use the block-matrix structure of $\rho$.
\begin{equation*}
    \begin{split}
        \EX_{\eta}[\Ind\zeta]
        &= \EX_{(g_1, \ldots, g_r)\in G_n^r} \brackets*{\prod_{i=1}^k \Ind\zeta_i(w_i(g_1, \ldots, g_r))}\\
        &= \EX_{(g_1, \ldots, g_r)\in G_n^r} \brackets*{\prod_{i=1}^k tr^{d_i n}_{1} \rho_i(w_i(g_1, \ldots, g_r))}.
    \end{split}
\end{equation*}
For every $i$, let $(V_{\zeta_i}, \Phi_i)$ be the representation that yields the irreducible character $\zeta_i$, i.e.\ $\zeta_i = tr(\Phi_i)$.
Now we define a ring $R \defeq \bigotimes_{i=1}^k \textup{End}(V_{\zeta_i}) = \textup{End}\prn*{\bigotimes_{i=1}^k V_{\zeta_i}}. $
This ring consists of $D\times D$ matrices over $\FG$, where $D\defeq \prod_{i=1}^k d_i. $
Now by Proposition~\ref{prop_tensors} (part 1) about Kronecker products we get
\[\prod_{i=1}^k tr^{d_in}_1 \rho_i(w_i(g_1, \ldots, g_r))
= tr^{D}_1 \prn*{\prod_{i=1}^k T_i\prn*{tr^{d_i n}_{d_i} \rho_i\prn*{w_i\prn*{g_1, \ldots, g_r}}}} \]
where $T_i\colon \textup{End}(V_{\zeta_i})\to R$ is the map defined earlier in Definition~\ref{def_tensor_embeddings}.
This lets us work with $tr^{d_i n}_{d_i} \rho_i\prn*{w_i\prn*{g_1, \ldots, g_r}}$
which is a trace of an $n\times n$ matrix over the ring $M_{d_i\times d_i}(\FG)$. 
Thus for a word $w_i = b_{j(1)}^{\varepsilon(1)}\cdots b_{j(L_i)}^{\varepsilon(L_i)} \in F_r$ of length $L_i$ (where $\varepsilon(1), \ldots, \varepsilon(L_i)\in \{\pm 1\}$, the functions $j$ and $\varepsilon$ depend on $i$, and $B = \{b_j\}_{j=1}^r$ is a basis for $F_r$), we have \allowdisplaybreaks
\begin{equation*}
    \begin{split}
        tr^{d_i n}_{d_i} \rho_i\prn*{w_i\prn*{g_1, \ldots, g_r}}
        &= tr^{d_i n}_{d_i} \rho_i\prn*{g_{j(1)}^{\varepsilon(1)}\cdots g_{j(L_i)}^{\varepsilon\prn*{L_i}}}\\
        &= tr^{d_i n}_{d_i} \rho_i\prn*{g_{j(1)}^{\varepsilon(1)}}\cdots \rho_i\prn*{ g_{j\prn*{L_i}}^{\varepsilon\prn*{L_i}}}\\
        &= \sum_{f\colon \Z/L_i\Z \to [n]} \rho_i\prn*{g_{j(1)}^{\varepsilon(1)}}_{f(0),f(1)} \cdots \rho_i\prn*{g_{j(L_i)}^{\varepsilon\prn*{L_i}}}_{f\prn*{L_i-1},f(0)}
    \end{split}
\end{equation*}
which gives us, by the formula for trace of product, and since $T_i$ are ring-homomorphisms,
\begin{equation*}
    \begin{split}
    \EX_{\eta}[\Ind\zeta] 
    &= \EX_{(g_1, \ldots, g_r)\in G_n^r} \brackets*{tr^{D}_1 \prn*{\prod_{i=1}^k T_i\prn*{tr^{d_i n}_{d_i} \rho_i(w_i(g_1, \ldots, g_r))}}} \\
    &= \EX_{(g_1, \ldots, g_r)\in G_n^r} \brackets*{tr^{D}_1 \prn*{\prod_{i=1}^k T_i\prn*{\sum_{f\colon \Z/L_i\Z \to [n]} \rho_i\prn*{g_{j(1)}^{\varepsilon(1)}}_{f(0),f(1)} \cdots \rho_i\prn*{g_{j(L_i)}^{\varepsilon(L_i)}}_{f(L_i-1),f(0)}}}} \\
    &= \EX_{(g_1, \ldots, g_r)\in G_n^r} \brackets*{tr^{D}_1 \prn*{\prod_{i=1}^k \sum_{f\colon \Z/L_i\Z \to [n]} T_i\rho_i\prn*{g_{j(1)}^{\varepsilon(1)}}_{f(0),f(1)} \cdots T_i\rho_i\prn*{g_{j(L_i)}^{\varepsilon(L_i)}}_{f(L_i-1),f(0)}}} \\
    &= \EX_{(g_1, \ldots, g_r)\in G_n^r} \brackets*{tr^{D}_1 \prn*{\sum_{f\colon V(\Vec{w})\to [n]}\prod_{i=1}^k T_i\rho_i\prn*{g_{j(1)}^{\varepsilon(1)}}_{f_i(0),f_i(1)} \cdots T_i\rho_i\prn*{g_{j(L_i)}^{\varepsilon(L_i)}}_{f_i(L_i-1),f_i(0)}}} \\
    &= \sum_{f\colon V(\Vec{w})\to [n]} \EX_{(g_1, \ldots, g_r)\in G_n^r} \brackets*{tr^{D}_1 \prn*{\prod_{i=1}^k T_i\rho_i\prn*{g_{j(1)}^{\varepsilon(1)}}_{f_i(0),f_i(1)} \cdots T_i\rho_i\prn*{g_{j(L_i)}^{\varepsilon(L_i)}}_{f_i(L_i-1),f_i(0)}}} \\
    \end{split}
    \end{equation*}
where $f_i = f\restriction_{\Gamma_B(w_i)}$ is the restriction of $f$ to the connected component $\Gamma_B(w_i)\subseteq \Gamma_B(\vec{w})$.
The $n\times n$ matrices $\rho_i$ (whose entries are in $\textup{End}(V_{\zeta_i})$) are monomial (that is, every row and column has precisely one nonzero entry) by the structure of induced representation.
Thus if we have two edges $e_1 = (s_1\overset{b_j}{\to} t_1), e_2 = (s_2\overset{b_j}{\to} t_2)\in E(\Gamma_B(\vec{w}))$ with a collision $f(s_1) = f(s_2)$, we must also have $f(t_1) = f(t_2)$ for the contribution of $f$ to be nonzero. (This is true also in the opposite direction, $f(t_1) = f(t_2)\implies f(s_1) = f(s_2)$, as the group element we substitute in $b_j$ is invertible).
In other words, the quotient graph obtained by gluing vertices with the same $f$-value is a multi core graph.
Therefore, we can decompose each $f$ to its unique surjective-then-injective decomposition, and obtain a quotient map $\eta_1\in \mathcal{Q}_B(\vec{w})$ with image graph $\Gamma = \textup{Im}(\eta_1)$, and we define $L_{\eta_1}'$ by considering only injective vertex-functions:
\[L_{\eta_1}'(\Ind\zeta) \defeq \sum_{f\colon V(\Gamma)\hookrightarrow [n]} \EX \brackets*{tr^{D}_1 \prn*{\prod_{i=1}^k \prod_{z_i\in \Z / L_i\Z} T_i\rho_i\prn*{g_{j(z_i)}^{\varepsilon(z_i)}}_{f_i(\eta_1(z_i)),f_i(\eta_1(z_i + 1))}}}. \]
where the expectation is over $(g_1, \ldots, g_r)\sim U\prn{G_n^r}$
(the notation assumes an implicit bijection between $V(\Gamma_B(w_i))$ and $\Z/L_i\Z$).
This means we proved 
\begin{equation*}
   \EX_{\eta}[\Ind\zeta] = \sum_{\eta_1\in \mathcal{Q}_B(\Vec{w})} L_{\eta_1}'(\Ind\zeta).
\end{equation*}

From now on we fix $\eta_1\in \mathcal{Q}_B(\Vec{w})$ with image $\Gamma = \textup{Im}(\eta_1)$ and we want to prove
\begin{equation*}
\begin{split}
    L_{\eta_1}'(\Ind\zeta) &\overset{?}{=} \EX_{\eta_1}[\zeta] \cdot \sum_{\substack{\eta_2:\quad (\eta_1, \eta_2)\\\in \DecompB(\eta)}}  L_{\eta_2}^B(n)\\
    &= \EX_{\eta_1}[\zeta]\cdot L_{\eta_2}^B(n)
\end{split}
\end{equation*}
where in the last expression, 
$\eta_2$ is the unique morphism 
$\Gamma \to \Omega_B$
(Recall that 
$\textup{Im}(\eta) = \Omega_B$
is the bouquet graph on the letters of 
$\vec{w}$.
)

\textbf{Step 3:}
Recall that we have fixed $\Gamma_B(\vec{w})\overset{\eta_1}{\to} 
\Gamma \overset{\eta_2}{\to} \Omega_B$.
We say that an injective function $f\colon V(\Gamma)\hookrightarrow [n]$ is \textbf{valid} with respect to $g_1, \ldots, g_r\in G_n$, 
if for every edge $e = (s\overset{b_j}{\to} t)\in E(\Gamma)$ we have $g_j(f(s)) = f(t)$. (Here we use the action $G_n\acts [n]$). 
By substituting trivial representations $\zeta_i = \textbf{1}$ for every $i$, we see that
\begin{equation*}
\begin{split}
L_{\eta_1}'(\Ind \textbf{1}) 
&= \sum_{f\colon V(\Gamma)\hookrightarrow [n]} \EX_{(g_1, \ldots, g_r)\in G_n^r} \brackets*{tr^{1}_1 \prn*{\prod_{i=1}^k \bbone_{g_{j(1)}^{\varepsilon(1)}.f_i(\eta_1(0)) = f_i(\eta_1(1))} \cdots \bbone_{g_{j(L_i)}^{\varepsilon(L_i)}.f_i(\eta_1(L_i-1)) = f_i(\eta_1(0))}}} \\
&= \EX_{(g_1, \ldots, g_r)\in G_n^r} \#\{\textup{valid injective functions }f\colon V(\Gamma)\hookrightarrow [n]\} \\
&= L^B_{\eta_2}(n).
\end{split}
\end{equation*}
(Recall that $\eta_2$ was determined by $\eta_1$ so this is not surprising).\\

\textbf{Step 4:}
Now is the first time we use the wreath-product structure of $G_n$. We have
$ \forall s, t\in [n], i\in [k], v\in G^n, \sigma\in S_n$:
\[\rho_i(v, \sigma)_{s, t} = \bbone_{\sigma(t)=s} \Phi_i(v(s))\]
(recall $\zeta_i = tr(\Phi_i)$, i.e.\ $\Phi_i\in \Hom(G, GL_{d_i}(\FG))$ is the representation yielding $\zeta_i$).
In particular for every $s, t\in [n], \varepsilon\in \{\pm 1\}$ and $i\in [k]$, 
\[ T_i\prn*{\rho_i(v, \sigma)^{\varepsilon}_{s, t}} = \bbone_{\sigma^{\varepsilon}(t)=s} T_i\Phi_i(v^{\varepsilon}(x)) \]
where 
\begin{equation*}
    x=
    \begin{cases}
    s& \textrm{ if } \varepsilon=1,\\
    t& \textrm{ if } \varepsilon=-1.
    \end{cases}
\end{equation*}
Sampling $g_1, \ldots, g_r\in G_n = G\wr S_n$ at random is the same as sampling $v_1, \ldots, v_r\in G^n, \sigma_1, \ldots, \sigma_r\in S_n$ at random.
For every oriented edge $e = (s\overset{b_j}{\to} t)\in \Vec{E}(\Gamma)$ and every $f\colon V(\Gamma)\hookrightarrow [n]$, associate to $e$ the $G$-random variable $\beta(e)\defeq v_j(f(s))$, and to the reversed edge $\beta(\bar{e})\defeq \beta(e)^{-1}$, so that $\beta$ is anti-symmetric as in Proposition~\ref{prop_geometric_interpretation_of_E_eta}. Every $\beta(e)\sim U(G)$ is distributed uniformly.
In the products below, a letter $b_j^{\varepsilon}$ of $w_i$ contributes $\Phi_i(\beta(e))$ for the oriented edge $e$ that it traverses, which is $\Phi_i\prn*{v_j(f(s))}^{\varepsilon}$.
If $e_1, e_2$ are different edges in $\Gamma$ with the same label $b_j$ and $f$ is valid, 
then $f(src(e_1))\neq f(src(e_2))$, so $\beta(e_1)$ and $\beta(e_2)$ are distinct coordinates of $v_j$, and thus the associated $G$-random variables are independent (conditioned on the validity of $f$). The variables associated with edges of different labels are independent as well. 
Now for every injective vertex-function $f$,
\begin{equation*}
    \begin{split}
    (\star) 
    \defeq& \EX_{(g_1, \ldots, g_r)\in G_n^r} \brackets*{tr^{D}_1 \prn*{\prod_{i=1}^k T_i\rho_i\prn*{g_{j(1)}^{\varepsilon(1)}}_{f_i(\eta_1(0)),f_i(\eta_1(1))} \cdots T_i\rho_i\prn*{g_{j(L_i)}^{\varepsilon(L_i)}}_{f_i(\eta_1(L_i-1)),f_i(\eta_1(0))}}}\\
    =& \EX_{(g_1, \ldots, g_r)\in G_n^r} \brackets*{\bbone_{f \textup{ is valid}} \cdot tr^{D}_1 \prn*{\prod_{i=1}^k \prod_{e\in {\eta_1}_*(\Gamma_B(w_i))} T_i \Phi_i (\beta(e))}}\\
    =& \PR\prn*{f \textup{ is valid}}\cdot \EX_{(g_1, \ldots, g_r)\in G_n^r} \brackets*{tr^{D}_1 \prn*{\prod_{i=1}^k T_i \Phi_i \prn*{\prod_{e\in {\eta_1}_*(w_i)} \beta(e)}} \middle\vert f \textup{ is valid}}.\\
    \end{split}
\end{equation*}
Here the product $\prod_{e\in {\eta_1}_*(\Gamma_B(w_i))}$ runs over the $\eta_1$-images of $w_i$-edges, according to their order in $\Gamma_B(w_i)$. 
By Proposition~\ref{prop_tensors} (part 1), and since $\beta(e)$ are all independent conditioned on the validity of $f$,
\begin{equation*}
    \begin{split}
        (\star) 
        =& \PR\prn*{f \textup{ is valid}}\cdot \EX_{(g_1, \ldots, g_r)\in G_n^r} \brackets*{\prod_{i=1}^k tr\Phi_i \prn*{\prod_{e\in {\eta_1}_*(w_i)} \beta(e)} \middle\vert f \textup{ is valid}}\\
        =& \PR\prn*{f \textup{ is valid}}\cdot \EX_{\beta\sim U\prn*{E(\Gamma)\to G}} \brackets*{\prod_{i=1}^k \zeta_i \prn*{\prod_{e\in {\eta_1}_*(w_i)} \beta(e)}}.\\
        (\textup{Prop.~\ref{prop_geometric_interpretation_of_E_eta}})\quad\quad\quad = & 
        \PR\prn*{f \textup{ is valid}}\cdot \EX_{\eta_1}[\zeta].
    \end{split}
\end{equation*}
This finishes the proof:
\begin{equation*}
\begin{split}
    L_{\eta_1}'(\Ind\zeta) 
    &= \sum_{f\colon V(\Gamma)\hookrightarrow [n]} (\star) \\
    &= \sum_{f\colon V(\Gamma)\hookrightarrow [n]} \PR\prn*{f \textup{ is valid}}\cdot \EX_{\eta_1}[\zeta]\\
    &= \EX_{\eta_1}[\zeta] \cdot \sum_{f\colon V(\Gamma)\hookrightarrow [n]} \PR\prn*{f \textup{ is valid}}\\
    &= \EX_{\eta_1}[\zeta] \cdot L^B_{\eta_2}(n)\\
    &= \EX_{\eta_1}[\zeta]\cdot L_{\eta_2}^B(n).
\end{split}
\end{equation*}
\end{proof}

\appendix
\section{Stable Representations: Brief Overview}
\label{appendix_stable_algebra}

Stability is formalized in \cite{Church_2015} for some specific families of groups.
In the case of the symmetric groups $S_n$, stable representations coincide with finitely generated \textbf{FI}-modules, defined in \cite{Church_2015}.
An \textbf{FI}-module over a ring $R$ is a functor from the category of \textbf{F}inite sets and \textbf{I}njective functions, to the category of $R$-modules.
This idea of analyzing stable representations via functors from \enquote{set-like} categories has been developed quite extensively since then, 
as demonstrated in the following table: 

\begin{table}[ht!]
\centering
\begin{tabular}{|| c | c | c | c | c||} 
\hline
Paper
& Category
& $G_n$ 
& Description \\ [0.5ex] 
\hline\hline

\cite{Church_2015}
& \textbf{FI} 
& $S_n$ 
& Symmetric groups\\  [1ex] \hline

\cite{WILSON2014269}
& $\textbf{FI}_{\mathcal{W}}$
& $B_n, D_n$
& Classical Weyl groups \\ [1ex] \hline

\cite{Sam_2019}, \cite{gan2016coinduction}, \cite{casto2016fig}
& $\FIG$
& $G\wr S_n$ 
& Wreath products with $S_n$ \\ [1ex] \hline

\cite{Putman_2017},
\cite{gan2017representation} 
& $\textbf{VI}_q, \textbf{VI}$
& $\textup{GL}_n(\F_q)$
& General linear groups over $\F_q$ \\ [1ex] \hline

\cite{Gadish_2017}
& $\textbf{C}$ of $\textbf{FI}$-type 
& $\textup{Aut}_{\textbf{C}}(c)$
& Automorphisms of $\textbf{C}$-objects \\ [1ex] \hline

 \hline
\end{tabular}
\caption{Papers about Categories of FI-type}
\label{table:PapersFI}
\end{table}
\FloatBarrier

\noindent where the categories are defined as
\begin{table}[ht!]
\centering
\begin{tabular}{|| c | c | c ||} 
\hline
Category
& Objects
& Morphisms \\ [0.5ex] 
\hline\hline

\textbf{FI} 
& $\{\{1, \ldots, n\}\}_{n=0}^{\infty}$
& Injective functions \\  [1ex] \hline

$\textbf{FI}_{\mathcal{W}}$
& $\{\{\pm 1, \ldots, \pm n\}\}_{n=0}^{\infty}$
& Injective odd functions\\ [1ex] \hline

$\FIG$
& $\{[n]\times G\}_{n=0}^{\infty}$ 
& $G$-equivariant injections \\ [1ex] \hline

$\textbf{VI}_q$
& $\{\F_q^n\}_{n=0}^{\infty}$
& $\F_q$-linear embeddings \\ [1ex] \hline

 \hline
\end{tabular}
\caption{Categories of FI-type: Objects and Morphisms}
\label{table:categoriesFI}
\end{table}
\FloatBarrier

\noindent (up to some technical details regarding $\textbf{FI}_{\mathcal{W}}$).

Consider the algebra of stable class functions on $G\wr S_{\bullet}$.
The \enquote{integral} elements of this algebra form the ring of stable (virtual) characters. 
This ring is the Grothendieck ring of the category of algebraic representations of $G\wr S_{\infty}$, that is, the ring of characters of algebraic\footnote{An algebraic representation of $S_{\infty}$ is a sub-quotient of a tensor power of the natural permutation representation $\textup{nat}(S_{\infty}) = \Q^{\infty}$. The analog for a wreath product is a sub-quotient of product of representations of the form $\Indphi$, see Definition~\ref{def_Ind_phi}. } representations of $G\wr S_{\infty}$.
See \cite{sam_snowden_stability} for more about algebraic representations and representation stability.

\section{Glossary}

\begin{longtable}{||c | c | c ||} 
 \hline
 Notation 
 & Description 
 & Comments \\[0.7ex] 
 \hline\hline
 $[n]$              & $\{1, \ldots, n\}$                    &                                   \\[0.7ex] 
 $(n)_t$            & $ n\cdot (n-1)\cdots (n-t+1)$         & Falling factorial                 \\[0.7ex]
 $F_r$              & $\textup{Free}(\{b_1, \ldots, b_r\})$ & The ambient free group            \\[0.7ex]
 $\Omega_B$         & $\Gamma_{\{b_1, \ldots, b_r\}}(F_r) $ & The bouquet with $r$ petals       \\[0.7ex]
 $\eta$             & Morphism of multi core graphs         & Also morphism in $\mucg$          \\[0.7ex]
 $\lambda$          & Integer partition                     & Also a Young diagram              \\[0.7ex]
 $\VecLambda$       & Multi partition                       & A function to Young diagrams      \\[0.7ex]
 $G$                & A finite group                        &                                   \\[0.7ex] 
 $\hat{G}$          & Irreducible characters of $G$         &                                   \\[0.7ex] 
 Irrep              & Irreducible representation            &                                   \\[0.7ex] 
 Rep-stable         & Representation-stable                 & A property of a group action      \\[0.7ex] 
 $\textup{char}(G)$ & All characters of $G$                 &                                   \\[0.7ex]
 $\conj(G)$         & Conjugacy classes of $G$              &                                   \\[0.7ex]
 $\FG$              & A splitting field of $G$              & In characteristic 0               \\[0.7ex]
 $\FGconjG$         & Class functions $G\to \FG$            &                                   \\[0.7ex]
 $G_n$              & $G\wr S_n$                            &                                   \\[0.7ex] 
 $C_m$              & $\Z/m\Z$                              & $C_{\infty} = \Ss^1$              \\[0.7ex]
 $\textbf{1}$       & $\textbf{1}\colon G\to \{1\}$         & The trivial character             \\[0.7ex]
 $w$                & Element of $F_r$                      &                                   \\[0.7ex]
 $\EX_w[f]$         & $\EX_{\alpha\sim U(\Hom(F_r, G))} [f(\alpha(w))]$ & The $w$-measure of $f$ \\ [1ex] 
 \hline
\caption{Glossary} 
\label{table:glossary}
\end{longtable}
\FloatBarrier

\printbibliography[heading=bibintoc]

@article{Sho23I,
  title={Word measures on wreath products. I},
  author={Shomroni, Yotam},
  journal={Israel Journal of Mathematics},
  volume={270},
  number={1},
  pages={131--184},
  year={2025},
  publisher={Springer}
}

@article{HP22,
  title={Word measures on symmetric groups},
  author={Hanany, Liam and Puder, Doron},
  journal={International Mathematics Research Notices},
  volume={2023},
  number={11},
  pages={9221--9297},
  year={2023},
  publisher={Oxford University Press}
}

@article{PP15,
  title={Measure preserving words are primitive},
  author={Puder, Doron and Parzanchevski, Ori},
  journal={Journal of the American Mathematical Society},
  volume={28},
  number={1},
  pages={63--97},
  year={2015}
}

@article{magee2021surface,
  title = {Surface Words are Determined by Word Measures on Groups},
  author={Magee, Michael and Puder, Doron},
  journal={Israel Journal of Mathematics},
  volume={241},
  number={2},
  pages={749--774},
  year={2021},
  publisher={Springer}
}

@article{Puder_2014,
   title = {Primitive Words, Free Factors and Measure Preservation},
   volume={201},
   number={1},
   journal={Israel Journal of Mathematics},
   publisher={Springer Science and Business Media LLC},
   author={Puder, Doron},
   year={2014},
   pages={25–73}
}

@article{Expansion_Puder_2015,
   title = {Expansion of Random Graphs: New Proofs, New Results},
   volume={201},
   number={3},
   journal={Inventiones Mathematicae},
   publisher={Springer Science and Business Media LLC},
   author={Puder, Doron},
   year={2015},
   pages={845–908}
}

@article{unitary_Magee_2019,
   title = {Matrix Group Integrals, Surfaces, and Mapping Class Groups I: U(N)},
   volume={218},
   number={2},
   journal={Inventiones Mathematicae},
   publisher={Springer Science and Business Media LLC},
   author={Magee, Michael and Puder, Doron},
   year={2019},
   pages={341–411}
}

@article{orthogonal_magee2019matrix,
  title={Matrix group integrals, surfaces, and mapping class 
  groups II: O(n) and Sp(n)},
  author={Magee, Michael and Puder, Doron},
  journal={Mathematische Annalen},
  volume={388},
  number={2},
  pages={1437--1494},
  year={2024},
  publisher={Springer}
}

@MastersThesis{Ord20,
      title = {Word Measures on the Signed Symmetric Group},
      author = {Gal Ordo},
      year = {2020},
      school = {Tel Aviv University}
}

@article{Alon86,
title = "Eigenvalues and expanders",
author = "Noga Alon",
year = "1986",
volume = "6",
pages = "83--96",
journal = "Combinatorica",
publisher = "Janos Bolyai Mathematical Society",
number = "2"
}

@article{NILLI1991207,
title = {On the second eigenvalue of a graph},
journal = {Discrete Mathematics},
volume = {91},
number = {2},
pages = {207-210},
year = {1991},
author = {A. Nilli}
}

@article{Friedman04,
author = {Friedman, Joel},
year = {2004},
pages = {},
title = {A proof of Alon's second eigenvalue conjecture and related problems},
volume = {195},
journal = {Memoirs of the American Mathematical Society}
}

@article{louder2021uniform,
  title={Uniform negative immersions and the coherence of one-relator groups},
  author={Louder, Larsen and Wilton, Henry},
  journal={Inventiones Mathematicae},
  volume={236},
  number={2},
  pages={673--712},
  year={2024},
  publisher={Springer}
}

@article{Calegari_2009,
   title = {Stable Commutator Length is Rational in Free Groups},
   volume={22},
   number={4},
   journal={Journal of the American Mathematical Society},
   publisher={American Mathematical Society (AMS)},
   author={Calegari, Danny},
   year={2009},
   pages={941–961}
}

@article{ernstwest2021word,
  title={Word measures on {$\mathrm{GL}_N(q)$} and free group algebras},
  author={Ernst-West, Danielle and Puder, Doron and Seidel, Matan},
  journal={Algebra \& Number Theory},
  volume={18},
  number={11},
  pages={2047--2090},
  year={2024},
  publisher={Mathematical Sciences Publishers}
}

@article{Brodsky22,
  title={Word measures on unitary groups: Improved bounds for small representations},
  author={Brodsky, Yaron},
  journal={International Mathematics Research Notices},
  volume={2024},
  number={14},
  pages={10790--10821},
  year={2024},
  publisher={Oxford University Press}
}

@book{ceccherini-silberstein_scarabotti_tolli_2014, 
place={Cambridge}, 
series={London Mathematical Society Lecture Note Series}, 
title = {Representation Theory and Harmonic Analysis of Wreath Products of Finite Groups}, 
publisher={Cambridge University Press}, 
author={Ceccherini-Silberstein, Tullio and Scarabotti, Fabio and Tolli, Filippo}, 
year={2014}, 
collection={London Mathematical Society Lecture Note Series}}

@phdthesis{nicolaus2019a,
  year = {2019},
  edition = {},
  number = {},
  journal = {},
  pages = {},
  publisher = {University of Oxford},
  school = {University of Oxford},
  title = {Constructions in Stable Commutator Length and Bounded Cohomology},
  volume = {},
  author = {Heuer, Nicolaus},
  editor = {},
  series = {}
}

@book{Cox73, 
    author = "Coxeter, Harold S.M.",
    title = "Regular Polytopes",
    edition = "3",
    publisher = "Dover",
    location = "New York",
    year = "1973"
}

@article{Church_2015,
   title = {FI-Modules and Stability for Representations of Symmetric Groups},
   volume={164},
   number={9},
   pages = {1833 - 1910},
   journal={Duke Mathematical Journal},
   publisher={Duke University Press},
   author={Church, Thomas and Ellenberg, Jordan S. and Farb, Benson},
   year={2015},
}

@article{Sam_2019,
   title = {Representations of Categories of G-Maps},
   volume={2019},
   number={750},
   journal={Journal für die Reine und Angewandte Mathematik (Crelles Journal)},
   publisher={Walter de Gruyter GmbH},
   author={Sam, Steven V and Snowden, Andrew},
   year={2019},
   pages={197–226}
}

@article{gan2017representation,
  title={A representation stability theorem for VI-modules},
  author={Gan, Wee Liang and Watterlond, John},
  journal={Algebras and Representation Theory},
  volume={21},
  pages={47--60},
  year={2018},
  publisher={Springer}
}

@article{ryba2019stable,
  title={Stable Grothendieck rings of wreath product categories},
  author={Ryba, Christopher},
  journal={Journal of Algebraic Combinatorics},
  volume={49},
  pages={267--307},
  year={2019},
  publisher={Springer}
}

@article{gan2016coinduction,
  title={Coinduction functor in representation stability theory},
  author={Gan, Wee Liang and Li, Liping},
  journal={Journal of the London Mathematical Society},
  volume={92},
  number={3},
  pages={689--711},
  year={2015},
  publisher={Oxford University Press}
}

@article{casto2016fig,
  title={$\mathrm{FI}_G$-modules, orbit configuration spaces, and complex reflection groups},
  author={Casto, Kevin},
  journal={arXiv preprint arXiv:1608.06317},
  year={2016}
}

@article{Gadish_2017,
   title = {Categories of FI Type: A Unified Approach to Generalizing Representation Stability and Character Polynomials},
   volume={480},
   journal={Journal of Algebra},
   publisher={Elsevier BV},
   author={Gadish, Nir},
   year={2017},
   pages={450–486}
}

@article{sam_snowden_stability,
  title={Stability patterns in representation theory},
  author={Sam, Steven V and Snowden, Andrew},
  journal={Forum of Mathematics, Sigma},
  volume={3},
  pages={e11},
  year={2015}
}

@article{Putman_2017,
author = {Andrew Putman and Steven V Sam},
title = {{Representation stability and finite linear groups}},
volume = {166},
journal = {Duke Mathematical Journal},
number = {13},
publisher = {Duke University Press},
pages = {2521 -- 2598},
year = {2017},
}

@article{Stallings1983TopologyOF,
  title = {Topology of Finite Graphs},
  author={John R. Stallings},
  journal={Inventiones Mathematicae},
  year={1983},
  volume={71},
  pages={551-565}
}

@article{WILSON2014269,
title = {$\mathrm{FI}_W$-Modules and Stability Criteria for Representations of Classical Weyl Groups},
journal = {Journal of Algebra},
volume = {420},
pages = {269-332},
year = {2014},
author = {Jennifer C.H. Wilson}
}

@article{Takahasi1951NoteOC,
  title={Note on Chain Conditions in Free Groups},
  author={Mutuo Takahasi},
  journal={Osaka Mathematical Journal},
  year={1951},
  volume={3},
  pages={221-225}
}

@article{KAPOVICH2002608,
title = {Stallings Foldings and Subgroups of Free Groups},
journal = {Journal of Algebra},
volume = {248},
number = {2},
pages = {608-668},
year = {2002},
author = {Ilya Kapovich and Alexei Myasnikov}
}

@InProceedings{MVW07,
author="Miasnikov, Alexei and Ventura, Enric and Weil, Pascal",
editor="Arzhantseva, Goulnara N. and Burillo, Jos{\'e} and Bartholdi, Laurent and Ventura, Enric",
title="Algebraic Extensions in Free Groups",
booktitle="Geometric Group Theory",
year="2007",
publisher="Birkh{\"a}user Basel",
address="Basel",
pages="225--253"
}

@article{puder2023stable,
  title={Stable invariants and their role in word measures on groups},
  author={Puder, Doron and Shomroni, Yotam},
  journal={arXiv preprint arXiv:2311.17733},
  year={2023}
}

@article{puder2025stable,
  title={Stable Invariants of Words from Random Matrices II: Formulas and Extensions},
  author={Puder, Doron and Shomroni, Yotam},
  journal={arXiv preprint arXiv:2509.17271},
  year={2025}
}
\noindent
Yotam Shomroni,\\
School of Mathematical Sciences,\\
Tel Aviv University,\\
Tel Aviv, 6997801, Israel\\
yotam.shomroni@gmail.com

\end{document}